\documentclass[english,10pt]{amsart}
\usepackage[utf8]{inputenc}
\usepackage{amsfonts}
\usepackage{amssymb, amsmath}
\usepackage{amsthm}
\usepackage{mathtools}
\usepackage{dsfont}
\usepackage{bm}
\usepackage{verbatim}
\usepackage{xcolor}
\usepackage{url}
\usepackage{subcaption}
\usepackage{float}
\usepackage{caption}
\usepackage[pdftex]{hyperref}
\hypersetup{plainpages=false,unicode=true,pdffitwindow=true}

\theoremstyle{plain}
\newtheorem{theorem}{Theorem}[section]
\newtheorem{lemma}[theorem]{Lemma}
\newtheorem{proposition}[theorem]{Proposition}
\newtheorem{corollary}[theorem]{Corollary}

\theoremstyle{definition}
\newtheorem{definition}[theorem]{Definition}

\newtheorem{example}[theorem]{Example}

\theoremstyle{remark}
\newtheorem{remark}[theorem]{Remark}

\newcommand{\N}{\mathbb{N}}
\newcommand{\R}{\mathbb{R}}
\newcommand{\C}{\mathbb{C}}

\newcommand{\E}{\mathbb{E}}

\newcommand{\lip}{\textup{Lip}}
\newcommand{\eps}{\varepsilon}

\newcommand{\bx}{\bm{x}}
\newcommand{\by}{\bm{y}}
\newcommand{\bz}{\bm{z}}
\newcommand{\bg}{\bm{g}}
\newcommand{\abs}[1]{\left\lvert #1 \right\rvert}
\newcommand{\id}{\textup{Id}}
\newcommand{\loc}{{\textup{loc}}}
\newcommand{\var}{\textup{-var}}

\newcommand{\dd}{\mathrm{d}}
\newcommand{\sdd}{\,\dd}
\DeclareMathOperator*{\std}{\textup{std}}
\DeclareMathOperator*{\med}{\textup{med}}

\newcommand{\vertiii}[1]{{\left\vert\kern-0.25ex\left\vert\kern-0.25ex\left\vert #1 
    \right\vert\kern-0.25ex\right\vert\kern-0.25ex\right\vert}}

\usepackage{todonotes}

\numberwithin{equation}{section}

\usepackage[backend=biber,style=alphabetic]{biblatex}

\bibliography{literature.bib}

\title[Adaptive algorithm]{An adaptive algorithm for rough differential equations}

\author{Christian Bayer}
\address{Weierstrass Institute, Mohrenstraße 39, 10117 Berlin, Germany}
\email{christian.bayer@wias-berlin.de}

\author{Simon Breneis}
\address{Weierstrass Institute, Mohrenstraße 39, 10117 Berlin, Germany}
\email{simon.breneis@wias-berlin.de}

\author{Terry Lyons}
\address{Mathematical Institute, University of Oxford, Andrew Wiles Building, Woodstock Road, Oxford OX2 6GG, United Kingdom}
\email{terry.lyons@maths.ox.ac.uk}

\date{\today}
\subjclass[2020]{65L50, 60L90, 65C30}
\keywords{Rough differential equations, Adaptive algorithms, Log-ODE method}

\thanks{C.B. and S.B. gratefully acknowledge the support of the DFG through the IRTG 2544. T. L. was funded in part by the EPSRC grant EP/S026347/1, by The Alan Turing Institute under the EPSRC grant EP/N510129/1, the Data Centric Engineering Programme (under the Lloyd’s Register Foundation grant G0095), the Defence and Security Programme (funded by the UK Government), the Office for National Statistics, The Alan Turing Institute (strategic partnership) and by the Hong Kong Innovation and Technology Commission (InnoHK Project CIMDA)}

\begin{document}
\maketitle

\begin{abstract}


We present an adaptive algorithm for effectively solving rough differential equations (RDEs) using the log-ODE method. The algorithm is based on an error representation formula that accurately describes the contribution of local errors to the global error. By incorporating a cost model, our algorithm efficiently determines whether to refine the time grid or increase the order of the log-ODE method. In addition, we provide several examples that demonstrate the effectiveness of our adaptive algorithm in solving RDEs.
\end{abstract}

\section{Introduction}

Rough path analysis (see \cite{friz2020course, friz2010multidimensional, lyons2007differential}) has found many important applications such as stochastic analysis, filtering, and machine learning. At its core, it provides a full solution theory for controlled differential equations of the form
\begin{equation*}
  \dd y_t = f(y_t) \sdd x_t, \quad y_0 = \xi,
\end{equation*}
when $x: [0,T] \to \R^d$ is a path of finite $p$-variation, $p \ge 1$, such as a typical trajectory of a Brownian motion (for $p > 2$).
More precisely, we can solve such equations in a pathwise way, provided we have access to a \emph{rough path} lift $\bx$ of $x$ encoding additional information about $x$ (such as iterated integrals of $x$ against itself).\footnote{Strictly speaking, $\bx$ does not need to be the lift of a path $x$, but, for instance, pure area paths (i.e., limits of lifts of highly oscillatory paths) are also covered.}
Given such a driving rough path $\bx$, we obtain continuity of the solution map $\bx \mapsto y$ of the \emph{rough differential equation} (RDE)
\begin{equation}\label{eqn:RDEEquation}
  \dd y_t = f(y_t) \sdd\bx_t, \quad y_0 = \xi,
\end{equation}
which can also be extended to a rough path lift of $y$, i.e., $\bx \mapsto \by$. (See Appendix \ref{sec:Notation} for precise definitions.)

When it comes to numerical solutions of RDEs, essentially three different methods have been proposed. The first approach consists of time-discretization of the rough differential equation based on local higher order Taylor expansions in terms of the extended ``increments'' of $\bx$ following \cite{davie2007differential}. For a truncation level $N = 2$ (as required for $2 \le p < 3$), this essentially corresponds to a Milstein approximation of a stochastic differential equation in Stratonovich sense. The method is extensively studied in \cite{friz2010multidimensional} under the name of ``Euler scheme'', which we adopt in this paper.

The second approach leverages stability of the solution map of an RDE: given a rough driving path $\bx$ and a sequence of smooth (i.e., bounded variation) paths $x^n$, $n \ge 1$, whose lifts $\bx^n$ converge to $\bx$ in the appropriate rough path metric (i.e., $p$-variation metric), approximate the solution $y$ of the RDE by the solutions
\begin{equation*}
  \dd y^n_t = f(y^n_t) \sdd x^n_t, \quad y^n_0 = \xi,
\end{equation*}
of the controlled differential equations driven by $x^n$.
Those equations fall under the classical theory, and can then be solved by a large number of standard ODE-solvers (Euler schemes, Runge--Kutta schemes, \ldots).
In the context of SDEs, Wong--Zakai approximation is an example of the second approach.

In this paper, we are following the third approach known as the \emph{log--ODE method}, which was first considered in \cite{lyons2014rough}, and further studied in \cite{boutaib2013dimension}. Before explaining the log-ODE method, let us first recall the definition of the \emph{signature} of a (rough) path $\bx$. If $\bx$ is actually a smooth path, the signature is the collection of iterated integrals, i.e.
$$S_N(\bx)_{s,t} = \left(1, \int_s^t \dd x_{\tau_1},\dots,\int_{s<\tau_1<\dots<\tau_N<t} \dd x_{\tau_1}\otimes\dots\otimes \dd x_{\tau_N}\right).$$ For a genuine $p$-rough path these iterated integrals may not be well-defined and have to be specified explicitly, c.f. Appendix \ref{sec:Notation}. Iterated integrals of degree larger than $\lfloor p \rfloor$ are then defined as rough integrals, and we obtain a definition of the truncated signature $S_N(\bx)_{s,t}$ at any level $N$.

Given a partition $0 = t_0 < \dots < t_n = T$, the log-ODE method consists in replacing the rough path $\bx$ in \eqref{eqn:RDEEquation} on every interval $[t_k,t_{k+1}]$ by a simpler rough path $\bx^k$ satisfying 
\begin{equation}\label{eqn:SignaturesMustAgree}
S_N(\bx^k)_{t_k,t_{k+1}} = S_N(\bx)_{t_k,t_{k+1}}.
\end{equation}
Then, one solves the RDE
\begin{equation}\label{eqn:NewLogODERDE2}
\dd\overline{y}_t = f(\overline{y}_t) \sdd\bx^k_t,
\end{equation}
on the interval $[t_k,t_{k+1}]$, with initial value $\overline{y}_{t_k}$ already computed before reaching the interval. In principle, one can find smooth paths $x^k$ satisfying \eqref{eqn:SignaturesMustAgree}, so that \eqref{eqn:NewLogODERDE2} effectively become ODEs. However, in practice it is easier to choose a specific rough path $\bx^k$ defined in Section \ref{sec:LogODE} such that the RDE \eqref{eqn:NewLogODERDE2} is equivalent to the ODE
\begin{align*}
\dd\overline{y}_t = \sum_{k=1}^N f^{\circ k}(\overline{y}_t) \pi_k(\log_N(S_N(\bx)_{t_k,t_{k+1}}))\sdd t,
\end{align*}
where $\log_N$ is the tensor algebra logarithm, defined in Appendix \ref{sec:Euler}. Thus, one reduces the problem of solving an RDE to that of solving a sequence of ODEs. These ODEs can then be solved my standard ODE solvers, like the Runge-Kutta method. A more precise formulation of the log-ODE method can be found in Section \ref{sec:LogODE}.

At first glance, the idea that one might solve a controlled differential equation using higher order information about the controlling signal looks optimistic. Certainly, if the original data was only monitored as a time series, then it is clear that the higher order description can only be constructed from that sampled data and so what exactly is the purpose of first making and then consuming this high order description, when compared with using the original data directly. This question was observed and answered in the original papers of Gaines and Lyons \cite{gaines1997variable}, where the benefits become immediate if the vector fields are in any sense expensive to compute.

The log-ODE method is similar to the Euler scheme described above in that one partitions the interval $[0,T]$, and then uses the signature up to some level $N$ on each of the partition intervals. And indeed, in \cite{boedihardjo2015uniform, boutaib2013dimension} it was shown that both methods achieve the same rate of convergence for fixed $N$ as the partition gets finer. The difference is that in the Euler scheme, we take a single Euler step to compute $y_{t_{k+1}}$ from $y_{t_k}$, whereas in the log-ODE method we solve a differential equation along the approximating path. The log-ODE method therefore has the advantage of respecting the geometry of the problem -- provided that a geometric ODE solver is used. For example, if the solution of the original RDE \eqref{eqn:RDEEquation} lies in a certain manifold, the approximation of the log-ODE method will also lie in that manifold, whereas the same may not be true for the approximation of the Euler scheme.

The log-ODE method has seen various applications, for example in solving neural controlled differential equations in \cite{kidger2020neural, morrill2020neural, morrill2021neural}, and learning stochastic differential equations in \cite{kidger2021neural, liao2019learning}.
Indeed, in the context of neural controlled differential equations, it is worth keeping in mind that one solves many thousands of controlled differential equations in the training phase driven by the same input paths (the data). So any efforts to compute the higher order descriptions are amortised over (potentially) millions of uses of the noise with different differential equations.
In addition, the log-ODE method is intuitively related to \emph{cubature on Wiener space}, see \cite{lyons2004cubature}.
Indeed, in the latter scheme, we consider paths $\omega_i$ with weights $\lambda_i$, $i=1, \ldots, M$, such that the \emph{expected signature}
\[
\sum_{i=1}^M \lambda_i S_N(\omega_i)_{t_k,t_{k+1}} = \E\left[ S_N(\mathbf{W})_{t_k,t_{k+1}} \right],
\]
where $\mathbf{W}$ denotes an appropriate rough path lift of Brownian motion.
Hence, cubature on Wiener space can be seen as a weak log-ODE method.

The goal of this paper is to provide and analyze an efficient, adaptive algorithm based on the log-ODE method.
For a fixed method of choosing a smooth surrogate path $\bx^k$ given an arbitrary sub-interval $[t_k, t_{k+1}]$ and degree $N = N_k$, the log-ODE method is determined by two numerical parameter sets:
\begin{itemize}
\item the grid $0 = t_0 < t_1 < \cdots < t_n =T$;
\item the degrees $N_k$ for the log-ODE method applied to the sub-interval $[t_{k}, t_{k+1}]$, $k=0, \ldots, n-1$.
\end{itemize}
In order to efficiently use the available computational resources, the choices of grid and degrees obviously need to be done in a problem-depending way.
For instance, if the driving path $\bx$ exhibits a local change of roughness at some point in time, we expect to adjust both the time-step and the degree.
Likewise, if the driving vector-field $f$ is ``singular'' in a certain part of the solution space (in the sense that the derivatives of $f$ become markedly larger in  norm compared to everywhere else), then we expect to choose finer time-steps and (possibly) larger degree whenever the solution $\overline{y}$ takes values in this region.
Hence, an efficient choice of grid and degree will, in many cases, not consist of uniform time-steps and constant degree $N_k \equiv N$.

This leaves us with the important practical question of how to choose the time-step $t_{k+1} - t_k$ and the corresponding degree $N_k$, $k=0, \ldots, N-1$.
In principle, the answer can be provided by a-priori error analysis.
For instance, if the local error is understood to be proportional to $\|\bx\|_{p\var;[t_k,t_{k+1}]}$ (or some appropriate power thereof, depending on $N_k$), then it would make sense to choose the time-step $t_{k+1} - t_k$ inversely proportionally.
However, in numerical analysis adaptive methods based on \emph{a-posteriori error estimates} are usually preferred.
Indeed, a-priori error estimates are usually very pessimistic (possibly unavoidably so for general purpose solvers), and often hard to compute.
A-posteriori error estimates, on the other hand, can be fine-tuned to properties of the specific problem at hand.
For this reason, adaptive numerical methods are state of the art for solving ODEs as well as PDEs.
We refer to \cite{eriksson1995introduction} for an introduction.

In this paper, we follow the approach of \cite{moon2003variational} for adaptive approximation of ODEs and extend it to rough differential equations.
We also refer to the related works \cite{szepessy2001adaptive,hoel2016construction} for adaptive weak and strong approximation of solutions to stochastic differential equations, respectively.
The approach is based on an intuitive error representation of the global approximation error as the sum of weighted local errors $\sum_{k=0}^{n-1} w_k e_k$ up to higher order terms, where $e_k$ denotes the local error on $[t_k, t_{k+1}]$ and $w_k$ denotes a weight.
The weights are given as solution to a backward ``dual'' problem, and are, thus, computable (in an approximate sense).
Hence, the weights can be thought of as multipliers due to error propagation.
Since the representation decomposes the total error into contribution attributable to specific sub-intervals (i.e., $w_k e_k$), optimal time steps can, in principle, be found by imposing that the error contributions for all intervals are equal, i.e., $w_k e_k \equiv \frac{\mathrm{TOL}}{n}$ for some error tolerance $\mathrm{TOL}$.
It should be noted that this approach takes into account the actual quantity of interest: I.e., if we want to compute a function $g(y_T)$, the ``error'' considered here is, indeed, $| g(y_T) - g(\overline{y_T})|$.

The dual problem characterizing the weights $w_k$ is itself a (backward) (ordinary, rough, or stochastic) differential equation, depending on the context.
Hence, a fully implementable scheme will require solving the dual problem numerically using an appropriate time-stepping algorithm.
Therefore, the adaptive method will be based on grid-refinement rather than direct adaptive choices of the increment $t_{k+1} - t_k$ based on an already constructed grid $t_0, \ldots, t_k$.
More precisely, we start with an initial grid $0 = t^{(0)}_0 < t^{(0)}_1 < \cdots < t^{(0)}_{n_0} = T$.
On this grid, we solve both the (forward) problem for $y$ and the dual problem, given us local errors and weights $e^{(0)}_k$, $w^{(0)}_k$, $k=0, \ldots, n_0$.
If the total error estimate $\sum_{k=0}^{n_0-1} w^{(0)}_k e^{(0)}_k \le \mathrm{TOL}$, we are satisfied with the result.
Otherwise, we refine increments $[t_k, t_{k+1}]$ such that the corresponding (estimated) error contribution $w^{(0)}_k e^{(0)}_k$ is large, leading to a new grid $t_0^{(1)} < t^{(1)}_1 < \cdots < t^{(1)}_{n_1}$ with $n_1 > n_0$, and iterate the procedure.
Usually (for efficiency reasons), the refinement is done in such a way that the refined grid is strictly finer than the original grid, i.e., $\{t_0^{(i)}, \ldots, t_{n_i}^{(i)}\} \subset \{t_0^{(i+1)}, \ldots, t_{n_{i+1}}^{(i+1)}\}$.
Moreover, intervals $[t^{(i)}_k, t^{(i)}_{k+1}]$ are usually divided into a fixed number of uniform sub-intervals.
We will choose binary intervals, i.e., intervals are exactly subdivided into two new intervals of equal size, if they are refined.

In the context of the log-ODE method, we want to adaptively choose the degree $N_k$ as well as the time-step $t_{k+1} - t_k$.
This complicates the algorithm, since ``refinements'' can be done by either sub-dividing intervals or by locally increasing the degree of the method -- or both.
Additionally, the criterion for refinement becomes more complicated, since it is no longer enough to merely track the error contributions of individual sub-intervals $[t_k, t_{k+1}]$.
Instead, these error contributions need to be balanced against a corresponding local contribution to the total computational cost.
Indeed, the computational cost of local log-ODE steps from $t_k$ to $t_{k+1}$ depends on the degree $N_k$.

We now give a short outline of the paper. The notation and the basic definitions of rough path theory are recalled in Appendix \ref{sec:Notation}. In Section \ref{sec:MinorExtensions}, we prove some technical lemmas which mostly mildly extend already well-known theorems in the rough path literature to make them more suitable for our setting.

In Section \ref{sec:MainResults} we derive and prove the error representation formula. In fact, we do not restrict ourselves to the Log-ODE method here, but rather prove the fomula for very general approximation schemes for RDEs, including the previously mentioned Euler method. We then give an algorithm for actually estimating the error. For this, we need to impose further, much stronger assumptions on our schemes for approximating RDEs, see Section \ref{sec:AdmissibleAlgorithms}. In particular, the Log-ODE method satisfies these assumptions, while the Euler method does not. Finally, in Section \ref{sec:EstimateWeights} we give the algorithm for computing the error representation formula.

In Section \ref{sec:LogODE}, we apply our results to the Log-ODE method and prove that the Log-ODE method satisfies all the requirements of Section \ref{sec:AdmissibleAlgorithms}. 

In Section \ref{sec:Numerics} we give some numerical examples to illustrate the accuracy of our error representation formula and the efficiency of the adaptive algorithm.
In these examples, the adaptive method is used together with a cost model for the Log-ODE method to determine which action, refining the partition or increasing the degree, is more efficient.
More details on the cost model, including numerical justification is provided in Appendix \ref{sec:CostModel}.

Appendix \ref{sec:Euler} contains some results on the Euler approximation needed in the proofs of the paper.

\section{Rough Path analysis for unbounded vector fields}\label{sec:MinorExtensions}

In this section we prove some technical lemmas, mostly extending well-known rough path theorems for bounded vector fields to unbounded vector fields satisfying the $p$ non-explosion condition. Many of the proofs are carried out through localization.

\subsection{The $p$ non-explosion condition}

Recall the definition of the vector field spaces $\mathcal{V}$ and $\mathcal{V}^\gamma$ for $\gamma > 0$ from Appendix \ref{sec:RDESolutions}. Below, we give the definition of the $p$ non-explosion condition.

\begin{definition}\label{def:PNonExplosion}\cite[cf. Definition 11.1]{friz2010multidimensional}
Let $f\in\mathcal{V}(\R^d,\R^e)$, and $p\ge 1$. We say that $f$ satisfies the $p$ non-explosion condition if for every $R>0$ there exists an $M(R)>0$, such that for every $\by_0\in G^{[p]}(\R^e)$ and every $\bx\in C^{p\var}([0,T],\R^d)$ with $\|\by_0\| + \|\bx\|_{p\textup{-var}} \le R$ we have that every solution $\by$ to the full rough differential equation
\begin{equation}\label{eqn:FullRDE}
\dd\by_t = f(y_t) \sdd\bx_t,\qquad \by_0 = \by_0
\end{equation}
satisfies $\|\by\|_\infty \le M(R).$

For all $\gamma > 0$, we denote the set of vector fields $f\in\mathcal{V}^\gamma(\R^d,\R^e)$ satisfying the $p$ non-explosion condition by $\mathcal{V}^{\gamma,p}(\R^d,\R^e).$ The set $\mathcal{V}^{\gamma+,p}(\R^d,\R^e)$ is defined similarly.
\end{definition}

The following proposition shows that we can restrict ourselves to first-level RDEs in the above definition for $f\in \mathcal{V}^{\gamma-1}(\R^d,\R^e)$ if $\gamma > p$. The proof is a simple localization argument. Since many similar localization arguments will follow, we will usually omit these proofs, and only show the general idea here once.

\begin{proposition}\label{prop:PNonExplosionFirstLevel}
Let $f\in\mathcal{V}^{\gamma-1}(\R^d,\R^e)$ with $\gamma > p\ge 1$. Assume that for every $R>0$ there exists an $M>0$, such that for every $y_0\in \R^e$ and every $\bx\in C^{p\var}([0,T],\R^d)$ with $\|y_0\| + \|\bx\|_{p\textup{-var}}\le R$ we have that every solution $y$ to the RDE
\begin{equation}\label{eqn:FirstLevelRDE}
dy_t = f(y_t) d\bx_t,\qquad y_0 = y_0
\end{equation}
satisfies $\|y\|_\infty \le M$. Then, $f$ satisfies the $p$ non-explosion condition.
\end{proposition}

\begin{proof}
Let $\by_0\in G^{[p]}(\R^e)$ and let $\bx\in C^{p\var}([0,T],\R^d)$ with $\|\by_0\| + \|\bx\|_{p\textup{-var}}\le R.$ Let $\by$ be a solution to the full RDE \eqref{eqn:FullRDE}. Then, $y = \pi_1(\by)$ is a solution to \eqref{eqn:FirstLevelRDE} with $y_0 = \pi_1(\by_0)$. By our assumptions, we have $\|y\|_\infty \le M(R)$. Since the vector field $f$ in the full RDE \eqref{eqn:FullRDE} depends only on $y$, and not on the other (higher) levels in $\by$, we may restrict $f$ onto $B_{M(R)}$, and then extend this restriction to a bounded vector field $\widetilde{f}\in L(\R^d,\lip^{\gamma-1}(\R^e,\R^e))$ by Theorem \ref{thm:WhitneysTheoremStein}. The norm of $\widetilde{f}$ then depends only on $f$ and $M(R)$. In particular, if $\|\by_0\| + \|\bx\|_{p\textup{-var}} \le R$, we may replace $f$ by $\widetilde{f}$, and \cite[Theorem 10.36]{friz2010multidimensional} gives us a bound $$\|\by\|_{p\textup{-var}} \le C(\|\widetilde{f}\|_{\lip^{\gamma-1}}, R) < \infty.$$ We conclude by noting that $\|\by\|_\infty$ can be bounded by $\|\by_0\| \le R$ and $\|\by\|_{p\textup{-var}}.$
\end{proof}

We will frequently use the following lemma without further reference. It follows immediately from Lemma \ref{lem:RoughPathDifferentPConsistency} and Proposition \ref{prop:PNonExplosionFirstLevel}.

\begin{lemma}\label{lem:PNonExplosionImpliesQNonExplosion}
For all $\gamma > p\ge q\ge 1$, $\mathcal{V}^{\gamma-1,p}(\R^d,\R^e)\subseteq\mathcal{V}^{\gamma-1,q}(\R^d,\R^e)$.
\end{lemma}

\subsection{Existence and uniqueness of solutions}

\begin{theorem}\label{thm:RDEExistence}
Let $\bx\in C^{p\var}([0,T],\R^d)$, let $f\in\mathcal{V}^{\gamma-1,p}(\R^d,\R^e)$ where $\gamma > p\ge 1$, and let $\by_0\in G^{[p]}(\R^e)$. Then, there exists a solution to the full RDE \eqref{eqn:FullRDE}. Moreover, if $\by$ is such a solution, then $$\|\by\|_{p\var;[s,t]} \le C\|\bx\|_{p\var;[s,t]}$$ for some $C = C(p,\gamma,f,\|\by_0\| + \|\bx\|_{p\var;[0,T]})$ increasing in the last component, and for all $0\le s\le t\le T$.

Moreover, let $(x^n)$ be a sequence of finite variation paths approximating $\bx$ in the sense of \eqref{eqn:ConvergenceToWeakGRP}. If $(\by^n)$ are solutions to the corresponding RDEs driven by $(x^n)$, then there exists a subsequence of $(\by^n)$ uniformly converging to a solution $\by$ of the RDE \eqref{eqn:FullRDE}.
\end{theorem}

\begin{proof}
By a localization argument, we can reduce this theorem to \cite[Theorem 10.36]{friz2010multidimensional}.
\end{proof}

\begin{theorem}\label{thm:RDEUniqueness}
Let $\omega$ be a control, let $\bx^1,\bx^2\in C^{p\var}([0,T],\R^d)$ be controlled by $\omega$, let $f\in \mathcal{V}^{\gamma,p}(\R^d,\R^e)$ where $\gamma > p\ge 1$, and let $\by_0^1,\by_0^2\in G^{[p]}(\R^e).$ Let $\by^i$ be solutions to the full RDEs $$\dd\by^i_t = f(y^i_t) \sdd\bx^i_t,\qquad \by^i_0 = \by^i_0,\qquad i=1,2.$$ Then, the solutions $\by^i$ exist and are unique, and we have $$\rho_{p,\omega}(\by^1,\by^2) \le C\left(\|y_0^1-y_0^2\| + \rho_{p,\omega}(\bx^1,\bx^2)\right),$$ where $C = C(\gamma,p,f,\|\by_0^1\| + \|\by_0^2\| + \omega(0,T))$ is increasing in the last component.

Similarly, if $(\bx^n)$ is a sequence in $C^{p\var}([0,T],\R^d)$ converging to $\bx$ in $p$-variation, then the corresponding solutions $\by^n$ for the RDEs driven by $\bx^n$ converge in $p$-variation to the solution $\by$ driven by $\bx$.
\end{theorem}

\begin{remark}
If one cares only about uniqueness, the $p$ non-explosion condition can of course be dropped. Uniqueness is a local issue.
\end{remark}

\begin{proof}
The existence of a solution was shown in Theorem \ref{thm:RDEExistence}. The statement of the theorem then follows from a simple localization argument, together with \cite[Theorem 10.38]{friz2010multidimensional} and \cite[Corollary 10.40]{friz2010multidimensional}.
\end{proof}

We can slightly strengthen the previous theorem for finite variation paths.

\begin{lemma}\label{lem:RDESolutionIsODESolution1RP}
Let $\omega$ be a control, let $\bx^1,\bx^2\in C^{1\var}([0,T],\R^d)$ be controlled by $\omega$, let $f\in \mathcal{V}^{1,1}(\R^d,\R^e)$, and let $\by^1_0,\by^2_0\in G^1(\R^e)$. Let $\by^i$ be solutions to the full RDEs 
\begin{equation}\label{eqn:RRRDE}
\dd\by^i_t = f(y^i_t)\sdd\bx^i_t,\qquad \by^i_0 = \by^i_0,\qquad i=1,2.
\end{equation}
Then, the solutions $\by^i$ exist and are unique, and we have
\begin{equation}\label{eqn:LipschitzSupremum}
\|\by^1-\by^2\|_\infty \le C\left(\|y^1_0-y^2_0\| + \|\bx^1-\bx^2\|_\infty\right),
\end{equation}
\begin{equation}\label{eqn:Lipschitz1Omega}
\rho_{1,\omega}(\by^1,\by^2) \le C\left(\|y_0^1-y_0^2\| + \rho_{1,\omega}(\bx^1,\bx^2)\right),
\end{equation}
and
\begin{equation}\label{eqn:Lipschitz1Var}
\rho_{1\var}(\by^1,\by^2) \le C\left(\|y^1_0-y^2_0\| + \rho_{1\var}(\bx^1,\bx^2)\right),
\end{equation}
where $C = C(f,\|\by_0^1\| + \|\by_0^2\| + \omega(0,T))$ is increasing in the last component. Moreover, $y^i = \pi_1(\by^i)$ are the unique solutions to the ODEs $$\dd y^i_t = f(y^i_t) \sdd x^i_t,\qquad y^i_0 = \pi_1(\by^i_0),\qquad i=1,2.$$
\end{lemma}

\begin{proof}
Let us first fix $\bx = \bx^1$, and discuss existence and uniqueness. Of course, the existence of a solution $\by$ follows immediately from Theorem \ref{thm:RDEExistence}. Moreover, by the $1$ non-explosion condition (or also by Theorem \ref{thm:RDEExistence}), any such solution can be uniformly bounded by a constant $C$ (depending on the various parameters). Using Theorem \ref{thm:WhitneysTheoremStein}, we can restrict $f$ to $B_C$ and then extend it again to get a vector field $\widetilde{f}\in L(\R^d,\lip^1(\R^e,\R^e))$ admitting the same solutions as $f$.

Now, consider the ODE $$\dd y_t = \widetilde{f}(y_t) \sdd x_t,\qquad y_0 = \pi_1(\by_0).$$ Since $\widetilde{f}$ is Lipschitz, we know by \cite[Theorem 3.8]{friz2010multidimensional} that this ODE admits a unique solution $y$. By Lemma \ref{lem:ODESolutionsAreRDESolutions}, $\widetilde{\by} \coloneqq (1, y)$ is a solution to $$\dd\by_t = \widetilde{f}(y_t) \sdd\bx_t,\qquad \by_0 = \by_0,$$ and by our choice of $\widetilde{f}$, it is also a solution to \eqref{eqn:RRRDE}. We can then use this to conclude that $y$ was already the unique solution to
\begin{equation}\label{eqn:ThisODE}
\dd y_t = f(y_t) \sdd x_t,\qquad y_0 = \pi_1(\by_0).
\end{equation}
We have thus shown that the ODE \eqref{eqn:ThisODE} admits a unique solution for all finite variation paths $x$. 

Next, we show that the solution to the RDE \eqref{eqn:RRRDE} is also unique. To that end, let $(x^n)$ and $(\widetilde{x}^n)$ be sequences of finite variation paths satisfying \eqref{eqn:ConvergenceToWeakGRP}. Let $(y^n)$ and $(\widetilde{y}^n)$ be the corresponding (unique) ODE solutions. By \cite[Theorem 3.15]{friz2010multidimensional}, both $(y^n)$ and $(\widetilde{y}^n)$ are Cauchy sequences in the uniform norm, and hence, $y^n\to y$ and $\widetilde{y}^n\to\widetilde{y}$. Then, by definition, $\by\coloneqq (1,y)$ and $\widetilde{\by}\coloneqq (1,\widetilde{y})$ are solutions to the RDE \eqref{eqn:RRRDE}. However, also by \cite[Theorem 3.15]{friz2010multidimensional}, both $y$ and $\widetilde{y}$ are solutions to the ODE \eqref{eqn:ThisODE}. Since \eqref{eqn:ThisODE} admits unique solutions, $\by = \widetilde{\by}$, implying that RDE \eqref{eqn:RRRDE} admits unique solutions.

Now that we have established that both the RDE and the ODE have unique solutions, and these solutions coincide, equations \eqref{eqn:LipschitzSupremum} and \eqref{eqn:Lipschitz1Var} follow immediately from \cite[Theorem 3.15]{friz2010multidimensional} and \cite[Theorem 3.18]{friz2010multidimensional}, respectively.

It remains to prove equation \eqref{eqn:Lipschitz1Omega}. Let $0\le s\le t\le T$, then, by \cite[Theorem 3.18]{friz2010multidimensional},
\begin{align*}
\|(y^1_t-y^2_t) - (y^1_s-y^2_s)\| &\le \|y^1-y^2\|_{1\var;[s,t]}\\
&\le C\left(\|y^1_s - y^2_s\|\omega(s,t) + \|x^1-x^2\|_{1\var;[s,t]}\right)\\
&\le C\left(\|y^1_s - y^2_s\| + \rho_{1,\omega;[s,t]}(\bx^1,\bx^2)\right)\omega(s,t),
\end{align*}
and we conclude the proof of the lemma.
\end{proof}

In the following lemmas we prove that ODEs really are special cases of RDEs, given that the vector field is sufficiently regular.

\begin{lemma}\label{lem:RDESolutionIsODESolution}
Let $x\in C([0,T],\R^d)$ be of finite variation, and let $\bx\coloneqq S_{[p]}(x)_{0, .}\in C^{p\var,0}([0,T],\R^d)$ be the associated geometric $p$-rough path. Let $f\in \mathcal{V}^{\gamma,p}(\R^d,\R^e)$, where $\gamma > p \ge 1$, and let $y_0\in \R^e$. Let $y$ be the unique solution to the RDE \eqref{eqn:FirstLevelRDE}. Then, $y$ is the unique solution to the ODE $$\dd y_t = f(y_t) \sdd x_t,\qquad y_0 = y_0.$$
\end{lemma}

\begin{proof}
We remark that by Theorem \ref{thm:RDEUniqueness}, the solution $y$ to the RDE indeed exists and is unique.

By Lemma \ref{lem:PNonExplosionImpliesQNonExplosion}, $f\in\mathcal{V}^{\gamma,1}(\R^d,\R^e)$. Let $\widetilde{\bx} \coloneqq S_1(x)$ be $x$ interpreted as a geometric $1$-rough path. By Lemma \ref{lem:RDESolutionIsODESolution1RP}, the unique solution $\widetilde{y}$ to the RDE driven by $\widetilde{\bx}$ is in fact the unique solution to the ODE. By Lemma \ref{lem:RoughPathDifferentPConsistency}, $\widetilde{y}$ is also a solution to the RDE driven by $\bx$. Since $y$ was the unique RDE solution driven by $\bx$, we have $y = \widetilde{y}$, and we conclude.
\end{proof}

\subsection{Extensions of vector fields}

We now show that certain extensions of vector fields satisfying the $p$ non-explosion condition again satisfy the $p$ non-explosion condition. We remark that not all of these statements would be correct if one considered invariance of these extensions in the class of bounded vector fields instead.

\begin{lemma}\label{lem:PNonExplosionAdjoin}
Let $f\in\mathcal{V}^{\gamma-1,p}(\R^d,\R^e)$ with $\gamma > p\ge 1$, and define the vector field $\widetilde{f}\in\mathcal{V}^{\gamma-1}(\R^d,\R^{d+e})$ by $$\widetilde{f}(x,y) = 
\begin{pmatrix}
\id\\
f(y)
\end{pmatrix}
$$ Then, $\widetilde{f}$ also satisfies the $p$ non-explosion condition.\footnote{To clarify the potentially misleading notation: Here, $x\in\R^d$, $y\in\R^e$, and, if $z\in\R^d$ is another element, then $\widetilde{f}(x,y)(z) = \begin{pmatrix}
\id\\
f(y)
\end{pmatrix}z
= \begin{pmatrix}
z\\
f(y)z
\end{pmatrix}\in \R^{d+e}.$ Hence, $\widetilde{f}$ is a function from $\R^{d+e}$, acting linearly on $\R^d$, taking values in $\R^{d+e}$.}
\end{lemma}

\begin{proof}
Let $z_0 = (x_0,y_0)\in\R^{d+e}$, and let $\bx\in C^{p\var}([0,T],\R^d)$ with $\|z_0\| + \|\bx\|_{p\textup{-var}} \le R$. Let $z$ be a solution to the first-level RDE $$\dd z_t = \widetilde{f}(z_t) \sdd\bx_t,\qquad z_0 = z_0.$$ By the definition of $\widetilde{f}$, it is clear that $z_t = (x_t,y_t),$ where $x_t = \pi_1(\bx_t)$, and where $y_t$ is a solution to $$\dd y_t = f(y_t) \sdd\bx_t,\qquad y_0 = y_0.$$ Hence, $$\|z\|_\infty \le C(R, M(R)),$$ where the $M(R)$ comes from the $p$ non-explosion condition of $f$. We conclude by Proposition \ref{prop:PNonExplosionFirstLevel}.
\end{proof}

\begin{lemma}\label{lem:PNonExplosionIntegral}
Let $f\in L(\R^d,\lip^{\gamma-1}_{\textup{loc}}(\R^d,\R^e)).$ Then, the vector field $\widetilde{f}\in\mathcal{V}^{\gamma-1}(\R^d,\R^{d+e})$ defined by $$\widetilde{f}(x, y) = 
\begin{pmatrix}
\id\\
f(x)
\end{pmatrix}
$$ satisfies the $p$ non-explosion condition for $p<\gamma$.
\end{lemma}

\begin{proof}
Let $z_0 = (x_0,y_0)\in \R^{d+e}$ and let $\bx\in C^{p\var}([0,T],\R^d)$ with $\|z_0\| + \|\bx\|_{p\textup{-var}}\le R$. Let $z$ be a solution to the RDE $$dz_t = \widetilde{f}(z_t)d\bx_t,\qquad z_0 = z_0.$$ Then, $z_t = (x_t, y_t)$. Now, $$\|x_t\| \le \|x_0\| + \|x_{0,t}\| \le \|z_0\| + \|\bx\|_{p\var} \le R.$$ Hence, we may restrict $f$ onto $B_R$, and then extend $f$ to a bounded vector field $f_1\in L(\R^d,\lip^{\gamma-1}(\R^d,\R^e))$ that agrees with $f$ on $B_R$. We then define the vector field $\widetilde{f}_1$ in a similar way as $\widetilde{f}$. Using $\widetilde{f}_1$ instead of $\widetilde{f}$ does not change the solutions if $\|z_0\| + \|\bx\|_{p\textup{-var}}\le R.$ Since $\widetilde{f}_1$ is bounded, we conclude by using the bound in \cite[Theorem 10.14]{friz2010multidimensional} and applying Proposition \ref{prop:PNonExplosionFirstLevel}.
\end{proof}

\begin{lemma}\label{lem:RoughIntegralWellDefined1}
Let $\bx\in C^{p\var}([0,T],\R^d)$, and let $f\in \mathcal{V}^{\gamma,p}(\R^d,\R^e)$ with $\gamma > p\ge 1$. Let $\by_0\in G^{[p]}(\R^e)$ and let $\by$ be the unique solution to the full RDE \eqref{eqn:FullRDE}. Then, the rough integral
\begin{equation}\label{eqn:RoughIntegral}
\bm{h} \coloneqq \int f'(y) \sdd\bx \in C^{p\var}([0,T],\R^{e\times e}) \cong C^{p\var}([0,T],\R^{e^2})
\end{equation}
is well-defined, i.e. it exists and is unique. Moreover, there exists a constant $C=C(p,\gamma,f,\|\by_0\| + \|\bx\|_{p\var,[0,T]})$ increasing in the last component, such that for all $0\le s\le t\le T$, we have $$\|\bm{h}\|_{p\var;[s,t]} \le C\|\bx\|_{p\var;[s,t]}.$$
\end{lemma}

\begin{proof}
First, we remark that by Theorem \ref{thm:RDEUniqueness}, there indeed exists a unique full solution $\by$ of the full RDE \eqref{eqn:FullRDE}. By Lemma \ref{lem:PNonExplosionAdjoin}, the vector field $$f_1(x, y) = 
\begin{pmatrix}
\id\\
f(y)
\end{pmatrix}
$$ is in $\mathcal{V}^{\gamma,p}(\R^d,\R^{d+e}),$ and by Theorem \ref{thm:RDEUniqueness}, there exists a unique solution $\bz$ of the full RDE $$\dd\bz_t = f_1(z_t)\sdd\bx_t,\qquad \bz_0 = (\bm{1}, \by_0)\in G^{[p]}(\R^{d+e}).$$ While this choice of $\bz_0$ is not unique, the increments of the rough path $\bz$ are independent of the particular choice. In fact, it is clear that the solution $\bz$ is the joint rough path $\bz = (\bx,\by)$.

Next, define the function $g\in L(\R^{d+e}, \lip^{\gamma-1}_{\textup{loc}}(\R^{d+e},\R^{e^2}))$, $$g(x, y) = 
\begin{pmatrix}
f'(y) & 0
\end{pmatrix}
.$$ The function $g$ has essentially the same effect as $f'$, except that it is defined on $z = (x,y)$.

By Lemma \ref{lem:PNonExplosionIntegral}, the vector field $$f_2(z, h) =
\begin{pmatrix}
\id\\
g(z)
\end{pmatrix}
$$ is in $\mathcal{V}^{\gamma-1,p}(\R^{d+e},\R^{d+e+e^2})$, and by Theorem \ref{thm:RDEExistence}, there exists a solution $\bm{v}$ of the RDE
\begin{equation}\label{eqn:SomeRDE}
d\bm{v}_t = f_2(v_t) d\bz_t,\qquad \bm{v}_0 = (\bz_0, \bm{1}).
\end{equation}

Such a solution must be of the form $\bm{v} = (\bz, \bm{h}),$ where $\bm{h}$ is a solution to the rough integral \eqref{eqn:RoughIntegral}. In particular, since $\|z\|_\infty$ is bounded, we may restrict $f'$ to a compact ball $B$ sufficiently large, and then extend this restriction again to a vector field $f_3\in L(\R^d,\lip^{\gamma-1}(\R^e,\R^e))$ that agrees with $f'$ on $B$. We can then replace $f'$ by $f_3$ in the definition of $f_2$  to get a new vector field $f_4$. The solutions of \eqref{eqn:SomeRDE} with $f_2$ replaced by $f_4$ must then remain unchanged. This again implies that instead of solving \eqref{eqn:RoughIntegral}, we may also compute the rough integral $$\int_s^u f_3(\by_t) \sdd\bx_t.$$ Since $f_3$ is $\lip^{\gamma-1}$, we conclude that the integral is well-defined by \cite[Theorem 10.47]{friz2010multidimensional}. Finally, the bound on $\|\bm{h}\|_{p\var}$ follows from $$\|\bm{h}\|_{p\var;[s,t]} \le \|\bm{v}\|_{p\var;[s,t]} \le C\|\bz\|_{p\var;[s,t]} \le C\|\bx\|_{p\var;[s,t]},$$ where the latter two inequalities follow from Theorem \ref{thm:RDEExistence}.
\end{proof}

\begin{lemma}\label{lem:RoughIntegralWellDefined2}
In the setting of Lemma \ref{lem:RoughIntegralWellDefined1}, if $(x^n)$ is a sequence of finite variation paths satisfying \eqref{eqn:ConvergenceToWeakGRP}, and $(y^n)$ are the corresponding solutions to the RDEs driven by $(x^n)$, and we define $$h^n\coloneqq \int f'(y^n) \sdd x^n,$$ then $(S_{[p]}(h^n))$ converges to $\bm{h}$ uniformly. If additionally, $\bx$ is a geometric $p$-rough path and $(S_{[p]}(x^n))$ converges to $\bx$ in $p$-variation, then $(S_{[p]}(h^n))$ converges to $\bm{h}$ in $p$-variation.
\end{lemma}

\begin{proof}
This statement follows easily from the proof of Lemma \ref{lem:RoughIntegralWellDefined1} and from Theorem \ref{thm:RDEUniqueness} and \cite[Corollary 10.48]{friz2010multidimensional}.
\end{proof}

\begin{lemma}\label{lem:RoughIntegralWellDefined3}
Let $\omega$ be a control, let $\bx^1,\bx^2\in C^{p\var}([0,T],\R^d)$ be controlled by $\omega$, and let $f\in \mathcal{V}^{\gamma+1,p}(\R^d,\R^e)$ with $\gamma > p \ge 1$. Let $\by_0^1,\by_0^2\in G^{[p]}(\R^e)$ and let $\by^1$ and $\by^2$ be the unique solutions to the full RDEs $$\dd\by^i_t = f(y^i_t) \sdd\bx^i_t,\qquad \by^i_0 = \by^i_0.$$ Define the rough integrals $$\bm{h}^i \coloneqq \int f'(y^i)\sdd\bx^i,\qquad i=1,2.$$ Then, there exists a constant $C = C(p,\gamma,f,\|\by_0^1\| + \|\by_0^2\| + \omega(0,T))$ that is increasing in the last component, such that $$\rho_{p,\omega}(\bm{h}^1,\bm{h}^2) \le C\left(\|y^1_0-y^2_0\| + \rho_{p,\omega}(\bx^1,\bx^2)\right).$$
\end{lemma}

\begin{proof}
This also follows immediately from Lemma \ref{lem:RoughIntegralWellDefined1} and Theorem \ref{thm:RDEUniqueness}.
\end{proof}

\begin{lemma}\label{lem:FullRDEAndFirstLevelRDEWithFullVF}
Let $f\in \mathcal{V}^{\gamma-1,p}(\R^d,\R^e)$, where $\gamma > p\ge 1$, let $\bx\in C^{p\var}([0,T],\R^d)$ and let $\by_0\in G^{[p]}(\R^e)$. Then, $\by$ is a solution to the full RDE $$\dd\by_t = f(y_t) \sdd\bx_t,\qquad \by_0 = \by_0,$$ if and only if it is a solution to the first-level RDE $$\dd\by_t = \bm{f}(\by_t) \sdd\bx_t,\qquad \by_0 = \by_0,$$ for $\bm{f}$ as defined in Definition \ref{def:FullVectorField}.
\end{lemma}

\begin{proof}
By a simple localization argument we may assume that $f$ is bounded.

Now, let $\by$ be a solution to the full RDE with respect to $f$. By \cite[Theorem 10.35]{friz2010multidimensional}, $\by$ is a solution to the first-level RDE with respect to $\bm{f}$.

To prove the converse statement, let $\by$ be a solution to the first-level RDE with respect to $\bm{f}$. By definition, there exists a sequence $(x^n)$ of finite variation paths with $$\|S_{[p]}(x^n) - \bx\|_\infty\to 0,\qquad \sup_n\|S_{[p]}(x^n)\|_{p\var} < \infty$$ and ODE solutions $(\by^n)$ to $$d\by^n_t = \bm{f}(\by^n_t)dx^n_t,\qquad \by^n_0 = \by_0$$ with $$\|\by^n-\by\|_\infty\to 0.$$ Denoting $y^n\coloneqq \pi_1(\by^n)$, the specific structure of $\bm{f}$ shows that $\by^n = \by_0\otimes S_{[p]}(y^n)$. But this immediately implies that $\by$ is a solution to the full RDE with respect to $f$.
\end{proof}

\begin{lemma}\label{lem:NonExplosionExtendsToFullRDEs}
Let $f\in\mathcal{V}^{\gamma-1,p}(\R^d,\R^e)$, where $\gamma > p \ge 1$. Then, the associated full vector field $\bm{f}$ also satisfies $\bm{f}\in \mathcal{V}^{\gamma-1,p}(\R^d,T_1^{[p]}(\R^e)).$
\end{lemma}

\begin{proof}
This is a consequence of Lemma \ref{lem:FullRDEAndFirstLevelRDEWithFullVF}.
\end{proof}

\section{Main results}\label{sec:MainResults}

\subsection{Error representation formula}

In this section, we follow the dual representation approach of \cite{moon2003variational}.

Let $\bx\in C^{p\var,0}([0,T],\R^d)$, and let $f\in\mathcal{V}^{\gamma,p}(\R^d,\R^e)$, where $\gamma > p\ge 1$. Let $\by_0\in G^{[p]}(\R^e)$ and consider the full rough differential equation \eqref{eqn:FullRDE}. By Theorem \ref{thm:RDEUniqueness}, there exists a unique full RDE solution $\by$. Let $\overline{\by}$ be a numerical approximation of $\by$ satisfying the same initial condition $\overline{\by}_0 = \by_0$ with time steps $$0 = t_0 < \dots < t_n = T.$$ The approximation $\overline{\by}$ only needs to be defined on the time grid $(t_k)_{k=0}^n$, and it needs to take values in $G^{[p]}(\R^e)$. For $t\in[0,T]$ and $\bz\in G^{[p]}(\R^e)$, define the flow $\by(s; t, \bz)$ as the solution to $$\dd\by(s;t,\bz) = f(y(s;t,\bz)) \sdd\bx_s,\quad t\le s \le T,\quad \by(t;t,\bz) = \bz.$$ Again, the solution $\by(.; t,\bz)$ exists and is unique for all $(t,\bz)$ by Theorem \ref{thm:RDEUniqueness}. For a given function $g:\R^e \to \R^c$, define the function $u:[0,T]\times \R^e\to \R^c$ by $$u(t, z) = g(y(T; t, z)).$$ Finally, define the local error $\bm{e}$ by $$\bm{e}(t_k) = \by(t_k;t_{k-1},\overline{\by}_{t_{k-1}}) - \overline{\by}_{t_k}.$$

Recall that we denote by normal, not bold letters the projections to the first level. The following theorem considers only the first level $y = \pi_1(\by)$ of the solution.
 
\begin{theorem}\label{thm:RDEExactErrorRepresentationInLocalErrors}\cite[cf. Theorem 2.1]{moon2003variational}
Let $\bx\in C^{p\var,0}([0,T],\R^d)$, let $f\in\mathcal{V}^{\gamma,p}(\R^d,\R^e)$ with $\gamma > p\ge 1$, and let $y_0\in \R^e$. Furthermore, let $g\in C^2(\R^e, \R^c)$ be twice continuously differentiable. Let $y$ be the unique solution to the first-level RDE \eqref{eqn:FirstLevelRDE}, and assume that the approximation $\overline{y}$ is continuous as a map of the rough path $\bx$ in the $p$-variation topology. Then, the global error is a weighted sum of the local errors $e(t_k)$ given by $$g(y_T) - g(\overline{y}_T) = \sum_{k=1}^n \left(\int_0^1 \Psi\left(t_k,\overline{y}_{t_k} + se(t_k)\right)\sdd s\right) e(t_k),$$ where the function $\Psi(t,z)\in\R^c\times\R^e$ is the first variation of $u$ in the sense that $$\Psi(t,z) = \nabla_z u(t,z),\quad \text{i.e.}\quad \Psi_{i,j}(t,z) = \partial_{z_j}u_i(t,z).$$ The weight function $\Psi$ satisfies, for $t<s<T$, the dual equation
\begin{align*}
-\dd\Psi(s,y(s;t,z))& = \Psi(s,y(s;t,z))f'(y(s;t,z))\sdd\bx_s,\\
\Psi(T,y(T;t,z)) &= (\nabla g)(y(T; t, z)).
\end{align*}
\end{theorem}

\begin{proof}
Assuming for a moment that $g:\R^e\to \R^e$ is the identity function, we see that $\Psi$ exists by \cite[Theorem 11.6]{friz2010multidimensional}. Hence, $\Psi$ also exists if $g$ is differentiable.

\textbf{Step 1: Error formula for finite variation paths.} Assume first that $x = \bx$ is a finite variation path, and use ordinary differential equations instead of rough differential equations. Using a combination of Lemma \ref{lem:PNonExplosionImpliesQNonExplosion}, Lemma \ref{lem:RDESolutionIsODESolution1RP}, and Theorem \ref{thm:RDEExistence}, together with the fact that we will only consider ODEs with starting values in a bounded set (depending on the finite sequence $(\overline{y}_{t_k})$), we see, using Theorem \ref{thm:WhitneysTheoremStein}, that we can assume without loss of generality that $f\in L(\R^d,\lip^\gamma(\R^e,\R^e))$ is bounded.

By the basic property of the flow $y$, we have
\begin{align*}
u(t_k,y(t_k,t_{k-1},\overline{y}_{t_{k-1}})) &= u(t_{k-1},\overline{y}_{t_{k-1}}).
\end{align*}
Hence, 
\begin{align*}
g(y_T) - g(\overline{y}_T) &= g(y(T;0,y_0)) - g(y(T; T,\overline{y}_T)) = u(0,\overline{y}_0) - u(T,\overline{y}_T)\\
&= \sum_{k=1}^n \left(u(t_k,y(t_k;t_{k-1},\overline{y}_{t_{k-1}})) - u(t_k, \overline{y}_{t_k})\right)\\
&= \sum_{k=1}^n \left(\int_0^1 \Psi\left(t_k,\overline{y}_{t_k} + se(t_k)\right)\sdd s\right) e(t_k),
\end{align*}
where the last line follows from the fundamental theorem of calculus. 

\textbf{Step 2: Dual equation for finite variation paths and smooth vector fields.} It remains to show that $\Psi$ satisfies the above ODE. Assume for now that we even have $f\in L(\R^d,\lip^{\gamma\lor (2+\eps)}(\R^e,\R^e))$. By \cite[Theorem 11.6]{friz2010multidimensional} (together with Lemma \ref{lem:RDESolutionIsODESolution1RP} and the fact that $g$ is twice continuously differentiable), $\Psi$ is therefore continuously differentiable. Consider, for some $w$ and some $\delta \to 0$ two solutions $y^1 \coloneqq y(.;t,z)$ and $y^2 \coloneqq y(.;t,z+\delta w),$ and note that the function $s\mapsto u(s, y^i_s)$ is constant for $i=1,2$. Then, for $s\ge t$,
\begin{align*}
0 &= u(s, y^2_s) - u(s, y^1_s) - \left(u(t, y^2_t) - u(t, y^1_t)\right)\\
&= \Psi(s, y^1_s)(y^2_s-y^1_s) - \Psi(t, y^1_t)(y^2_t-y^1_t) + o(\delta)\\
&= \int_t^s \dd\Psi(r, y^1_r)(y^2_r-y^1_r) + \int_t^s \Psi(r, y^1_r) \left(\dd y^2_r - \dd y^1_r\right) + o(\delta)\\
&= \int_t^s \dd\Psi(r, y^1_r)(y^2_r-y^1_r) + \int_t^s \Psi(r, y^1_r) \left(f(y^2_r) - f(y^1_r)\right)\sdd x_r + o(\delta)\\
&= \int_t^s \dd\Psi(r, y^1_r)(y^2_r-y^1_r) + \int_t^s \Psi(r, y^1_r) f'(y^1_r)\left(y^2_r - y^1_r\right)\sdd x_r + o(\delta)\\
&= \int_t^s \left(\left(\dd\Psi(r, y^1_r) + \Psi(r, y^1_r) f'(y^1_r) \sdd x_r\right)(y^2_r-y^1_r)\right) + o(\delta).
\end{align*}
The second equality follows from the fact that $u$ is twice continuously differentiable (or equivalently that $\Psi$ is continuously differentiable), together with the fact that the $y^1$ and $y^2$ take values in a bounded set, as was discussed previously already. Similarly, the penultimate equality follows from the fact that $f$ is twice continuously differentiable (with bounded derivative), and that $\Psi$ is continuous, and $y^1$ takes values in a bounded set, making $\Psi(r, y_r^1)$ uniformly bounded.

Since the map $z\mapsto y(s; t, z)$ is a diffeomorphism with non-zero derivative, we have that $y^2_s-y^1_s = \Theta(\delta)$ (i.e. it decays not faster and not slower than $\delta$). Hence, dividing by $\delta$ and taking $\delta \to 0$, we get $$0 = \int_t^s \left(\left(\dd\Psi(r, y^1_r) + \Psi(r, y^1_r) f'(y^1_r) \sdd x_r\right)(\nabla_z y(r; t, z) w)\right).$$ Subtracting the same integral equation on a different interval $[t, \tau]$, with $\tau < s$, we get
\begin{align*}
0 &= \int_\tau^s \left(\left(\dd\Psi(r, y^1_r) + \Psi(r, y^1_r) f'(y^1_r) \sdd x_r\right)(\nabla_z y(r; t, z) w)\right)\\
&= \int_\tau^s \left(\left(\dd\Psi(r, y^1_r) + \Psi(r, y^1_r) f'(y^1_r) \sdd x_r\right) w\right)\\
&\qquad + \int_\tau^s \left(\left(\dd\Psi(r, y^1_r) + \Psi(r, y^1_r) f'(y^1_r) \sdd x_r\right)\left((\nabla_z y(r; t, z) - \nabla_z y(t; t, z)) \cdot w\right)\right),
\end{align*}
where we used that $\nabla_z y(t; t, z)$ is the identity. Since $f$ is twice continuously differentiable, $\tau\mapsto \nabla_z y(\tau; t, z)$ is Lipschitz continuous, with a Lipschitz constant independent of $t$ and $z$. Hence, the second summand above is of order $O((s-\tau)^2)$, compared to $O(s-\tau)$ for the first summand. Taking $\tau\uparrow s$, we obtain the differential equation $$\left(\dd\Psi(s, y^1_s) + \Psi(s, y^1_s) f'(y^1_s) \sdd x_s\right) w = 0.$$ Since this holds for all $w$, we have shown the differential equation in the theorem.

\textbf{Step 3: Dual equation for finite variation paths and all vector fields.} Now, assume that we only have $f\in L(\R^d,\lip^\gamma(\R^e,\R^e))$. We have to prove that $\Psi$ still satisfies the above ODE. We can find an approximating sequence $(f^n)$ in $L(\R^d,\lip^{\gamma\lor(2+\eps)}(\R^e,\R^e))$ converging to $f$ in $L(\R^d,\lip^\gamma(\R^e,\R^e))$. By \cite[Proposition 4.6]{friz2010multidimensional} together with \cite[Theorem 3.15]{friz2010multidimensional}, $\Psi^n(t,z)\to \Psi(t,z)\coloneqq \nabla_z u(t,z) = \nabla g(y(T; t,z)) \nabla_z y(T; t,z)$ uniformly, uniformly in $z$ for $z$ in a bounded set, since $\gamma > 1$.

By \cite[Theorem 3.15]{friz2010multidimensional}, the associated solutions $(y^n(.; t, z))_n$ satisfy $$\|y^n(.; t, z) - y(.; t, z)\|_\infty \le C\|f^n - f\|_\infty.$$ Therefore,
\begin{align*}
\|(f^n)'(y^n(.; t, z)) - f'(y(.; t, z))\|_\infty &\le \|(f^n)'(y^n(.; t, z)) - f'(y^n(.; t, z))\|_\infty\\
&\qquad + \|f'(y^n(.; t, z)) - f'(y(.; t, z))\|_\infty\\
&\le \|f^n-f\|_{\lip^\gamma}\\
&\qquad + C\|y^n(.; t, z) - y(.; t, z)\|_\infty^{(\gamma-1)\land 1}\\
&\le C\left(\|f^n-f\|_{\lip^\gamma}\lor\|f^n-f\|_{\lip^\gamma}^{\gamma-1}\right).
\end{align*}

Since $f'$ is bounded, the path $$h^{t,z}_s \coloneqq \int_t^s f'(y(r; t, z)) \sdd x_r$$ is well-defined, and again of finite variation. We have
\begin{align*}
\|h^{t,z}_u - h^{t,z}_s - (h^{n,t,z}_u - h^{n,t,z}_s)\| &= \left\|\int_s^u\left(f'(y(r;t,z)) - (f^n)'(y^n(r;t,z))\right)\sdd x_r\right\|\\
&\le C\left(\|f^n-f\|_{\lip^\gamma}\lor\|f^n-f\|_{\lip^\gamma}^{\gamma-1}\right)\|x\|_{1\var;[s,u]}.
\end{align*}
Since the right-hand side is a control function, we conclude that $h^{n,t,z}\to h^{t,z}$ in $1$-variation.

Since $h^{t,z}$ is of finite variation, the differential equation $$-\dd\Phi^{t,z}(s) = \Phi^{t,z}(s) \sdd h^{t,z}_s,\qquad \Phi^{t,z}(T) = (\nabla g)(y(T; t,z))$$ admits a unique solution, which we call $\Phi^{t,z}$. Of course, $\Psi^{n,t,z}$ satisfy similar differential equations. Then, by \cite[Corollary 3.20]{friz2010multidimensional}, $\Psi^{n,t,z}\to \Phi^{t,z}$ in $1$-variation, and hence pointwise. We conclude by noting that $$\Phi^{t,z}(s) = \lim_{n\to\infty} \Psi^{n,t,z}(s) = \lim_{n\to\infty} \Psi^n(s, y^n(s; t, z)) = \Psi(s, y(s; t, z)) =: \Psi^{t,z}(s).$$ This proves the theorem for finite variation paths.

\textbf{Step 4: Error formula for rough paths.} Let $\bx\in C^{p\var,0}([0,T],\R^d)$, and let $(x^n)$ be a sequence of finite variation paths converging to $\bx$ in $p$-variation. We define everything that has been defined for $\bx$ also for $x^n$ by adding a superscript $n$, for example, $y^n$ is the solution to $$\dd y^n_t = f(y^n_t) \sdd x^n_t,\quad y^n_0 = y_0.$$ Since $f\in \mathcal{V}^{\gamma,p}(\R^d,\R^e)$, $(y^n)$ converges to $y$ in $p$-variation by Theorem \ref{thm:RDEUniqueness}. Also, since the approximation scheme is continuous in the $p$-variation topology, $(\overline{y}^n)$ converges to $\overline{y}$. Therefore, also $(e^n)$ converges to $e$.

Now, by \cite[Theorem 11.12]{friz2010multidimensional}, $\Psi$, the first variation of $u$, depends continuously on the rough path $\bx$ in the $p$-variation topology. Hence, as $(x^n)$ converges to $\bx$ in $p$-variation, we have $\Psi^n\to \Psi$ pointwise. Moreover, also by \cite[Theorem 11.12]{friz2010multidimensional}, we see that $(\Psi^n)$ can be uniformly bounded on compact sets, uniformly over $n$. By the dominated convergence theorem,
\begin{align*}
\bigg\|\int_0^1\big(\Psi^n(t_k,\overline{y}^n_{t_k} &+ se^n(t_k)) - \Psi(t_k,\overline{y}_{t_k} + se(t_k))\big)\sdd s\bigg\|\\
&\le \int_0^1\left\|\Psi^n(t_k,\overline{y}^n_{t_k} + se^n(t_k))-\Psi(t_k,\overline{y}_{t_k} + se(t_k))\right\|\sdd s\\
&\le \int_0^1\left\|\Psi^n(t_k,\overline{y}^n_{t_k} + se^n(t_k))-\Psi(t_k,\overline{y}^n_{t_k} + se^n(t_k))\right\|\sdd s\\
&\qquad + \int_0^1\left\|\Psi(t_k,\overline{y}^n_{t_k} + se^n(t_k))-\Psi(t_k,\overline{y}_{t_k} + se(t_k))\right\|\sdd s\\
&\to 0
\end{align*}
Combining this with the convergence of $y$, $\overline{y}$ and $e$, we see that the error representation formula also holds for geometric rough paths $\bx$.

\textbf{Step 5: Dual equation for rough paths.} It remains to show that $\Psi$ still satisfies the dual equation. To see that, fix $(t, z)$, and define $\Psi^{t,z}(s) \coloneqq \Psi(s, y(s, t, z))$. Also, define the rough path $$\bm{h}^{t,z} \coloneqq \int f'(y(r;t,z))\sdd\bx_r.$$ This rough path is well-defined by Lemma \ref{lem:RoughIntegralWellDefined1}. Moreover, since $(x^n)$ converges to $\bx$ in $p$-variation, we have that the lifts $(S_{[p]}(h^{n,t,z}))$ of $(h^{n,t,z})$ (which are defined similarly) also converge to $\bm{h}^{t,z}$ in $p$-variation.

It remains to show that 
\begin{equation}\label{eqn:SomeOtherRDE}
-\dd\Psi^{t,z}(s) = \Psi^{t,z}(s)\sdd\bm{h}^{t,z}_s,\qquad \Psi^{t,z}(T) = (\nabla g)(y(T; t, z)).
\end{equation}
We know that $\Psi^{n,t,z}$ satisfies the same RDE with $\bm{h}^{t,z}$ replaced by $\bm{h}^{n,t,z}$. The vector field in this RDE is obviously linear, and hence in $\mathcal{V}^\gamma(\R^{e^2},\R^{e^2})$. By the bound in \cite[Theorem 10.53]{friz2010multidimensional}, this linear vector field also satisfies the $p$ non-explosion condition. Thus, by Theorem \ref{thm:RDEUniqueness}, there exists a unique solution $\Phi^{t,z}$ to \eqref{eqn:SomeOtherRDE}, and we have that $(\Psi^{n,t,z})$ converges to $\Phi^{t,z}$ in $p$-variation, and hence also pointwise.

It remains to show that $\Phi^{t,z}(s) = \Psi^{t,z}(s)$. We see that $$\Phi^{t,z}(s) = \lim_{n\to\infty} \Psi^{n,t,z}(s) = \lim_{n\to\infty} \Psi^n(s, y(s; t, z)) = \Psi(s, y(s; t, z)) = \Psi^{t, z}(s).$$ This finishes the proof of the theorem.
\end{proof}

Theorem \ref{thm:RDEExactErrorRepresentationInLocalErrors} gives an error representation formula for first-level RDEs. The following corollary gives a similar statement for full RDEs.

\begin{corollary}\label{cor:FullRDEExactErrorRepresentationInLocalErrors}\cite[cf. Theorem 2.1]{moon2003variational}
Let $\bx\in C^{p\var,0}([0,T],\R^d)$, let $f\in\mathcal{V}^{\gamma,p}(\R^d,\R^e)$ with $\gamma > p\ge 1$, and let $\by_0\in G^{[p]}(\R^e)$. Furthermore, let $g:T_1^{[p]}(\R^e)\to \R^c$. Let $\by$ be the unique solution to the full RDE \eqref{eqn:FullRDE}, and assume that the approximation $\overline{\by}$ takes values in $G^{[p]}(\R^e)$ and is continuous as a map of the rough path $\bx$ in the $p$-rough path topology. Then, the global error is a weighted sum of the local errors $\bm{e}(t_k)$ given by $$g(\by_T) - g(\overline{\by}_T) = \sum_{k=1}^n \left(\int_0^1 \bm{\Psi}\left(t_k,\overline{\by}_{t_k} + s\bm{e}(t_k)\right)\sdd s\right) \bm{e}(t_k),$$ where the function $\bm{\Psi}(t,\bz)\in T_1^{[p]}(\R^e)\times T_1^{[p]}(\R^e)$ is the first variation of $u(t, \bz) \coloneqq g(\by(T; t, \bz))$ in the sense that $$\bm{\Psi}(t,\bz) = \nabla_{\bz} u(t,\bz),\quad \text{i.e.}\quad \bm{\Psi}_{i,j}(t,\bz) = \partial_{\bz_j}u_i(t,\bz).$$ The weight function $\bm{\Psi}$ satisfies, for $t<s<T$, the dual equation
\begin{align*}
-\dd\bm{\Psi}(s,\by(s;t,\bz)) &= \bm{\Psi}(s,\by(s;t,\bz))\bm{f}'(\by(s;t,\bz))\sdd\bx_s,\\
\bm{\Psi}(T,\by(T;t,\bz)) &= (\nabla g)(\by(T; t, \bz)).
\end{align*}
This dual equation is, despite its appearance, a first-level $\R^c\times T_1^{[p]}(\R^e)$-valued RDE.
\end{corollary}

\begin{proof}
By \cite[Theorem 10.35]{friz2010multidimensional}, the full solution $\by$ is a solution to a first-level $T_1^{[p]}(\R^e)$-valued RDE with vector field $\bm{f}.$ By Lemma \ref{lem:NonExplosionExtendsToFullRDEs}, $\bm{f}$ satisfies the $p$ non-explosion condition. Then, the theorem is a direct corollary of Theorem \ref{thm:RDEExactErrorRepresentationInLocalErrors}.
\end{proof}

\subsection{Admissible algorithms}\label{sec:AdmissibleAlgorithms}

We have now proved an error representation under very general assumptions on the underlying approximation algorithm. However, this error representation is of little use in practice if it cannot be computed. Our next goal will therefore be to give an algorithm to compute this error representation to a sufficent degree of accuracy. To achieve that goal, we need some more conditions on our approximation algorithm. These conditions are given below.

First, we define what an admissible 1-step scheme for RDEs is.

\begin{definition}
An admissible 1-step scheme $A$ is a family $A = (A_N)_{N=1}^\infty$ of (collections of) functions such that for all $d, e, N\ge 1$ the following conditions are satisfied.
\begin{enumerate}
\item The 1-step scheme $A_N$ is a function $$A_N:\mathcal{V}^{N+,N}(\R^d,\R^e)\times \R^e\times G^N(\R^d) \to \R^e.$$
\item Let $p<N$, let $\bx\in C^{p\var,0}([0,T],\R^d)$ be controlled by some control $\omega$, let $f\in L(\R^d,\lip^{N+\eps}(\R^e,\R^e))$, and let $y_0\in U$. Let $y$ be the unique solution to 
\begin{equation}\label{eqn:AlgorithmFirstLevelRDE}
dy_t = f(y_t)d\bx_t,\qquad y_0 = y_0.
\end{equation}
Then, there exists a constant $C = C(N, p, d, e, \|f\|_{\lip^N}, \omega(0,T))$, increasing in the last two parameters, such that $$\|y_T - A_N(f,y_0,S_N(\bx)_{0,T})\| \le C\omega(0,T)^{\frac{N+1}{p}}.$$
\end{enumerate}
\end{definition}

\begin{remark}
Condition 1 simply states that $A_N$ takes as input a vector field, a starting value, and the signature of the rough path, and returns the final value.

Condition 2 guarantees that the approximation is actually ``good'' for bounded vector fields. The constant $C$ may still depend on $\omega(0,T)$, as the error bound may explode for large $\omega(0,T)$ (which is a case we are not interested in, but for correctness this has to be included). The constant must be increasing in $\omega(0,T)$ so that the convergence rate of $\frac{N+1}{p}$ for $\omega(0,T)\to 0$ is not destroyed. Similar considerations hold for $\|f\|_{\lip^N}$. This can be seen since there is a scaling invariance in the differential equation \eqref{eqn:AlgorithmFirstLevelRDE} between $f$ and $\bx$.
\end{remark}

Given a 1-step scheme $A$, we obtain an algorithm for solving an RDE by fixing the degree $N$ and choosing a partition $0=t_0 < \dots < t_n = T$, and then successively applying $$\overline{y}_{t_{i+1}} = A_N(f, \overline{y}_{t_i}, S_N(\bx)_{t_i,t_{i+1}}),$$ where $\overline{y}_0 = y_0.$

The definition of an admissible 1-step scheme $A$ guaranteed that the approximation of $A$ is close to the true solution for bounded Lipschitz vector fields. However, since we work with unbounded vector fields, we need to impose some additional constraints on $A$ to ensure that $A$ still yields accurate approximations for vector fields satisfying non-explosion conditions. In the definition below, the first condition is rather weak and says that the 1-step scheme $A$ depends on $f$ only in a neighbourhood of the starting point $y_0$. The second condition is a (usually) stronger stability condition ensuring that the approximations of the entire algorithm $A$ over the interval $[0, T]$ do not explode. We remark that A- and B-stability which are more commonly found in numerical analysis imply this second stability condition.

\begin{definition}\label{def:LocalAdmissibleAlgorithm}
Let $A$ be an admissible 1-step scheme. $A$ is called local if it additionally satisfies the following two conditions.
\begin{enumerate}
\item For all $d, e, N \ge 1$, and all $f\in \mathcal{V}^{N+,N}(\R^d,\R^e)$, there exists an increasing function $r:[0,\infty)\to [0,\infty)$ such that for all $y\in \R^e$ and all $\bm{g}\in G^N(\R^d)$, $A_N(f, y,\bm{g})$ does not depend on $f$ outside the ball $B$ centered at $0$ with radius $r(\|y\|+ \|\bm{g}\|).$ By that we mean that for all $\widetilde{f}\in\mathcal{V}^{N+,N}(\R^d,\R^e)$ that agree with $f$ on $B$, we have $$A_N(f,y,\bm{g}) = A_N(\widetilde{f},y,\bm{g}).$$
\item Let $d, e, N\ge 1$, $f\in \mathcal{V}^{N+,N}(\R^d,\R^e)$, $\bx\in C^{p\var,0}([0,T],\R^d)$, where $p<N$, and $y_0\in \R^e$. Then, there exist constants $C < \infty$ and $\delta > 0$, such that for all partitions $0 = t_0 < \dots < t_n = T$ of $[0,T]$ with $\sup_k\abs{t_k - t_{k-1}} < \delta$, if we define $$y_k \coloneqq A_N(f,y_{k-1},S_N(\bx)_{t_{k-1},t_k}),$$ then $$\sup_k\|y_k\| \le C.$$
\end{enumerate}
\end{definition}

The following lemma shows that local admissible 1-step schemes give good approximations even for locally Lipschitz vector fields satisfying a non-explosion condition.

\begin{lemma}\label{lem:LocalAdmissibleAlgorithmErrorBound}
Let $A$ be a local admissible 1-step scheme, let $\bx\in C^{p\var}([0,T],\R^d)$ controlled by some control $\omega$, let $f\in \mathcal{V}^{N+,N}(\R^d,\R^e)$ with $N > p$, let $y_0\in \R^e$, and let $y$ be a solution to the RDE \eqref{eqn:AlgorithmFirstLevelRDE}. Then, there exists a constant $C = C(N,p,d,e,f,\omega(0,T)+\|y_0\|)$ increasing in the last component, such that $$\|y_T - A_N(f,y_0,S_N(\bx)_{0,T}))\| \le C\omega(0,T)^{\frac{N+1}{p}}.$$
\end{lemma}

\begin{proof}
Since $f$ satisfies the $p$ non-explosion condition, we can find a number $M>0$ depending only on $\omega(0,T)$ and $\|y_0\|$ in an increasing way, such that the solution $y$ is independent of $f$ outside $B_M$.

Since $A$ is local, we may find a ball centered around $0$ with radius $r(\|y_0\| + \|S_N(\bx)_{0,T}\|)\le r(\|y_0\| + \omega(0,T))$ such that $A_N(f,y_0,S_N(\bx)_{0,T})$ is independent of $f$ outside that ball.

Restricting $f$ to the larger one of these two balls, and then extending this restriction to a bounded vector field $\widetilde{f}$ using Theorem \ref{thm:WhitneysTheoremStein}, we see that
\begin{align*}
\|y_T - A_N(f,y_0,S_N(\bx)_{0,T}))\| &\le C(N,p,d,e,\|\widetilde{f}\|_{\lip^N},\omega(0,T))\omega(0,T)^{\frac{N+1}{p}}\\
&\le C(N,p,d,e,f,\omega(0,T)+\|y_0\|)\omega(0,T)^{\frac{N+1}{p}}.\qedhere
\end{align*}
\end{proof}

\begin{lemma}\label{lem:LocalAdmissibleAlgorithmAppliedToPartition}
Let $A$ be a local admissible 1-step scheme, let $\bx\in C^{p\var}([0,T],\R^d)$ controlled by some control $\omega$, let $f\in \mathcal{V}^{N+,N}(\R^d,\R^e)$ with $N > p$, let $y_0\in \R^e$, and let $y$ be a solution to the RDE \eqref{eqn:AlgorithmFirstLevelRDE}. Let $0 = t_0 < \dots < t_n = T$ be a sufficiently fine partition of $[0,T]$, and define $$\overline{y}_{t_0} = y_0,\qquad \overline{y}_{t_{k+1}} = A_N(f,\overline{y}_{t_k},S_N(\bx)_{t_k,t_{k+1}}).$$ Then, $$\|y_{t_k} - \overline{y}_{t_k}\| \le C\sum_{j=1}^k \omega(t_{j-1},t_j)^{\frac{N+1}{p}},$$ where $C = C(N,p,d,e,f,\omega(0,T) + \|y_0\|)$ is increasing in the last component, and independent of $k$ and the chosen partition.
\end{lemma}

\begin{proof}
It is of course sufficient to prove this lemma for $k=n$, i.e. $t_k = T$. Let $y^k_t = y(t;t_k,\overline{y}_{t_k})$ for $t\in[t_k,T]$. Then, $y^0_T = y_T$ and $y^n_T = \overline{y}_T$, and hence, $$\|y_T - \overline{y}_T\| \le \sum_{k=1}^n \|y^k_T - y^{k-1}_T\|.$$ Since $A$ is a local admissible 1-step scheme, $(\overline{y}_{t_k})$ stays uniformly bounded, independent of the partition (since the partition is sufficiently fine). Hence, we may apply Theorem \ref{thm:RDEUniqueness} to find a constant $C$ independent of the partition such that $$\rho_{p,\omega;[t_k,T]}(y^k,y^{k-1}) \le C\|y^k_{t_k} + y^{k-1}_{t_k}\|.$$ In particular,
\begin{align*}
\|y^k_T - y^{k-1}_T\| &\le \|y^k_{t_k,T} - y^{k-1}_{t_k,T}\| + \|y^k_{t_k} - y^{k-1}_{t_k}\|\\
&\le \rho_{p,\omega;[t_k,T]}(y^k,y^{k-1})\omega(t_k,T)^{1/p} + \|y^k_{t_k} - y^{k-1}_{t_k}\|\\
&\le C\|y^k_{t_k} - y^{k-1}_{t_k}\|\\
&\le C\|A_N(f,\overline{y}_{t_{k-1}},S_N(\bx)_{t_{k-1},t_k}) - y(t_k;t_{k-1},\overline{y}_{t_{k-1}})\|\\
&\le C\omega(t_{k-1},t_k)^{\frac{N+1}{p}},
\end{align*}
where we used Lemma \ref{lem:LocalAdmissibleAlgorithmErrorBound}. Summing over $k$ finishes the proof.
\end{proof}

\begin{lemma}\label{lem:LocalAdmissibleAlgorithmLocallyLipschitz}
Let $A$ be a local admissible 1-step scheme. Then, for all $N\ge 1$, all $f\in \mathcal{V}^{N+,N}(\R^d,\R^e)$, and every $R>0$, there exists a constant $C<\infty$ such that, for all $y\in \R^e$, and all $\bg_1,\bg_2\in G^N(\R^d)$ with $\|y\| + \|\bg_1\|\lor\|\bg_2\| \le R$, we have $$\|A_N(f,y,\bg_1) - A_N(f,y,\bg_2)\| \le  C\left(\vertiii{\bg_1}^{N+1} + \|\bg_1-\bg_2\|\right)$$
\end{lemma}

\begin{proof}
Let $\vertiii{.}_{CC}$ denote the Carnot-Caratheodory norm (see \cite[Theorem 7.32]{friz2010multidimensional}), and let $\gamma_1\in C([0,1],\R^d)$ be a geodesic path from $\bm{1}$ to $\bm{g}_1$ (as in \cite[Theorem 7.32]{friz2010multidimensional}), and let $\gamma_2\in C([0,1],\R^d)$ be a geodesic path from $\bm{1}$ to $\bm{g}_2$. We remark that these geodesic paths are by definition of finite variation.

Let $y^i$ denote the solution to $$dy^i = f(y^i)d\gamma_i,\qquad y^i_0 = y,\qquad i=1,2.$$ This solution exists and is unique by Lemma \ref{lem:RDESolutionIsODESolution1RP}. Recall the 1-step scheme $A^{\textup{Euler}}$ from Appendix \ref{sec:Euler}. We have
\begin{align*}
\|A_N(f,y,\bg_1) - A_N(f,y,\bg_2)\| &\le \|A_N(f,y,\bg_1) - y^1_1\| + \|y^1_1 - A_N^{\textup{Euler}}(f,y,\bg_1)\|\\
&\qquad + \|A_N^{\textup{Euler}}(f,y,\bg_1) - A_N^{\textup{Euler}}(f,y,\bg_2)\|\\
&\qquad + \|A_N^{\textup{Euler}}(f,y,\bg_2) - y^2_1\| + \|y^2_1 - A_N(f,y,\bg_2)\|.
\end{align*}

By Lemma \ref{lem:LocalAdmissibleAlgorithmErrorBound}, \cite[Theorem 10.30]{friz2010multidimensional} (applicable since we have a single Euler step and $A_N^{\textup{Euler}}$ depends on $f$ only in a neighbourhood of $y$), and Lemma \ref{lem:EulerIsLipschitz},
\begin{align*}
\|A_N(f,y,\bg_1) - A_N(f,y,\bg_2)\| &\le C\Big(\|\gamma_1\|^{N+1}_{1\var} + \|\gamma_1\|^{N+1}_{1\var} + \|\bg_1-\bg_2\|\\
&\qquad + \|\gamma_2\|^{N+1}_{1\var} + \|\gamma_2\|^{N+1}_{1\var}\Big)\\
&\le C\left(\vertiii{\bg_1}^{N+1} + \vertiii{\bg_2}^{N+1} + \|\bg_1-\bg_2\|\right).
\end{align*}

To finish the proof, we show that we can estimate $\vertiii{\bg_2}^{N+1}$ by the other terms. Indeed,
\begin{align*}
\vertiii{\bg_2}^{N+1} &= \sup_{i=1,\dots,N} \|\pi_i(\bg_2)\|^{\frac{N+1}{i}}\\
&\le C\sup_{i=1,\dots,N} \left(\|\pi_i(\bg_2 - \bg_1)\|^{\frac{N+1}{i}} + \|\pi_i(\bg_1)\|^{\frac{N+1}{i}}\right)\\
&\le C\bigg(\|\bg_2 - \bg_1\|\sup_{i=1,\dots,N}\|\pi_i(\bg_2 - \bg_1)\|^{\frac{N+1-i}{i}} + \vertiii{\bg_1}^{N+1}\bigg)\\
&\le C\bigg(\|\bg_2 - \bg_1\|\left(\sup_{i=1,\dots,N}\|\pi_i(\bg_2)\|^{\frac{N+1-i}{i}} + \sup_{i=1,\dots,N}\|\pi_i(\bg_1)\|^{\frac{N+1-i}{i}}\right)\\
&\qquad + \vertiii{\bg_1}^{N+1}\bigg)\\
&\le C\bigg(\|\bg_2 - \bg_1\|\left(\vertiii{\bg_2}\lor\vertiii{\bg_2}^N + \vertiii{\bg_1}\lor\vertiii{\bg_1}^N\right) + \vertiii{\bg_1}^{N+1}\bigg)\\
&\le C\left(\|\bg_2 - \bg_1\| + \vertiii{\bg_1}^{N+1}\right).\qedhere
\end{align*}
\end{proof}

\begin{definition}
Let $A$ be an admissible 1-step scheme. We call $A$ group-like, if for all $N\ge 1$, all $\bm{g}\in G^N(\R^d)$, all $f\in\mathcal{V}^{N+,N}(\R^d,\R^e),$ and all $\by\in G^N(\R^e)$, we have that $$A_N(\bm{f},\by,\bm{g}) \in G^N(\R^e).$$
\end{definition}

\subsection{Computing the error representation formula}\label{sec:EstimateWeights}

By Theorem \ref{thm:RDEExactErrorRepresentationInLocalErrors}, we can represent the global error as a weighted sum of the local errors. The weights involve the first variation function $\Psi$ (or $\bm{\Psi}$), and the goal of this section is to give an algorithm for approximating $\Psi$, and thus the error representation formula. We remark that if one is interested in estimating the error $g(\overline{\by}_T) - g(\by_T)$ of the full solution, rather than just the first level $g(\overline{y}_T) - g(y_T)$, then one just needs to replace the vector field $f$ by its full extension $\bm{f}$ in the following discussion. Since everything is completely analogous, we focus on the first level for simplicity.

Before diving into the details, let us recall that if we are given two group-like elements $\bm{a}\in G^N(R^a)$ and $\bm{b}\in G^N(\R^b)$, there is in general no canonical choice for $\bm{c}\in G^N(\R^{a+b})$, with projections $\bm{a}$ and $\bm{b}$ to the first, respectively last, components. This is essentially because we do not know how to make sense of the missing cross-iterated integrals between $\bm{a}$ and $\bm{b}$ that would be required for the definition of $\bm{c}$. However, if, say, $\bm{a} = \bm{1}$ is actually the trivial element, there is an obvious canonical choice. We may just use the trivial zero path for $\bm{a}$, which has the correct signature, and we can easily integrate any path irrespective of regularity against the zero path (the result is of course always $0$). Hence, we may canonically choose $\bm{c}$ such that all its components which contain a component of $\bm{a}$ are simply put to $0$ (while the components exclusively pertaining to $\bm{b}$ are of course already given by $\bm{b}$). We denote this canonical element as $(\bm{1}, \bm{b}) \coloneqq \bm{c}$. The element $(\bm{a}, \bm{1})$ is of course defined similarly.

The algorithm for approximating the error representation formula is outlined below. We assume that we are given a rough path $\bx\in C^{p\var,0}([0,T],\R^d)$, a vector field $f\in\mathcal{V}^{(N+1)+,N}(\R^d,\R^e)$, where $N > p \ge 1$, and an initial value $\by_0\in G^N(\R^e)$. Moreover, our payoff function satisfies $g\in C^2(\R^e, \R^c).$ Also, we are given a local group-like admissible 1-step scheme $A$, and a sufficiently fine partition $0 = t_0 < \dots < t_n = T$ of $[0,T]$.

\begin{enumerate}
\item Let $\bz_0\in G^N(\R^{d+e})$ be the (canonical) element with $\bz_0 = (\bm{1},\by_0).$ Define the vector field $f_1\in\mathcal{V}^{(N+1)+}(\R^d,\R^{d+e})$ given by $$f_1(x,y) = 
\begin{pmatrix}
\id\\
f(y)
\end{pmatrix}
.$$ By Lemma \ref{lem:PNonExplosionAdjoin}, $f_1$ satisfies the $N$ non-explosion condition, and by Theorem \ref{thm:RDEUniqueness}, there exists a unique solution $\bz$ to $$\dd\bz_t = f_1(\bz_t) \sdd\bx_t,\qquad \bz_0 = \bz_0.$$ Define and compute the approximation $\overline{\bz}_0 \coloneqq \bz_0,$ $$\overline{\bz}_{t_{k+1}} = A_N(\bm{f}_1,\overline{\bz}_{t_k},S_N(\bx)_{t_k,t_{k+1}}),\qquad k=0,\dots,n-1.$$
\item Define the function $g\in L(\R^{d+e},\lip^{N+}_{\textup{loc}}(\R^{d+e},\R^{e^2}))$ by $$g(x,y) =
\begin{pmatrix}
f'(y) & 0
\end{pmatrix}.$$ For all $t\in[0,T]$ and $\widehat{\bz}\in G^N(\R^{d+e})$, define the rough path $$\bm{h}^{t,\widehat{\bz}}_{s,u} = \int_s^u g(\bz(r;t,\widehat{\bz}))\sdd\bz(r;t,\widehat{\bz}),$$ where $t\le s \le u\le T$. This rough path is well-defined by Lemma \ref{lem:RoughIntegralWellDefined1}.

The path $\bm{h}^{t,\widehat{\bz}}_{s,u}$ is the solution of a rough integral. To solve this integral, we rewrite it as a differential equation, so that we can apply our algorithm. To this end, define the vector field $f_2\in\mathcal{V}^{N+}(\R^{d+e},\R^{d+e+e^2})$, $$f_2(z,h) = 
\begin{pmatrix}
\id\\
g(z)
\end{pmatrix}
.$$ By Lemma \ref{lem:PNonExplosionIntegral}, $f_2$ satisfies the $N$ non-explosion condition. Let $\bm{v}_0\in G^N(\R^{d+e+e^2})$ be the (canonical) element satisfying $\bm{v}_0 = (\widehat{\bz},\bm{1}).$ Let $\bm{v}$ be a solution to the RDE $$\dd\bm{v}_s = f_2(\bm{v}_s) \sdd\bz_s,\qquad \bm{v}_t = \bm{v}_0,\qquad s\ge t,$$ which exists by Theorem \ref{thm:RDEExistence}. Then, $\bm{h}^{t,\widehat{\bz}}$ is the projection of $\bm{v}$ onto the last $e^2$ coordinates. In fact, it is then not hard to see, by the definition of $f_2$, that the increments of $\bm{h}^{t,\widehat{\bz}}$ only depend on the last $e$ coordinates of $\widehat{\bz}$. Hence, given $\widehat{\by}\in G^N(\R^e)$, let $\widehat{\bz}\in G^N(\R^{d+e})$ be the (canonical) element satisfying $\widehat{\bz} = (\bm{1},\widehat{\by})$. Then, we can safely define $$\bm{h}^{t,\widehat{\by}}_{s,u} \coloneqq \bm{h}^{t,\widehat{\bz}}_{s,u}.$$ Let $\widehat{\bm{v}}_{t_k}\in T_1^N(\R^{d+e+e^2})$ be the (canonical) element satisfying $\widehat{\bm{v}}_{t_k} = (\overline{\bz}_{t_k},\bm{1})$. Define and compute $\overline{\bm{v}}_{t_k,t_{k+1}}\in G^N(\R^{d+e+e^2})$ by $$\overline{\bm{v}}_{t_k,t_{k+1}} \coloneqq \widehat{\bm{v}}_{t_k}^{-1}\otimes A_N(\bm{f}_2,\widehat{\bm{v}}_{t_k}, \overline{\bz}_{t_k}^{-1}\otimes \overline{\bz}_{t_{k+1}}),\qquad k=0,\dots,n-1.$$ This is well-defined since $A$ is group-like. Then, define $\overline{\bm{h}}_{t_k,t_{k+1}}$ to be the projection of $\overline{\bm{v}}_{t_k,t_{k+1}}$ onto the last $e^2$ coordinates.
\item Recall that $\Psi^{t,\widehat{\by}}(s) \coloneqq \Psi(s,\by(s;t,\widehat{\by}))$ satisfies the RDE
\begin{align*}
-\dd\Psi^{t,\widehat{\by}}(s) &= \Psi^{t,\widehat{\by}}(s) f'(\by(s;t,\widehat{\by}))\sdd\bx_s = \Psi^{t,\widehat{\by}}(s) \sdd\bm{h}^{t,\widehat{\by}}_s,\\
\Psi^{t,\widehat{\by}}(T) &= (\nabla g)(y(T; t, \widehat{y})).
\end{align*}
Define the linear vector field $f_3\in\mathcal{V}^{N+}(\R^{e\times e},\R^{c\times e}),$ which is given by $$f_3(y)x = yx,$$ where $y\in \R^{c\times e}$ and $x\in\R^{e\times e}$, and $yx$ is the matrix multiplication. Being linear, $f_3$ satisfies the $N$ non-explosion condition by \cite[Theorem 10.53]{friz2010multidimensional}. We then define and compute the approximation $\overline{\Psi}$ by $\overline{\Psi}(T) = (\nabla g)(\overline{y}_T),$ and $$\overline{\Psi}(t_k) = A(f_3,\overline{\Psi}(t_{k+1}),\overline{\bm{h}}_{t_k,t_{k+1}}^{-1}),\qquad k=0,\dots,n-1.$$ Here, $\overline{y}_T$ is of course the projection of $\overline{z}_T$ onto the last $e$ coordinates.
\item We estimate the local error $e(t_k)$ by $\overline{e}(t_k)$, for example by using a different grid, or a different level of the approximation scheme.
\item We compute the approximation of the global error given by $$g(y_T) - g(\overline{y}_T) \approx \sum_{k=1}^n \overline{\Psi}(t_k)\overline{e}(t_k).$$ 
\end{enumerate}

We will now prove that the approximated error computed by the above algorithm is sufficiently close to the actual error. In what follows, let $\omega$ be a control of $\bx$.

\begin{lemma}\label{lem:ApproximationOfZ}
Under the above assumptions, and in particular assuming that the partition is fine enough, we have $$\|\overline{\bz}_{t_k,t_{k+1}} - S_N(\bz(.;t_k,\overline{\bz}_{t_k}))_{t_k,t_{k+1}}\| \le C\omega(t_k,t_{k+1})^{\frac{N+1}{p}},$$ where $C = C(d,e,p,N,f,\omega(0,T),\|\bz_0\|)$ does not depend on the partition $(t_k)$, and also not on $k$.
\end{lemma}

\begin{proof}
By the second condition in the definition of a local admissible 1-step scheme, we have
\begin{align*}
\|\overline{\bz}_{t_k,t_{k+1}} &- S_N(\bz(.;t_k,\overline{\bz}_{t_k}))_{t_k,t_{k+1}}\|\\
&= \left\|\overline{\bz}_{t_k}^{-1}\otimes A_N(\bm{f}_1,\overline{\bz}_{t_k},S_N(\bx)_{t_k,t_{k+1}}) - \overline{\bz}_{t_k}^{-1}\otimes \bz(t_{k+1};t_k,\overline{\bz}_{t_k})\right\|\\
&\le C\left\|A_N(\bm{f}_1,\overline{\bz}_{t_k},S_N(\bx)_{t_k,t_{k+1}}) - \bz(t_{k+1};t_k,\overline{\bz}_{t_k})\right\|.
\end{align*}
The result then follows from Lemma \ref{lem:LocalAdmissibleAlgorithmErrorBound}.
\end{proof}

\begin{lemma}\label{lem:ApproximationOfH}
Under the above assumptions, we have $$\|\overline{\bm{h}}_{t_k,t_{k+1}} - \bm{h}^{t_k,\overline{\bz}_{t_k}}_{t_k,t_{k+1}}\| \le C\omega(t_k,t_{k+1})^{\frac{N+1}{p}},$$ where $C = C(d,e,p,N,f,\omega(0,T),\|\bz_0\|)$ does not depend on the partition $(t_k)$, and also not on $k$.
\end{lemma}

\begin{proof}
Let $\rho:T_1^N(\R^{d+e+e^2}) \to T_1^N(\R^{e^2})$ be the projection onto the last $e^2$ coordinates. Then, using a similar notation for $\bm{v}$ as for $\bm{h}$, we have, using Lemma \ref{lem:LocalAdmissibleAlgorithmLocallyLipschitz}, and Lemma \ref{lem:LocalAdmissibleAlgorithmErrorBound},
\begin{align*}
\|\overline{\bm{v}}_{t_k,t_{k+1}} - \bm{v}^{t_k,\overline{\bz}_{t_k}}_{t_k,t_{k+1}}\| &= \|\widehat{\bm{v}}_{t_k}^{-1}\otimes A_N(\bm{f}_2,\widehat{\bm{v}}_{t_k},\overline{\bz}_{t_k,t_{k+1}}) - \widehat{\bm{v}}_{t_k}^{-1}\otimes \bm{v}(t_{k+1};t_k,\widehat{\bm{v}}_{t_k})\|\\
&\le C\|A_N(\bm{f}_2,\widehat{\bm{v}}_{t_k},\overline{\bz}_{t_k,t_{k+1}}) - \bm{v}(t_{k+1};t_k,\widehat{\bm{v}}_{t_k})\|\\
&\le C\|A_N(\bm{f}_2,\widehat{\bm{v}}_{t_k},\overline{\bz}_{t_k,t_{k+1}}) - A_N(\bm{f}_2,\widehat{\bm{v}}_{t_k},S_N(\bz(.;t_k,\overline{\bz}_{t_k}))_{t_k,t_{k+1}})\|\\
&\qquad + C\|A_N(\bm{f}_2,\widehat{\bm{v}}_{t_k},S_N(\bz(.;t_k,\overline{\bz}_{t_k})_{t_k,t_{k+1}})) - \bm{v}(t_{k+1};t_k,\widehat{\bm{v}}_{t_k})\|\\
&\le C\vertiii{S_N(\bz(.;t_k,\overline{\bz}_{t_k}))_{t_k,t_{k+1}}}^{N+1}\\
&\qquad + C\left\|\overline{\bz}_{t_k,t_{k+1}} - S_N(\bz(.;t_k,\overline{\bz}_{t_k}))_{t_k,t_{k+1}}\right\|\\
&\qquad + C\|\bz(.;t_k,\overline{\bz}_{t_k})\|_{p\var;[t_k,t_{k+1}]}^{N+1}.
\end{align*}

By \cite[Theorem 9.5]{friz2010multidimensional} and Theorem \ref{thm:RDEExistence},
\begin{align*}
\vertiii{S_N(\bz(.;t_k,\overline{\bz}_{t_k}))_{t_k,t_{k+1}}} &\le \|S_N(\bz(.;t_k,\overline{\bz}_{t_k}))\|_{p\var;[t_k,t_{k+1}]}\\
&\le C\|\bz(.;t_k,\overline{\bz}_{t_k})\|_{p\var;[t_k,t_{k+1}]}\\
&\le C\|\bx\|_{p\var;[t_k,t_{k+1}]} \le C\omega(t_k,t_{k+1})^{1/p}.
\end{align*}

By Lemma \ref{lem:ApproximationOfZ},
\begin{align*}
\left\|\overline{\bz}_{t_k,t_{k+1}} - S_N(\bz(.;t_k,\overline{\bz}_{t_k}))_{t_k,t_{k+1}}\right\| &\le C\omega(t_k,t_{k+1})^{\frac{N+1}{p}}.
\end{align*}

Putting everything together, we get
\begin{align*}
\|\overline{\bm{h}}_{t_k,t_{k+1}} - \bm{h}^{t_k,\overline{\bz}_{t_k}}_{t_k,t_{k+1}}\| &= \|\rho(\overline{\bm{v}}_{t_k,t_{k+1}}) - \rho(\bm{v}^{t_k,\overline{\bz}_{t_k}}_{t_k,t_{k+1}})\| \le \|\overline{\bm{v}}_{t_k,t_{k+1}} - \bm{v}^{t_k,\overline{\bz}_{t_k}}_{t_k,t_{k+1}}\|\\
&\le C\omega(t_k,t_{k+1})^{\frac{N+1}{p}}.\qedhere
\end{align*}
\end{proof}

\begin{lemma}\label{lem:ApproximationOfPsi}
Under the above assumptions, we have $$\|\overline{\Psi}(t_k) - \Psi(t_k,\overline{y}_{t_k})\| \le C\sum_{i=k+1}^n \left(\omega(t_{i-1},t_i)^{1/p}\sum_{j=1}^{i-1}\omega(t_{j-1},t_j)^{\frac{N+1}{p}} + \omega(t_{i-1},t_i)^{\frac{N+1}{p}}\right),$$ where $C = C(d,e,p,N,f,\omega(0,T),\|\bz_0\|)$ does not depend on the partition $(t_k)$, and also not on $k$.
\end{lemma}

\begin{proof}
We prove this lemma for the case $t_k = t_0 = 0$.

Note that\footnote{We recall the notation $\Psi^{t,z}(s) = \Psi(s, y(s; t,z))$. Also, by Theorem \ref{thm:RDEExactErrorRepresentationInLocalErrors}, $\Psi^{t,z}$ satisfies a RDE, and we again denote by $\Psi^{t,z}(s; u, \psi)$ the solution to this RDE if we start $\Psi^{t,z}$ at time $u$ at $\psi$.} $$\Psi(0,\overline{y}_0) = \Psi^{0,\overline{y}_0}(0) = \Psi^{0,\overline{y}_0}(0;T,\Psi^{0,\overline{y}_0}(T)).$$

Define $\Psi^k(t_j)$ on the grid $(t_j)$ such that $$\Psi^k(T) = (\nabla g)(\overline{y}_T),\qquad \Psi^k(t_j) =
\begin{cases}
\Psi^{0,\overline{\by}_0}(t_j;t_{j+1},\Psi^k(t_{j+1})),\qquad & j < k,\\
A_N(f_3,\Psi^k(t_{j+1}),\overline{\bm{h}}^{-1}_{t_j,t_{j+1}}),\qquad & j \ge k.
\end{cases}
$$
In particular, $\Psi^0(t_j) = \overline{\Psi}(t_j)$, and $\Psi^n(t_j) = \Psi^{0,\overline{\by}_0}(t_j; T, (\nabla g)(\overline{y}_T)),$ and $\Psi^n(0) = \Psi^{0,\overline{\by}_0}(0; T, (\nabla g)(\overline{y}_T)) = \Psi(0,y_0; T, (\nabla g)(\overline{y}_T)).$

We have
\begin{align}
\|\overline{\Psi}(0) - \Psi(0,y_0)\| &\le \sum_{k=1}^n \|\Psi^k(0) - \Psi^{k-1}(0)\| \label{eqn:PsiApproximationBound}\\
&\qquad + \left\|\Psi(0, y_0; T, (\nabla g)(\overline{y}_T)) - \Psi(0, y_0; T, (\nabla g)(y_T))\right\|. \nonumber
\end{align}

Let us first consider the last summand. By Theorem \ref{thm:RDEUniqueness},
\begin{align*}
\left\|\Psi(0, y_0; T, (\nabla g)(\overline{y}_T)) - \Psi(0, y_0; T, (\nabla g)(y_T))\right\| &\le C\left\|(\nabla g)(\overline{y}_T) - (\nabla g)(y_T)\right\|.
\end{align*}
Since $A$ is a local admissible 1-step scheme, $(\overline{y}_T$ stays uniformly bounded, independently of the partition (for sufficiently fine partitions). Since $g$ is twice continuously differentiable, $(\nabla g)$ is locally Lipschitz, and we conclude that $$\left\|\Psi(0, y_0; T, (\nabla g)(\overline{y}_T)) - \Psi(0, y_0; T, (\nabla g)(y_T))\right\| \le C\left\|\overline{y}_T - y_T\right\|.$$ Then, by Lemma \ref{lem:LocalAdmissibleAlgorithmAppliedToPartition}, $$\left\|\Psi(0, y_0; T, (\nabla g)(\overline{y}_T)) - \Psi(0, y_0; T, (\nabla g)(y_T))\right\| \le C\sum_{j=1}^n \omega(t_{j-1}, t_j)^{\frac{N+1}{p}}.$$

Now, consider the first summands of \eqref{eqn:PsiApproximationBound}. By Theorem \ref{thm:RDEUniqueness},
\begin{align*}
\|\Psi^k(0) - \Psi^{k-1}(0)\| &= \|\Psi^{0,\overline{y}_0}(0;t_k,\Psi^k(t_k)) - \Psi^{0,\overline{y}_0}(0;t_{k-1},\Psi^{k-1}(t_{k-1}))\|\\
&\le \|\Psi^{0,\overline{y}_0}(0;t_{k-1},\Psi^{0,\overline{y}_0}(t_{k-1};t_k,\Psi^k(t_k)))\\
&\qquad - \Psi^{0,\overline{y}_0}(0;t_{k-1},\Psi^{k-1}(t_{k-1}))\|\\
&\le C\|\Psi^{0,\overline{y}_0}(t_{k-1};t_k,\Psi^k(t_k)) - \Psi^{k-1}(t_{k-1})\|,
\end{align*}
where we also used that $\bm{h}^{0,\overline{\by}_0}$ is still controlled by (a constant times) $\omega$ by Lemma \ref{lem:RoughIntegralWellDefined1}. Estimating further, we get
\begin{align*}
\|\Psi^{0,\overline{y}_0} &(t_{k-1};t_k,\Psi^k(t_k)) - \Psi^{k-1}(t_{k-1})\|\\
&= \|\Psi^{t_{k-1},\by(t_{k-1};0,\overline{y}_0)}(t_{k-1};t_k,\Psi^{k-1}(t_k)) - A_N(f_3,\Psi^{k-1}(t_k),\overline{\bm{h}}^{-1}_{t_{k-1},t_k})\|\\
&\le \|\Psi^{t_{k-1},\by(t_{k-1};0,\overline{y}_0)}(t_{k-1};t_k,\Psi^{k-1}(t_k)) - \Psi^{t_{k-1},\overline{y}_{t_{k-1}}}(t_{k-1};t_k,\Psi^{k-1}(t_k))\|\\
&\qquad + \|\Psi^{t_{k-1},\overline{y}_{t_{k-1}}}(t_{k-1};t_k,\Psi^{k-1}(t_k)) - A_N(f_3,\Psi^{k-1}(t_k), \left(\bm{h}^{t_{k-1},\overline{\by}_{t_{k-1}}}_{t_{k-1},t_k}\right)^{-1})\|\\
&\qquad + \|A_N(f_3,\Psi^{k-1}(t_k), \left(\bm{h}^{t_{k-1},\overline{\by}_{t_{k-1}}}_{t_{k-1},t_k}\right)^{-1}) - A_N(f_3,\Psi^{k-1}(t_k),\overline{\bm{h}}^{-1}_{t_{k-1},t_k})\|.
\end{align*}

For the first term, we note that the only difference is that one of the terms is driven by the rough path $\bm{h}^{t_{k-1},\by(t_{k-1};0,\by_0)}$, and the other by the rough path $\bm{h}^{t_{k-1},\overline{\by}_{t_{k-1}}},$ both times only on the interval $[t_{k-1},t_k]$. By Lemma \ref{lem:RoughIntegralWellDefined3}, we see that $$\rho_{p,\omega;[t_{k-1},t_k]}(\bm{h}^{t_{k-1},\by(t_{k-1};0,\by_0)},\bm{h}^{t_{k-1},\overline{\by}_{t_{k-1}}})\le C\|y(t_{k-1};0,y_0) - \overline{y}_{t_{k-1}}\|.$$
By Lemma \ref{lem:LocalAdmissibleAlgorithmAppliedToPartition}, we have $$\rho_{p,\omega;[t_{k-1},t_k]}(\bm{h}^{t_{k-1},\by(t_{k-1};0,\by_0)},\bm{h}^{t_{k-1},\overline{\by}_{t_{k-1}}})\le C\sum_{j=1}^{k-1}\omega(t_{j-1},t_j)^{\frac{N+1}{p}}.$$ Then, by Theorem \ref{thm:RDEUniqueness},
\begin{align*}
\rho_{p,\omega;[t_{k-1},t_k]}(\Psi^{t_{k-1},y(t_{k-1};0,\overline{y}_0)} &(.;t_k,\Psi^{k-1}(t_k)), \Psi^{t_{k-1},\overline{y}_{t_{k-1}}}(.;t_k,\Psi^{k-1}(t_k)))\\
&\le C\sum_{j=1}^{k-1}\omega(t_{j-1},t_j)^{\frac{N+1}{p}}.
\end{align*}
Therefore,
\begin{align*}
\|\Psi^{t_{k-1},y(t_{k-1};0,\overline{y}_0)}(t_{k-1};t_k,\Psi^{k-1}(t_k)) &- \Psi^{t_{k-1},\overline{y}_{t_{k-1}}}(t_{k-1};t_k,\Psi^{k-1}(t_k))\|\\
&\le C\omega(t_{k-1},t_k)^{1/p}\sum_{j=1}^{k-1}\omega(t_{j-1},t_j)^{\frac{N+1}{p}}.
\end{align*}

For the second term, Lemma \ref{lem:LocalAdmissibleAlgorithmErrorBound} together with the fact that $\bm{h}$ is controlled by $C\omega$ (by Lemma \ref{lem:RoughIntegralWellDefined1}) implies that
\begin{align*}
\|\Psi^{t_{k-1},\overline{y}_{t_{k-1}}}(t_{k-1};t_k,\Psi^{k-1}(t_k)) &- A_N(f_3,\Psi^{k-1}(t_k), \left(\bm{h}^{t_{k-1},\overline{\by}_{t_{k-1}}}_{t_{k-1},t_k}\right)^{-1})\|\\
&\le C\omega(t_{k-1},t_k)^{\frac{N+1}{p}}.
\end{align*}

For the third term, we use Lemma \ref{lem:LocalAdmissibleAlgorithmLocallyLipschitz} to get
\begin{align*}
\|A_N(f_3,\Psi^{k-1}(t_k), &\left(\bm{h}^{t_{k-1},\overline{\by}_{t_{k-1}}}_{t_{k-1},t_k}\right)^{-1}) - A_N(f_3,\Psi^{k-1}(t_k),\overline{\bm{h}}^{-1}_{t_{k-1},t_k})\|\\
&\le C\left(\vertiii{\bm{h}^{t_{k-1},\overline{\by}_{t_{k-1}}}_{t_{k-1},t_k}}^{N+1} + \|\bm{h}^{t_{k-1},\overline{\by}_{t_{k-1}}}_{t_{k-1},t_k} - \overline{\bm{h}}_{t_{k-1},t_k}\|\right).
\end{align*}
Since $A$ is a local admissible 1-step scheme and $(\overline{\by})$ stays uniformly bounded, we can find a constant $C$ independent of the partition and $k$ such that $$\vertiii{\bm{h}^{t_{k-1},\overline{\by}_{t_{k-1}}}_{t_{k-1},t_k}} \le \|\bm{h}^{t_{k-1},\overline{\by}_{t_{k-1}}}\|_{p\var;[t_{k-1},t_k]} \le C\|\bx\|_{p\var;[t_{k-1},t_k]} \le C\omega(t_{k-1},t_k)^{1/p}$$ by Lemma \ref{lem:RoughIntegralWellDefined1}. By Lemma \ref{lem:ApproximationOfH}, $$\|\bm{h}^{t_{k-1},\overline{\by}_{t_{k-1}}}_{t_{k-1},t_k} - \overline{\bm{h}}_{t_{k-1},t_k}\| \le C\omega(t_{k-1},t_k)^{\frac{N+1}{p}}.$$

Putting everything together, we have
\begin{align*}
\|\Psi^k(0) - \Psi^{k-1}(0)\| &\le C\omega(t_{k-1},t_k)^{1/p}\sum_{j=1}^{k-1}\omega(t_{j-1},t_j)^{\frac{N+1}{p}} + C\omega(t_{k-1},t_k)^{\frac{N+1}{p}}.
\end{align*}

Therefore, $$\|\overline{\Psi}(0) - \Psi(0,\overline{y}_0)\| \le C\sum_{k=1}^n \left(\omega(t_{k-1},t_k)^{1/p}\sum_{j=1}^{k-1}\omega(t_{j-1},t_j)^{\frac{N+1}{p}} + \omega(t_{k-1},t_k)^{\frac{N+1}{p}}\right).$$
\end{proof}

\begin{theorem}
Under the above assumptions, we have
\begin{align*}
\bigg\|y_T &- \overline{y}_T - \sum_{k=1}^n \overline{\Psi}(t_k)e(t_k)\bigg\|\\
&\le C\sum_{k=1}^n \sum_{i=k}^n \left(\omega(t_{i-1},t_i)^{1/p}\sum_{j=1}^{i-1}\omega(t_{j-1},t_j)^{\frac{N+1}{p}} + \omega(t_{i-1},t_i)^{\frac{N+1}{p}}\right)\omega(t_{k-1},t_k)^{\frac{N+1}{p}},
\end{align*}
where $C=C(d,e,p,N,f,\omega(0,T),\|y_0\|)$ does not depend on the partition $(t_k)$.
\end{theorem}

\begin{remark}
If the partition $(t_k)$ is chosen such that $$\omega(t_k,t_{k+1}) \le \frac{\omega(0,T)}{n}$$ for all $k$, then
\begin{align*}
\bigg\|y_T - \overline{y}_T &- \sum_{k=1}^n \overline{\Psi}(t_k)e(t_k)\bigg\| \le Cn^{3-\frac{2N+3}{p}}.
\end{align*}
Note that by Lemma \ref{lem:LocalAdmissibleAlgorithmAppliedToPartition}, we have $$\|y_T - \overline{y}_T\|\le C n^{1-\frac{N+1}{p}}.$$ Hence, the theorem shows that correcting the approximation using the error representation formula leads to a faster rate of convergence if $$3 - \frac{2N+3}{p} < 1 - \frac{N+1}{p}\quad\Leftrightarrow\quad N > 2p-2.$$
\end{remark}

\begin{proof}
Recall that $\Psi(t,z)$ is the derivative of $y(T;t,z)$ in $z$. Since $f\in \mathcal{V}^{N+1,N}$ and $p<N$, $z\mapsto y(T; t, z)$ is twice continuously differentiable by \cite[Theorem 11.6]{friz2010multidimensional}. Since the 1-step scheme $A$ is local, we know that $\overline{y}$ stays uniformly bounded, uniformly over the partition. Hence, we may restrict ourselves to a bounded (hence, compact) set, and can therefore assume that the above map has a bounded second derivative. This in turn implies that $\Psi(t,.)$ is Lipschitz continuous. It is easy to see that this Lipschitz constant can also be chosen independently of $t$. This discussion implies that
\begin{align*}
\|\Psi(t_k,\overline{y}_{t_k} &+ se(t_k)) - \overline{\Psi}(t_k)\|\\
&\le \|\Psi(t_k,\overline{y}_{t_k} + se(t_k)) - \Psi(t_k,\overline{y}_{t_k})\| + \|\Psi(t_k,\overline{y}_{t_k}) - \overline{\Psi}(t_k)\|\\
&\le Cs\|e(t_k)\| + C\sum_{i=k+1}^n \left(\omega(t_{i-1},t_i)^{1/p}\sum_{j=1}^{i-1}\omega(t_{j-1},t_j)^{\frac{N+1}{p}} + \omega(t_{i-1},t_i)^{\frac{N+1}{p}}\right)\\
&\le C\omega(t_{k-1},t_k)^{\frac{N+1}{p}}\\
&\qquad + C\sum_{i=k+1}^n \left(\omega(t_{i-1},t_i)^{1/p}\sum_{j=1}^{i-1}\omega(t_{j-1},t_j)^{\frac{N+1}{p}} + \omega(t_{i-1},t_i)^{\frac{N+1}{p}}\right),
\end{align*}
where we have used Lemma \ref{lem:ApproximationOfPsi} and Lemma \ref{lem:LocalAdmissibleAlgorithmErrorBound}.

Then, Theorem \ref{thm:RDEExactErrorRepresentationInLocalErrors} and Lemma \ref{lem:LocalAdmissibleAlgorithmErrorBound} imply that 
\begin{align*}
\Bigg\|y_T &- \overline{y}_T - \sum_{k=1}^n \overline{\Psi}(t_k)e(t_k)\Bigg\|\\
&= \left\|\sum_{k=1}^n \left(\int_0^1\left(\Psi(t_k,\overline{y}_{t_k} + se(t_k)) - \overline{\Psi}(t_k)\right)ds\right)e(t_k)\right\|\\
&\le C\sum_{k=1}^n \Bigg(\omega(t_{k-1},t_k)^{\frac{N+1}{p}}\\
&\qquad + \sum_{i=k+1}^n \left(\omega(t_{i-1},t_i)^{1/p}\sum_{j=1}^{i-1}\omega(t_{j-1},t_j)^{\frac{N+1}{p}} + \omega(t_{i-1},t_i)^{\frac{N+1}{p}}\right)\Bigg)\|e(t_k)\|\\
&\le C\sum_{k=1}^n \sum_{i=k}^n \left(\omega(t_{i-1},t_i)^{1/p}\sum_{j=1}^{i-1}\omega(t_{j-1},t_j)^{\frac{N+1}{p}} + \omega(t_{i-1},t_i)^{\frac{N+1}{p}}\right)\omega(t_{k-1},t_k)^{\frac{N+1}{p}}.
\end{align*}
\end{proof}

\section{Application to the Log-ODE method}\label{sec:LogODE}

Assume we are given $\bg\in G^N(\R^d)$, $f\in \mathcal{V}^{N+,N}(\R^d,\R^e)$, and $y_0\in \R^e$. Consider the ODE
\begin{equation}\label{eqn:LOGODEDefinition}
\dd y_t = \sum_{k=1}^N f^{\circ k}\pi_k(\log_N(\bg))(\id)(y_t)\sdd t,\qquad y_0 = y_0.
\end{equation}
Here, $$f^{\circ 1} = f,\qquad\text{and}\qquad f^{\circ (k+1)} = D(f^{\circ k}) f,$$ are the iterated vector field derivatives of $f$, $\log_N$ is the tensor algebra logarithm defined via the power series expansion as in Appendix \ref{sec:Euler}, and $\pi_k$ is the projection onto the $k$-th level in the tensor algebra. If the solution $y$ to this ODE exists on $[0,1]$ and is unique, we define $$A^{\textup{Log-ODE}}_N(f,y_0,\bg) \coloneqq y_1.$$ We remark that \eqref{eqn:LOGODEDefinition} may look quite intimidating on first sight, yet it is merely an ODE which can be solved with standard ODE solvers. The upcoming Lemma \ref{lem:LogODEIsRDE} will give some geometrical insight into the Log-ODE method and \eqref{eqn:LOGODEDefinition}.

\begin{lemma}\label{lem:XGNRoughPath}
Let $\bg\in G^N(\R^d)$, and define the path $$\bx^{\bg}_{s,t} \coloneqq \exp_N((t-s)\log_N(\bg)),$$ where $$\exp_N(\bm{v}) \coloneqq \sum_{n=0}^N \frac{v^{\otimes n}}{n!}.$$ Then, $\bx^{\bg}\in C^{N\var}([0,1],\R^d)$ is controlled by $\omega(s,t) = C\|\bg\|^N(t-s)$, where $C = C(N)$.
\end{lemma}

\begin{proof}
Indeed, for $u\in[0,1]$, we have 
\begin{align*}
\left\|\pi_k\left(\exp_N\left(u\log_N(\bg)\right)\right)\right\| &= \left\|\pi_k\left(\sum_{\ell=0}^N\frac{1}{\ell!}\left(u\sum_{n=1}^N \frac{(-1)^{n+1}}{n}(\bg-\bm{1})^{\otimes n}\right)^{\otimes \ell}\right)\right\|\\
&\le C\|\bg\|^k u,
\end{align*}
where $C = C(N)$. Hence, $$\vertiii{{\bx}_{s,t}} \le C\|\bg\| (t-s)^{1/N},$$ and $\bx$ is controlled by $\omega(s,t) = C\|\bg\|^N(t-s).$

That $\bx$ takes values in $G^N(\R^d)$ follows from the fact that $\log_N$ and $\exp_N$ are inverses of each other between $G^N$ and the step-$N$ free Lie algebra. This free Lie algebra is a vector space, which means that since $\log_N(\bg)$ is in the free Lie algebra, so is $u\log_N(\bg)$.
\end{proof}

\begin{lemma}\label{lem:LogODEIsRDE}
Let $\bg\in G^N(\R^d)$, let $f\in \mathcal{V}^{N+,N}(\R^d,\R^e)$, and let $y_0\in\R^e$. Then both the ODE $$\dd y_t = \sum_{k=1}^N f^{\circ k}\pi_k(\log_N(\bg))(\id)(y_t)\sdd t,\qquad y_0 = y_0$$ and the RDE $$\dd y_t = f(y_t)\sdd\bx^{\bg}_t,\qquad y_0=y_0$$ admit a unique solution, and these solutions agree.
\end{lemma}

\begin{proof}
We assume first that $f\in L(\R^d,\lip^{N+\eps}(\R^e,\R^e))$. By Lemma \ref{lem:XGNRoughPath}, $\bx^{\bg}\in C^{N\var}([0,1],\R^d)$, and by \cite[Theorem 10.26]{friz2010multidimensional}, the RDE admits a unique solution $y$. Similarly, the vector field of the ODE is bounded and Lipschitz, so there also exists a unique solution $\widetilde{y}$ to the ODE by Picard-Lindelöf. It remains to show that $y = \widetilde{y}$.

Let $\omega(s,t) \coloneqq C\|f\|_{\lip^{N+\eps}}^N\|\bg\|^N(t-s),$ where $C$ is the constant in Lemma \ref{lem:XGNRoughPath}. For every $n\in\N$, let $\mathcal{P}_n = (k2^{-n})_{k=0}^{2^n}$ be a partition of $[0,1]$. Recall the Euler scheme $A^{\textup{Euler}}$ from Appendix \ref{sec:Euler}, and define $(y^n_{k2^{-n}})_{k=0}^n$ to be the Euler approximation of $y$, i.e. $A^{\textup{Euler}}_N$ applied on the partition $\mathcal{P}^n$. By \cite[Theorem 10.30]{friz2010multidimensional}, there exists a constant $C = C(N,N+\eps)$, such that
\begin{align*}
\sup_{k=0,\dots,n}\|y_{t_k^n} - y^n_{t_k^n}\| &\le C e^{C\omega(0,T)}\sum_{k=1}^{2^n}\omega((k-1)2^{-n},k2^{-n})^{\frac{N+1}{N}} \le C 2^{-n/N}.
\end{align*}
In particular, the linear interpolations of $y^n$ (which we again call $y^n$) converge uniformly to $y$.

Now, for all $k$ and $n$, let $y^{n,k}$ be defined to be the linear interpolation of $y^n$ on $[0,k2^{-n}]$, and the solution of $$\dd y_t = \sum_{j=1}^N f^{\circ j}\pi_j(\log_N(\bg))(\id)(y_t)\sdd t,\qquad y_{k2^{-n}} = y^n_{k2^{-n}}$$ on $[k2^{-n},1]$. Of course, $y^{n,0} = \widetilde{y}$, and $y^{n,2^n} = y^n$.

Let $t\in[0,1]$. We have to show that $\|y_t - \widetilde{y}_t\| = 0.$ As $$\|y_t-\widetilde{y}_t\| \le \|y_t - y^n_t\| + \|y^n_t-\widetilde{y}_t\|,$$ and since we have already shown $\|y_t-y^n_t\|\to 0$, it remains to prove $\|y^n_t-\widetilde{y}_t\| \to 0.$ By continuity (and since the sequence $(y^n)$ is equicontinuous, converging to $y$), we may assume that $t$ is a dyadic number, say $t=\ell 2^{-m}$. Let us then assume in the following discussion that $n\ge m$. Then, 
\begin{align*}
\|y^n_t - \widetilde{y}_t\| &= \|y^{n,\ell 2^{n-m}}_t - y^{n,0}_t\| \le \sum_{k=1}^{\ell 2^{n-m}}\|y^{n,k}_t - y^{n,k-1}_t\|.
\end{align*}

Now, we may apply \cite[Theorem 3.8]{friz2010multidimensional} to find a constant $C$ independent of $k$ and $n$ such that 
\begin{align*}
\|y^n_t - \widetilde{y}_t\| &\le C\sum_{k=1}^{\ell 2^{n-m}}\|y^{n,k}_{k2^{-n}} - y^{n,k-1}_{k2^{-n}}\|.
\end{align*}

Recall next that
\begin{align*}
y^{n,k}_{k2^{-n}} &= A^{\textup{Euler}}_N(f,y^n_{(k-1)2^{-n}},\bx^{\bg}_{(k-1)2^{-n},k2^{-n}}).
\end{align*}

Also, $y^{n,k-1}_t$ is the solution to 
$$y^{n,k-1}_{k2^{-n}} = y^n_{(k-1)2^{-n}} + \int_{(k-1)2^{-n}}^{k2^{-n}} \sum_{j=1}^N f^{\circ j}\pi_j(\log_N(\bg))(\id)(y_t^{n,k-1})\sdd t.$$
Now, it is easy to see that in fact,
\begin{align*}
y^{n,k-1}_{k2^{-n}} &= A^{\textup{Log-ODE}}_N(f,y^n_{(k-1)2^{-n}},\exp_N(2^{-n}\log_N(\bg)))\\
&= A^{\textup{Log-ODE}}_N(f,y^n_{(k-1)2^{-n}}, \bx^{\bg}_{(k-1)2^{-n},k2^{-n}}).
\end{align*}
Defining $y^{n,k,\textup{true}}_{k2^{-n}}$ to be the solution at time $k2^{-n}$ to the RDE $$dy_t = f(y_t) d\bx^{\bg}_t,\qquad y_{(k-1)2^{-n}} = y^n_{(k-1)2^{-n}},$$ we see with Lemma \ref{lem:XGNRoughPath}, \cite[Theorem 10.30]{friz2010multidimensional}, and \cite[Theorem 18]{boutaib2013dimension} that $$\|y^{n,k}_{k2^{-n}} - y^{n,k-1}_{k2^{-n}}\| \le \|y^{n,k}_{k2^{-n}} - y^{n,k,\textup{true}}_{k2^{-n}}\| + \|y^{n,k,\textup{true}}_{k2^{-n}} - y^{n,k-1}_{k2^{-n}}\| \le C\|\bg\|^{N+1}2^{-n\frac{N+1}{N}}.$$ Hence, $$\|y^n_t - \widetilde{y}_t\| \le C\ell 2^{n-m}\|\bg\|^{N+1}2^{-n\frac{N+1}{N}} \to 0$$ as $n\to\infty.$ This proves the lemma for $f\in L(\R^d,\lip^{N+\eps}(\R^e,\R^e))$.

Now, let $f\in\mathcal{V}^{N+\eps,N}(\R^d,\R^e)$. By Theorem \ref{thm:RDEUniqueness}, the RDE still admits a unique solution. Conversely, since the vector field in the ODE is still locally Lipschitz, and since uniqueness is a local issue, the solution to the ODE is unique, if it exists. Since the ODE vector field is continuous, a unique solution exists locally, potentially up to some explosion time $\tau$. If the ODE solution does not explode, then both the RDE and the ODE solution stay bounded, and we may replace $f$ with a bounded vector field using Theorem \ref{thm:WhitneysTheoremStein}, and can apply the case that we have already proved. Conversely, if explosion of the ODE solution does happen, then the ODE and the RDE solution must already differ at some time $\tau-\delta\in(0,\tau)$. We may then consider the ODE and the RDE on the time interval $[0,\tau-\delta]$. On this interval, solutions to both the ODE and RDE exist, and no explosion happens. We may thus again apply the bounded case that we have already proved, to conclude that the ODE and the RDE solution coincide. But this is a contradiction to our assumption that they differ at time $\tau-\delta$. Therefore, explosion cannot happen and the lemma is true for $f\in\mathcal{V}^{N+\eps,N}(\R^d,\R^e)$.
\end{proof}

\begin{lemma}
The Log-ODE method $A^{\textup{Log-ODE}}$ is a local group-like admissible 1-step scheme.
\end{lemma} 

\begin{proof}
First, we note that the map $$A_N^{\textup{Log-ODE}}\colon \mathcal{V}^{N+,N}(\R^d,\R^e)\times\R^e\times G^N(\R^d)\to \R^e$$ is well-defined by Lemma \ref{lem:LogODEIsRDE}.

Next, assume $p<N$, let $\bx\in C^{p\var,0}([0,T],\R^d)$ be controlled by $\omega$, and let $f\in L(\R^d,\lip^{N+}(\R^e,\R^e))$. Let $y$ be the unique solution to $$\dd y_t = f(y_t)\sdd\bx_t,\qquad y_0 = y_0.$$ Then, the error bound $$\|y_T - A_N^{\textup{Log-ODE}}(f,y_0,S_N(\bx)_{0,T})\| \le C\omega(0,T)^{\frac{N+1}{p}}$$ was proved, for example, in \cite[Theorem 18]{boutaib2013dimension} (together with \cite[Remark 16]{boutaib2013dimension}). This proves that $A^{\textup{Log-ODE}}$ is an admissible 1-step scheme.

The fact that $A^{\textup{Log-ODE}}$ is group-like follows from $A^{\textup{Log-ODE}}(\bm{f},\by_0,\bg)$ being the solution of a RDE (see Lemma \ref{lem:LogODEIsRDE}), together with Lemma \ref{lem:FullRDEAndFirstLevelRDEWithFullVF} and the fact that full RDE solutions take values in $G^N(\R^e)$.

It remains to prove that $A^{\textup{Log-ODE}}$ is local. To see that, we once again note that $A^{\textup{Log-ODE}}_N(f,y_0,\bg)$ really solves a RDE with respect to the path $\bx^{\bg}\in C^{N\var}([0,1],\R^d)$ by Lemma \ref{lem:LogODEIsRDE}. Also, by Lemma \ref{lem:XGNRoughPath}, there exists a constant $C = C(N)$ such that $$\|\bx^{\bg}\|_{N\var;[0,1]} \le C\|\bg\|^N.$$ Applying Theorem \ref{thm:RDEExistence} shows the first condition of locality.

For the second condition of locality, let $f\in\mathcal{V}^{N+,N}(\R^d,\R^e)$, let $\bx\in C^{p\var,0}([0,T],\R^d)$ where $p<N$, and let $y_0\in U$. Moreover, let $y$ be the unique solution to the RDE $$dy_t = f(y_t) d\bx_t,\qquad y_0 = y_0.$$ By Theorem \ref{thm:RDEExistence}, $y$ stays uniformly bounded by some constant $C$. Hence, define $\widetilde{f}\in L(\R^d,\lip^N(\R^e,\R^e))$ to agree with $f$ on $B_A$, as in Theorem \ref{thm:WhitneysTheoremStein}, where $A> C + 1$ is a constant that will be determined later, and that will not depend on the partition. Then, $y$ is still the unique solution to the RDE with vector field $\widetilde{f}$.

For every partition $\mathcal{P}$ of $[0,T]$, define $\overline{y}^{\mathcal{P}}$ by applying $A^{\textup{Log-ODE}}_N$ to the above RDE on the partition $\mathcal{P}$, but with $f$ replaced by $\widetilde{f}$, and interpolate linearly between the partition points. Then, by an argument similar to Lemma \ref{lem:LocalAdmissibleAlgorithmAppliedToPartition}, we see that the approximation $(\overline{y}^{\mathcal{P}})$ converges uniformly to $y$ as $\abs{\mathcal{P}}\to 0$. In particular, (by uniform continuity of $\omega$) we can find $\delta > 0$ such that for all $\mathcal{P}$ with $\abs{\mathcal{P}} < \delta$, we have $\|y-\overline{y}^{\mathcal{P}}\|_\infty \le 1,$ which implies $$\sup_{\mathcal{P}\colon \abs{\mathcal{P}} < \delta} \sup_{t\in\mathcal{P}} \|\overline{y}^{\mathcal{P}}_t\| \le \|y\|_\infty + 1 \le C + 1.$$

The only thing that remains to show is that we would have gotten the same approximation $\overline{y}^{\mathcal{P}}$ had we applied $A^{\textup{Log-ODE}}$ with the original vector field $f$. To see that, note that for all intervals $[s,t]\subseteq[0,T]$, we have by \cite[Theorem 9.5]{friz2010multidimensional} $$\vertiii{S_N(\bx)_{s,t}} \le \|S_N(\bx)\|_{p\var;[s,t]} \le C\|\bx\|_{p\var;[s,t]} \le C\omega(0,T)^{1/p}.$$ Additionally using that we have already proved that $\|\overline{y}^{\mathcal{P}}_t\| \le C+1$, for all $\mathcal{P}$ fine enough and all $t\in\mathcal{P}$, we may apply the first condition of locality, which we have already proved, to conclude that there exists a constant $D<\infty$ independent of the partition $\mathcal{P}$ and independent of $\widetilde{f}$, such that $A^{\textup{Log-ODE}}_N(f,y,S_N(\bx)_{s,t})$ does not depend on $f$ outside $D$. Choosing $A = D\lor(C+2)$ finishes the proof that $A^{\textup{Log-ODE}}$ is local.
\end{proof}

\section{Numerical examples}\label{sec:Numerics}

We now apply the adaptive algorithm using the error representation formula to a number of examples. These examples were chosen from a wide range of different scenarios where adaptive algorithms may prove useful. They include an example with a path that has a spike, an example with a vector field that has a spike, an example with a path that is rough on some intervals and smooth on others, and an ergodic example. In all these examples, we see that our algorithm for predicting the global error yields an accurate approximation of the true error (which was estimated by computing the solution on a finer grid). We remark that we compute the local errors $e(t_k)$ by subdividing every time interval into 8 subintervals and using as the actual local solution the result of the Log-ODE method applied on these 8 smaller intervals.

Throughout this section, we compare four different algorithms.
\begin{enumerate}
\item ``ER predicting'' computes the error representation formula and uses the cost model of Appendix \ref{sec:CostModel} to determine whether to refine intervals or increase degrees.
\item ``ER testing'' computes the error representation formula. To decide whether to refine intervals or increase degrees, it always tries both options for every interval in question, and then decides based on which option worked better for this particular interval. In a certain heuristic sense, ``ER testing'' always makes the best possible decision on whether it refines the interval or increases the degree, but this comes at the additional cost of having to do extra computations to determine which option is better. We hope to illustrate that ``ER predicting'' uses a similar number of intervals with similar degrees as ``ER testing'', while being faster.
\item ``Simple first level'' does not compute the error representation formula. Rather, it estimates the error by comparing the result to the result using a coarser time grid. If the desired accuracy has not been reached, it uniformly refines all intervals. In particular, it never increases the degrees. This is perhaps the simplest ``reasonable'' algorithm for reaching a certain error tolerance, and we hope to show that ``ER predicting'' can achieve the same error tolerance with significantly fewer intervals.
\item ``Simple full solution'' is essentially the same as ``Simple first level'', except that it computes the full rough path $\bz = (\bx, \by)$. This is because for computing the error representation formula, one already needs to compute $\bz$. Hence, if the full solution is needed, the algorithm ``Simple full solution'' yields a better benchmark for comparing the computational cost than ``Simple first level''.
\end{enumerate}

\subsection{Singularity in the path}

Consider the RDE $$\dd y_t = f(y_t)\sdd\bx_t,\qquad y_0 = 
\begin{pmatrix}
0\\
0
\end{pmatrix},
$$ where $\bx$ is the canonical rough path associated to the finite variation path $x:[0,1]\to \R^2$, $$x_t = \left(\frac{1}{5000(t-0.5)^2 + 1},\ t\right)^T,$$ and where the vector field $f$ is given by $$f(y) = f(y_1, y_2) = 
\begin{pmatrix}
y_2-y_1 & -y_2\\
\tanh(-y_2) & \cos(2y_2-y_1)
\end{pmatrix}.
$$ We use an absolute and a relative error tolerance of $10^{-4}$.

The solution $y$ is shown in Figure \ref{fig:SingularityInThePathSolution}, and the lengths and degrees of the intervals chosen by the adaptive algorithms is in Figure \ref{fig:SingularityInThePathIntervals}. We see that the time discretization is particularly fine around $t=0.5$, just as expected.

Finally, in Table \ref{tab:SingularityInThePathData}, we can find a comparison of the different algorithms. We see that the error representation formula yielded incredibly accurate results. Despite needing significantly more intervals, the simple algorithms were faster at achieving the required error tolerance than the algorithms using the error representation formula. This is of course because computing the error representation formula can be costly. However, if one is only interested in the final point $y_1$, then one can use the error representation formula to achieve a significantly higher accuracy. Indeed, the error representation formula estimates $\textup{err} \approx y_1 -\overline{y}_1$, so that $\overline{y}_1 + \textup{err}$ usually is a better approximation of $y_1$ than $\overline{y}_1$ itself. For comparison, the algorithm ``Simple first level'' takes 524288 intervals and 1282 seconds to achieve an accuracy of $2\cdot 10^{-9}$, while the algorithm ``Simple full solution'' takes 524288 intervals and 1919 seconds.

\begin{figure}
\centering
\includegraphics[scale=0.6]{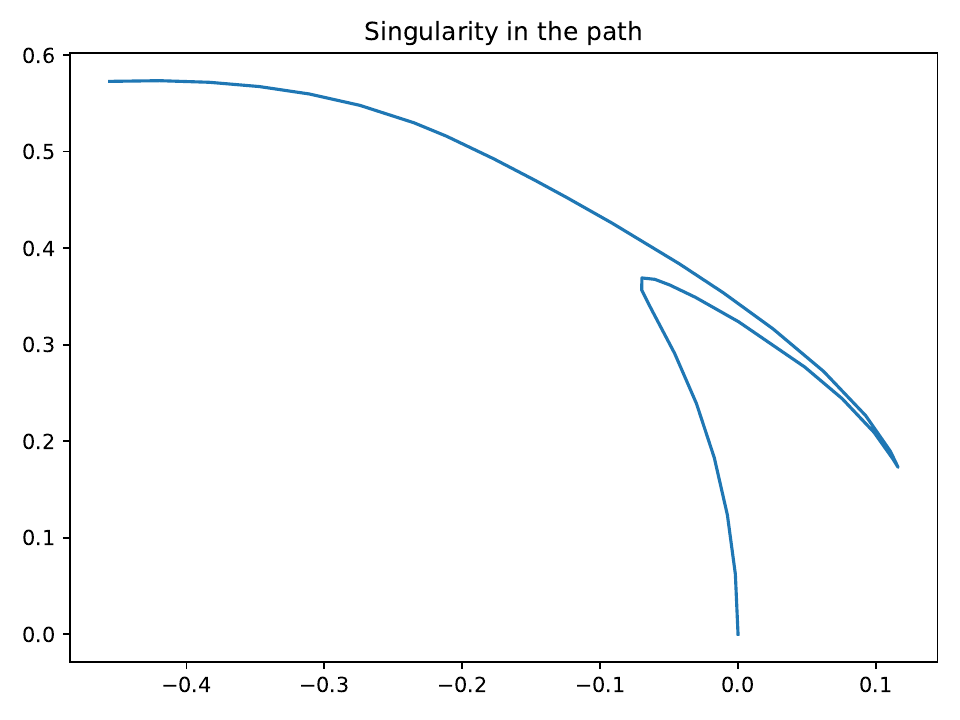}
\caption{The solution path of the example ``Singularity in the path''.}
\label{fig:SingularityInThePathSolution}
\end{figure}

\begin{figure}
\centering
\begin{minipage}{.5\textwidth}
  \centering
  \includegraphics[width=\linewidth]{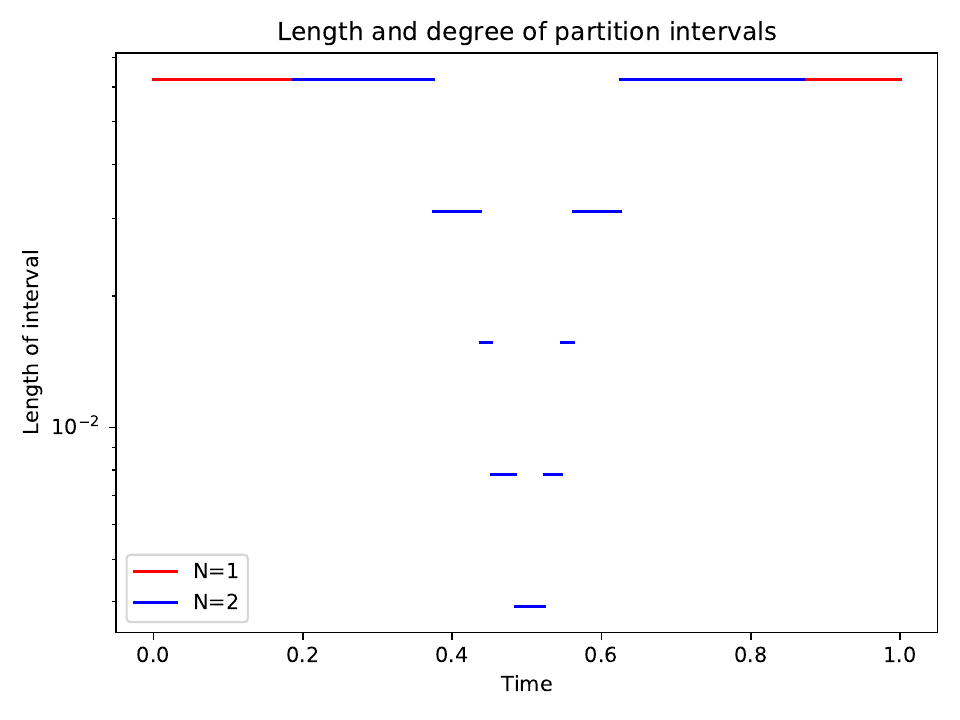}
\end{minipage}%
\begin{minipage}{.5\textwidth}
  \centering
  \includegraphics[width=\linewidth]{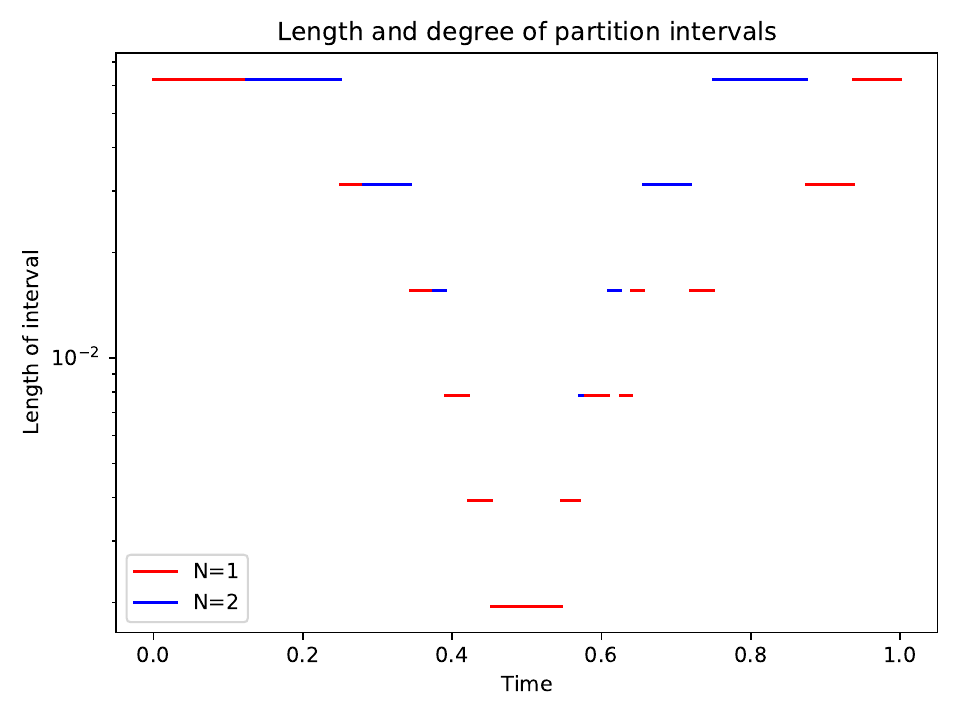}
\end{minipage}
\caption{The lengths of the partition intervals together with the degrees of the Log-ODE method for the example ``Singularity in the path''. The left plot corresponds to ``ER predicting'', the right plot to ``ER testing''.}
\label{fig:SingularityInThePathIntervals}
\end{figure}

\begin{table}[!htbp]
\centering
\resizebox{\textwidth}{!}{\begin{tabular}{c|c|c|c|c}
 & ER predicting & ER testing & Simple first level & Simple full solution\\ \hline
Error & $8.81\cdot 10^{-6}$ & $9.88\cdot 10^{-5}$ & $2.75\cdot 10^{-5}$ & $2.75\cdot 10^{-5}$\\
Estimated error & $8.81\cdot 10^{-6}$ & $9.88\cdot 10^{-5}$ & - & -\\
Error after correction & $1.28\cdot 10^{-9}$ & $7.02\cdot 10^{-9}$ & - & -\\
Degree 1 intervals & 5 & 83 & 1024 & 1024\\
Degree 2 intervals & 30 & 11 & 0 & 0\\
Runtime (s) & 46.14 & 77.40 & 1.385 & 1.934
\end{tabular}}
\caption{Errors, intervals, and runtime of the various algorithms for the example ``Singularity in the path''.}
\label{tab:SingularityInThePathData}
\end{table}

\subsection{Singularity in the vector field}

Consider the RDE $$\dd y_t = f(y_t)\sdd\bx_t,\qquad y_0 = 
\begin{pmatrix}
0\\
0
\end{pmatrix},
$$ where $\bx$ is the canonical rough path associated to the finite variation path $x:[0,1]\to \R^2$, $$x_t = \frac{1}{2}
\begin{pmatrix}
\sin(8\pi t)\\
\cos(8\pi t)
\end{pmatrix},$$ and where the vector field $f$ is given by $$f(y) = f(y_1, y_2) = 
\begin{pmatrix}
y_2-y_1 & -y_2\\
1 + \frac{20}{1000(y_1+1)^2+1} & \frac{20}{1000(y_2+1)^2+1}
\end{pmatrix}.
$$ We use an absolute and a relative error tolerance of $10^{-4}$. 

The solution $y$ is shown in Figure \ref{fig:SingularityInTheVectorFieldSolution}, and the lengths and degrees of the intervals chosen by the adaptive algorithms is in Figure \ref{fig:SingularityInTheVectorFieldIntervals}. From Figure \ref{fig:SingularityInTheVectorFieldSolution} we can read off that the singularities in the vector field are hit roughly at the times $$0.094, 0.188, 0.322, 0.414, 0.524.$$ This perfectly coincides with the locations of the finest time scales in Figure \ref{fig:SingularityInTheVectorFieldIntervals}. 

Finally, in Table \ref{tab:SingularityInTheVectorFieldData}, we can find a comparison of the different algorithms. We see that the error representation formula yielded very accurate results. Once again, despite using significantly fewer intervals, the simple algorithms were faster since they do not have to compute the error representation formula. If we use the error representation formula to correct the approximated solution, we obtain an error of roughly $2\cdot 10^{-6}$. For comparison, ``Simple first level'' needs 65536 intervals and $177.6$ seconds to achieve that error tolerance, while ``Simple full solution'' needs 65536 intervals and $263.4$ seconds.

\begin{figure}
\centering
\begin{minipage}{.5\textwidth}
	\centering
	\includegraphics[width=\linewidth]{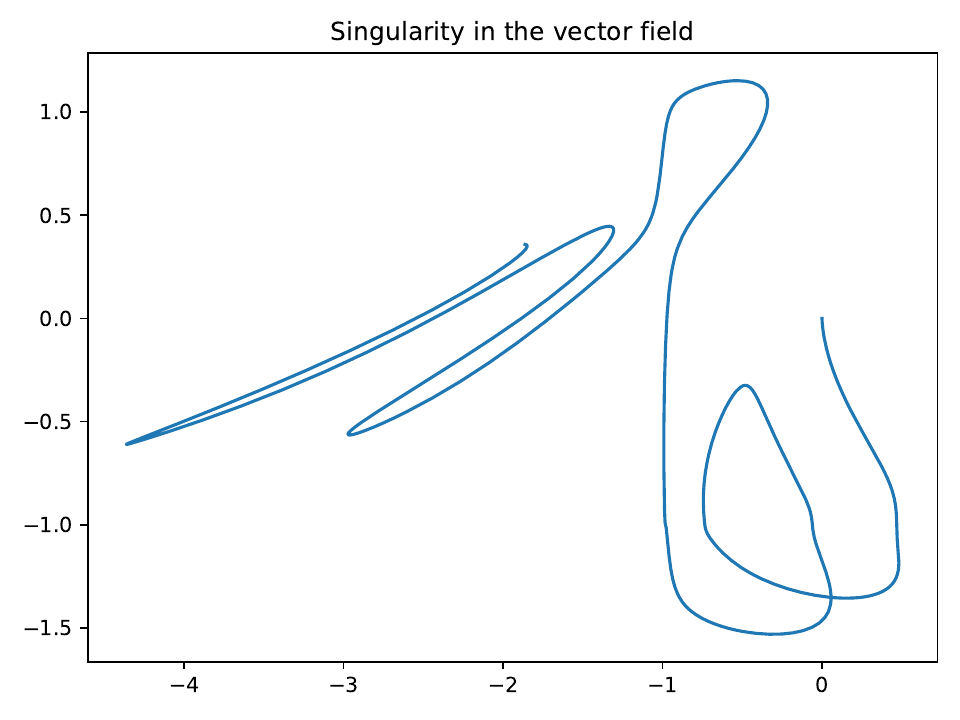}
\end{minipage}%
\begin{minipage}{.5\textwidth}
	\centering
	\includegraphics[width=\linewidth]{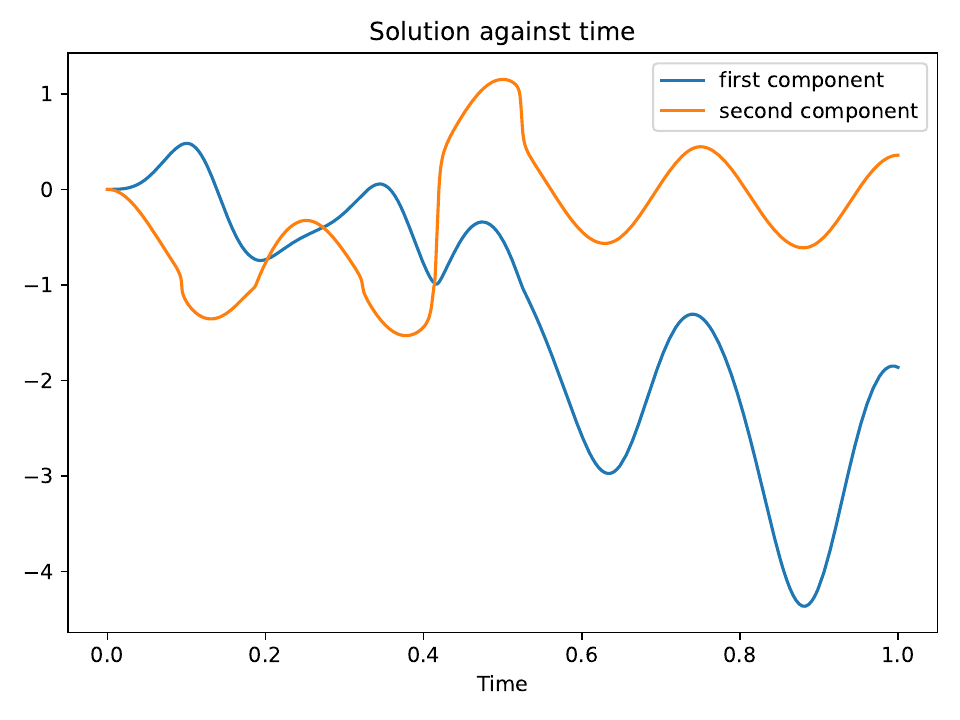}
\end{minipage}
\caption{The solution path of the example ``Singularity in the vector field''. The left plot is the solution path, and the right plot are the two components plotted against time.}
\label{fig:SingularityInTheVectorFieldSolution}
\end{figure}

\begin{figure}
\centering
\begin{minipage}{.5\textwidth}
  \centering
  \includegraphics[width=\linewidth]{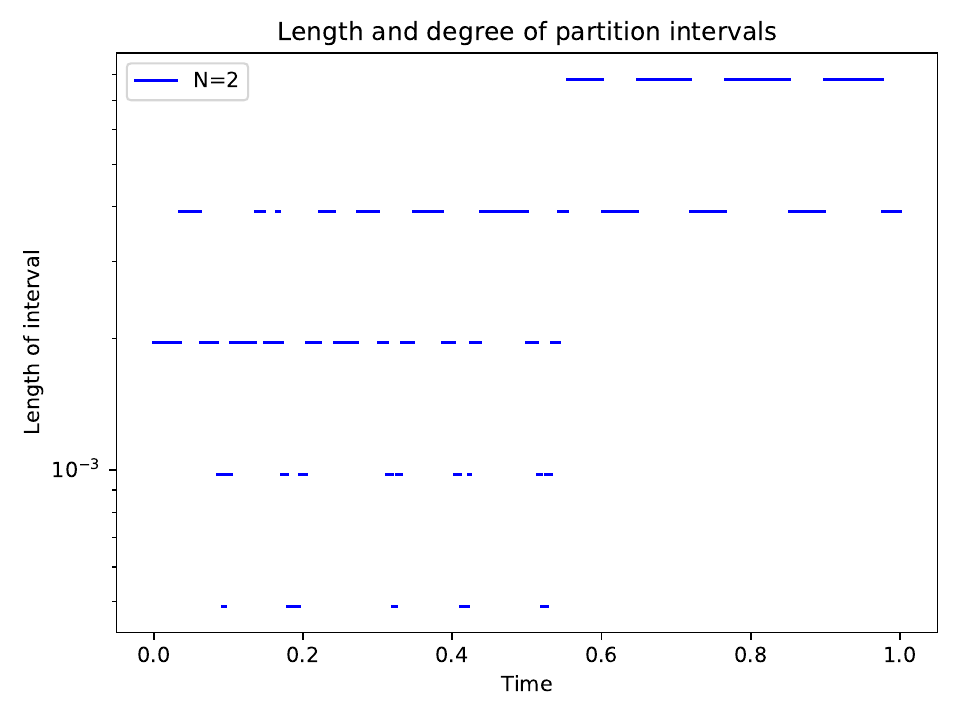}
\end{minipage}%
\begin{minipage}{.5\textwidth}
  \centering
  \includegraphics[width=\linewidth]{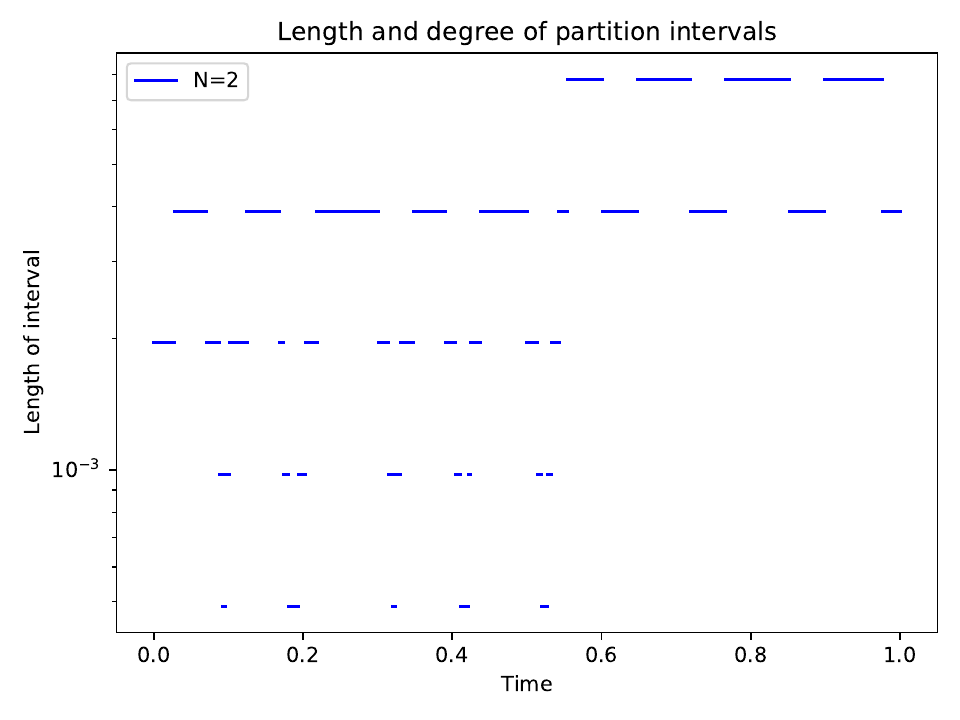}
\end{minipage}
\caption{The lengths of the partition intervals together with the degrees of the Log-ODE method for the example ``Singularity in the vector field''. The left plot corresponds to ``ER predicting'', the right plot to ``ER testing''.}
\label{fig:SingularityInTheVectorFieldIntervals}
\end{figure}

\begin{table}[!htbp]
\centering
\resizebox{\textwidth}{!}{\begin{tabular}{c|c|c|c|c}
 & ER predicting & ER testing & Simple first level & Simple full solution\\ \hline
Error & $3.47\cdot 10^{-5}$ & $3.08\cdot 10^{-5}$ & $3.52\cdot 10^{-5}$ & $3.52\cdot 10^{-5}$\\
Estimated error & $3.34\cdot 10^{-5}$ & $2.84\cdot 10^{-5}$ & - & -\\
Error after correction & $1.25\cdot 10^{-6}$ & $2.38\cdot 10^{-6}$ & - & -\\
Degree 1 intervals & 0 & 0 & 8192 & 8192\\
Degree 2 intervals & 411 & 380 & 0 & 0\\
Runtime (s) & 42.33 & 149.6 & 10.36 & 15.80
\end{tabular}}
\caption{Errors, intervals, and runtime of the various algorithms for the example ``Singularity in the vector field''.}
\label{tab:SingularityInTheVectorFieldData}
\end{table}

\subsection{Path changing roughness}

Consider the RDE $$\dd y_t = f(y_t)\sdd\bx_t,\qquad y_0 = 
\begin{pmatrix}
0\\
0
\end{pmatrix},
$$ where $\bx$ is the canonical rough path associated to the finite variation path $$x_t = \frac{1}{2}
\begin{pmatrix}
\sin(8\pi t)\\
\cos(8\pi t)
\end{pmatrix}$$ on the time intervals $[0, 1/4]$ and $[3/4, 1]$, and a fractional Brownian motion (fBm) with Hurst parameter $H=0.4$ on $[1/4, 3/4]$ (of course glued together correctly). In fact, we did not use an exact fBm, but rather a piecewise linear interpolation, using $1048576$ time intervals on $[0, 1]$. The vector field $f$ is given by $$f(y) = f(y_1, y_2) = 
\begin{pmatrix}
y_2-y_1 & -y_2\\
\tanh(-y_2) & \cos(-y_1 + 2y_2)
\end{pmatrix}.
$$ We use an absolute and a relative error tolerance of $5\cdot 10^{-4}$.

The solution $y$ is shown in Figure \ref{fig:PathChangingRoughnessSolution}, and the lengths and degrees of the intervals chosen by the adaptive algorithms is in Figure \ref{fig:PathChangingRoughnessIntervals}. We see that much finer time scales are used in the interval $[1/4, 3/4]$ corresponding to the fractional Brownian motion.

Finally, in Table \ref{tab:PathChangingRoughnessData}, we can find a comparison of the different algorithms. We see that the error representation formula yielded very accurate results, despite the roughness of the driving path. Moreover, the algorithms using the error representation formula need far fewer intervals, and also tend to be faster. Of course, in this example this is not so much due to the refinement of the intervals, but due to using the degree 3 Log-ODE method over the degree 2 method. Indeed, the interval $[1/4,3/4]$ contributes most to the error, and the fBm on that interval is about equally rough everywhere. Thus, and since $[1/4, 3/4]$ already covers half of the total interval $[0, 1]$, merely choosing optimal intervals will only yield very small improvements.

\begin{figure}
\centering
\includegraphics[scale=0.6]{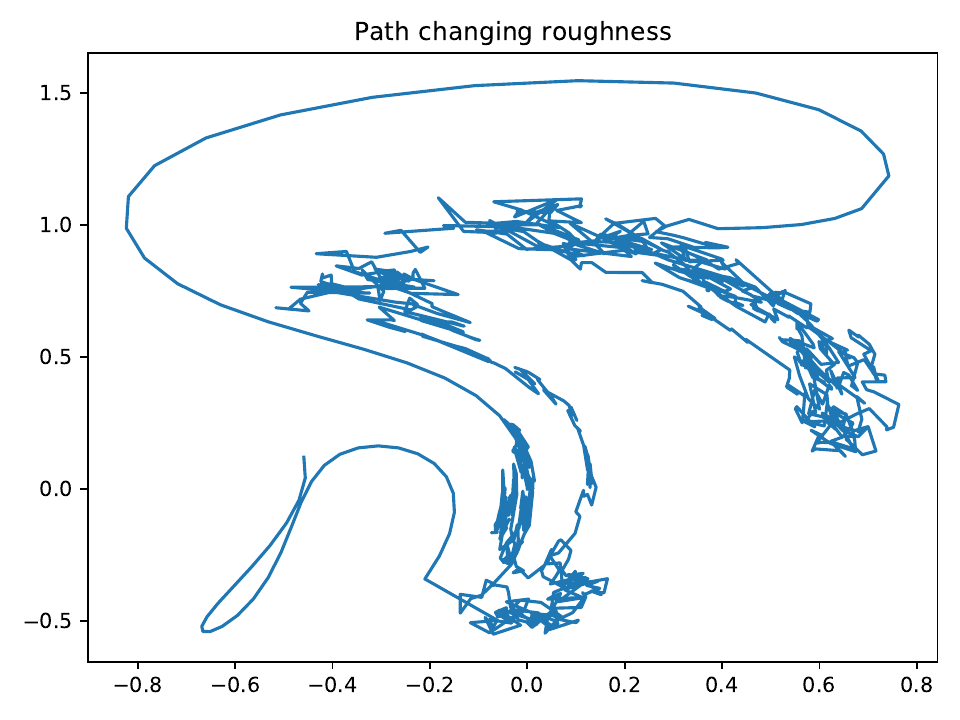}
\caption{The solution path of the example ``Path changing roughness''.}
\label{fig:PathChangingRoughnessSolution}
\end{figure}

\begin{figure}
\centering
\begin{minipage}{.5\textwidth}
  \centering
  \includegraphics[width=\linewidth]{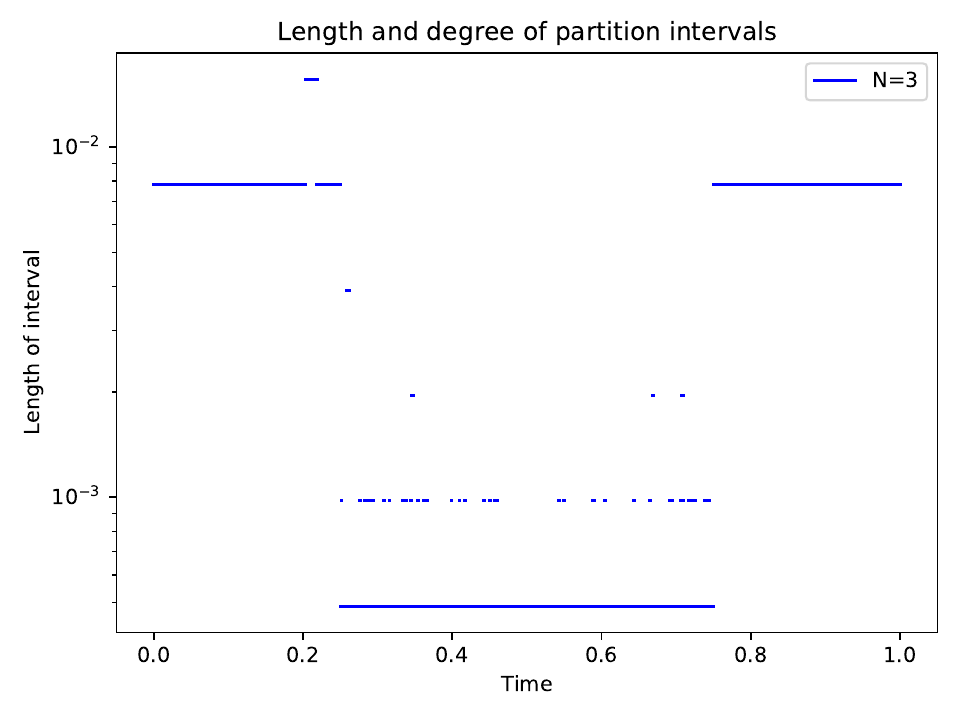}
\end{minipage}%
\begin{minipage}{.5\textwidth}
  \centering
  \includegraphics[width=\linewidth]{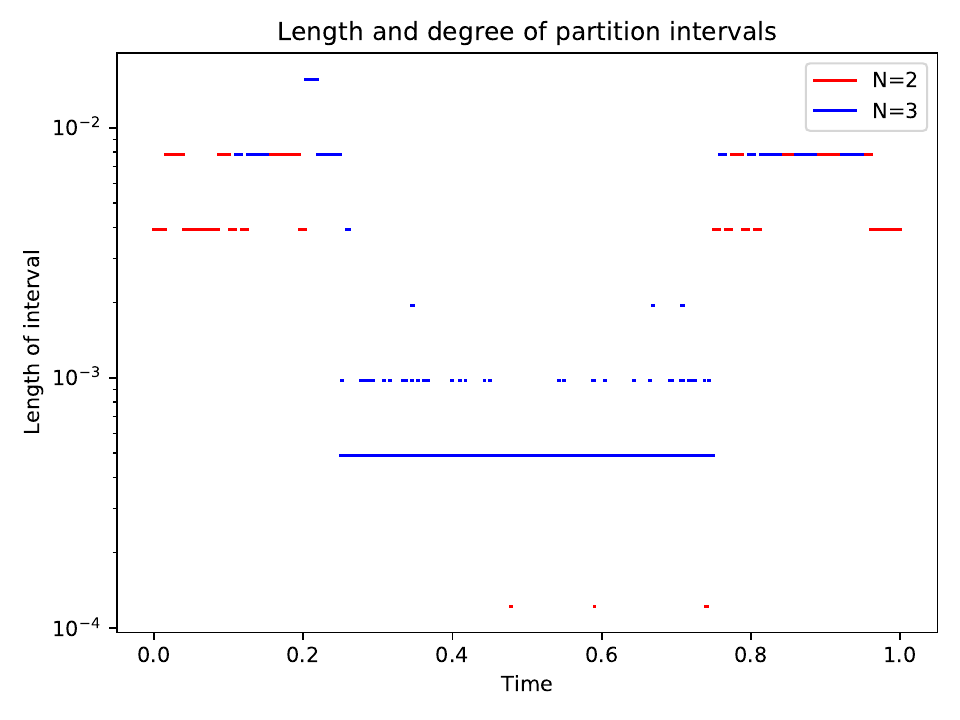}
\end{minipage}
\caption{The lengths of the partition intervals together with the degrees of the Log-ODE method for the example ``Path changing roughness''. The left plot corresponds to ``ER predicting'', the right plot to ``ER testing''.}
\label{fig:PathChangingRoughnessIntervals}
\end{figure}

\begin{table}[!htbp]
\centering
\resizebox{\textwidth}{!}{\begin{tabular}{c|c|c|c|c}
 & ER predicting & ER testing & Simple first level & Simple full solution\\ \hline
Error & $1.51\cdot 10^{-4}$ & $2.20\cdot 10^{-4}$ & $1.48\cdot 10^{-4}$ & $1.48\cdot 10^{-4}$\\
Estimated error & $1.40\cdot 10^{-4}$ & $2.08\cdot 10^{-4}$ & - & -\\
Error after correction & $1.17\cdot 10^{-5}$ & $1.25\cdot 10^{-5}$ & - & -\\
Degree 2 intervals & 0 & 87 & 131072 & 131072\\
Degree 3 intervals & 1027 & 987 & 0 & 0\\
Runtime (s) & 1022 & 2441 & 1926 & 5795
\end{tabular}}
\caption{Errors, intervals, and runtime of the various algorithms for the example ``Path changing roughness''.}
\label{tab:PathChangingRoughnessData}
\end{table}

\subsection{Underdamped Langevin equation}

As a final example, we consider the Langevin equation 
\begin{align*}
\dd Q_t &= P_t \sdd t,\\
\dd P_t &= -\nabla U(Q_t)\sdd t - \nu P_t \sdd t + \sqrt{\frac{2\nu}{\beta}}\sdd W_t.
\end{align*}

Motivated by \cite[Section 5.3]{foster2020numerical}, we we choose $U(q) = (q^2-1)^2$, $\nu=1$, $\beta=3$, $Q_0=0$, $P_0=0$, and the time horizon $T=1000$. Furthermore, suppose we are mainly interested in the position of the particle $Q$ at the final point. We hence choose the payoff function $g(q, p) = q$, and the absolute error tolerance $10^{-6}$. The driving path $(t, W_t)$ is given as a discretized time-enhanced Brownian motion where we use $1048576$ time steps.

The solution $Q$ is shown in Figure \ref{fig:UnderdampedLangevinEquationSolution}, and the lengths and degrees of the intervals chosen by the adaptive algorithms is in Figure \ref{fig:UnderdampedLangevinEquationIntervals}. We see that the algorithm only refines towards the end of the interval $[0, T]$. This is explained by the ergodic nature of the dynamical system, and since we are merely interested in the position of $Q$ at the final time $T$. 

Finally, in Table \ref{tab:UnderdampedLangevinEquationData}, we can find a comparison of the different algorithms. We see that the error representation formula yielded very accurate results. Furthermore, the algorithms using the error representation formula used substantially fewer intervals, and were faster.

\begin{figure}
\centering
\includegraphics[scale=0.6]{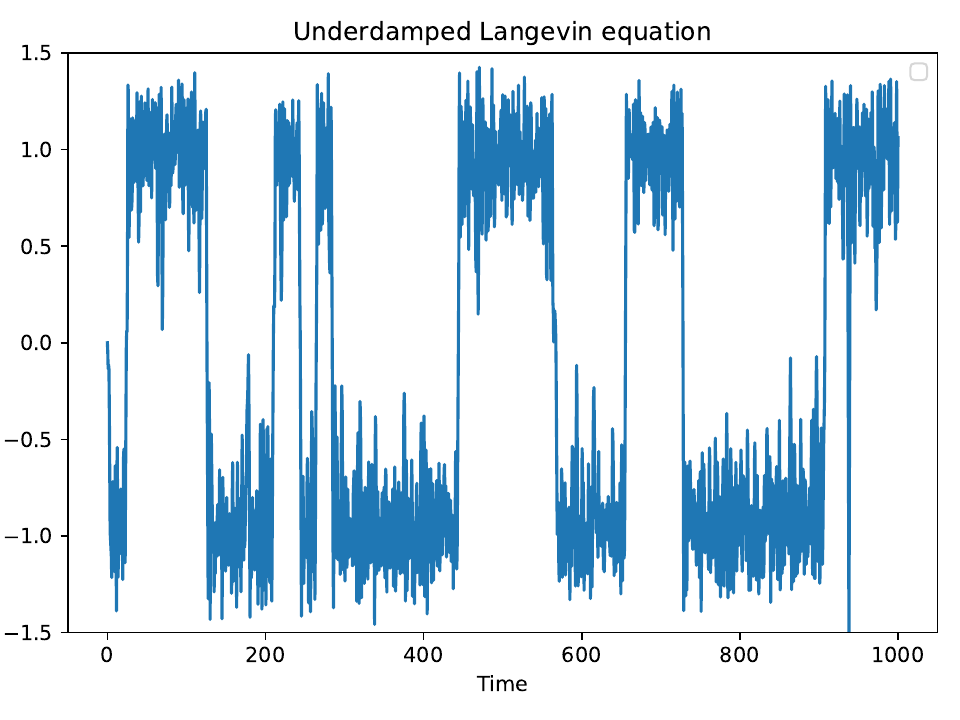}
\caption{The solution path of the example ``Underdamped Langevin equation''.}
\label{fig:UnderdampedLangevinEquationSolution}
\end{figure}

\begin{figure}
\centering
\begin{minipage}{.5\textwidth}
  \centering
  \includegraphics[width=\linewidth]{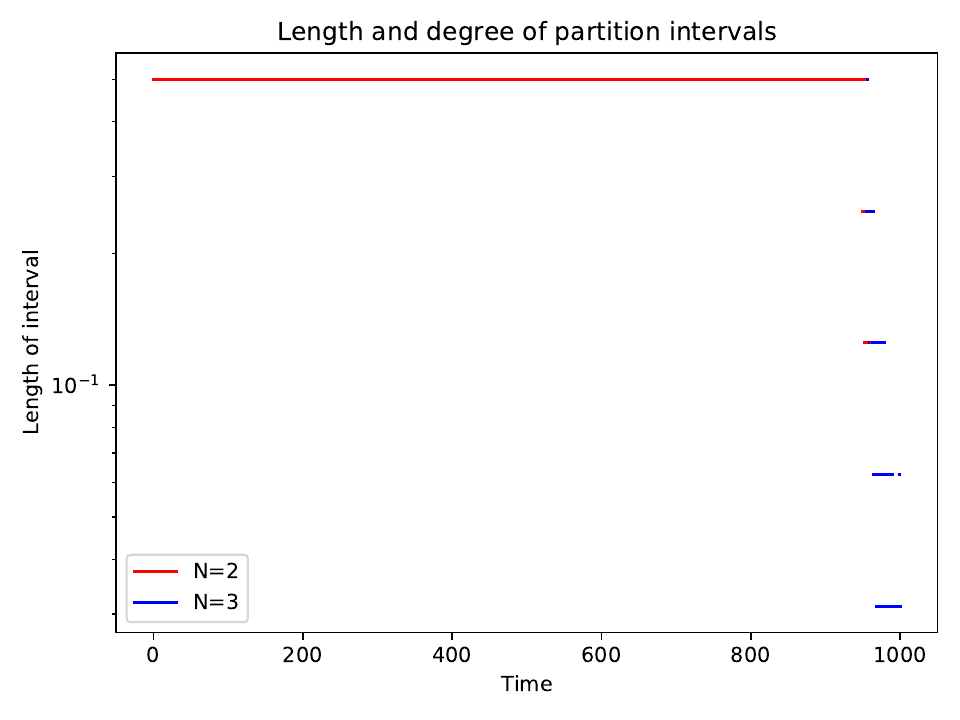}
\end{minipage}%
\begin{minipage}{.5\textwidth}
  \centering
  \includegraphics[width=\linewidth]{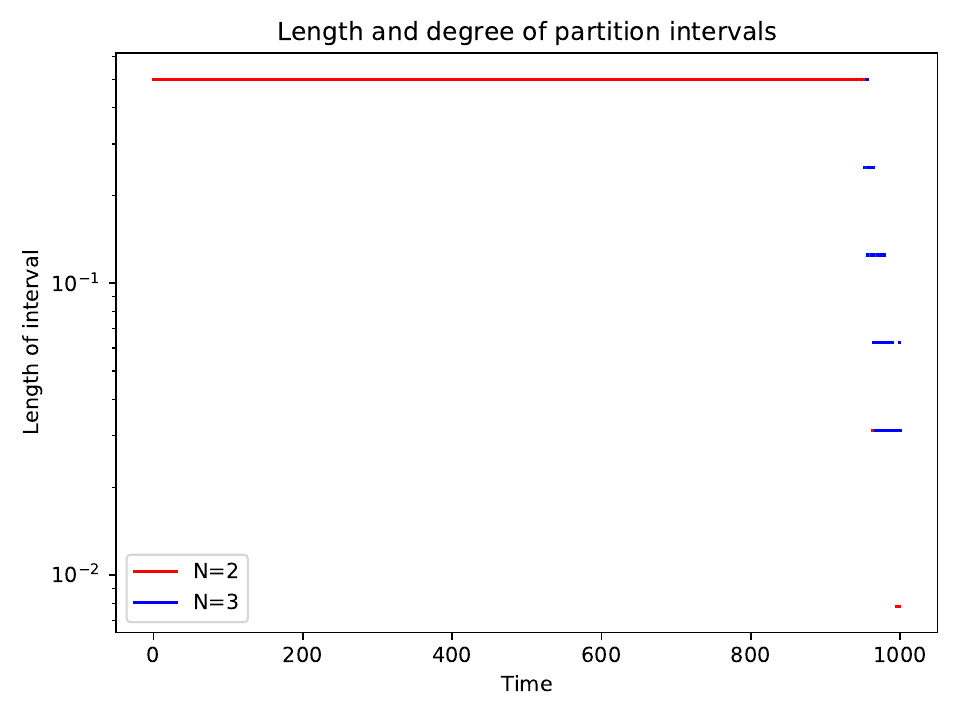}
\end{minipage}
\caption{The lengths of the partition intervals together with the degrees of the Log-ODE method for the example ``Underdamped Langevin equation''. The left plot corresponds to ``ER predicting'', the right plot to ``ER testing''.}
\label{fig:UnderdampedLangevinEquationIntervals}
\end{figure}

\begin{table}[!htbp]
\centering
\resizebox{\textwidth}{!}{\begin{tabular}{c|c|c|c|c}
 & ER predicting & ER testing & Simple first level & Simple full solution\\ \hline
Error & $9.25\cdot 10^{-7}$ & $1.37\cdot 10^{-7}$ & $1.16\cdot 10^{-6}$ & $1.16\cdot 10^{-6}$\\
Estimated error & $9.27\cdot 10^{-7}$ & $1.39\cdot 10^{-7}$ & - & -\\
Error after correction & $1.78\cdot 10^{-9}$ & $1.18\cdot 10^{-9}$ & - & -\\
Degree 2 intervals & 1932 & 1926 & 256000 & 256000\\
Degree 3 intervals & 903 & 914 & 0 & 0\\
Runtime & 1110 & 3086 & 3289 & 9571
\end{tabular}}
\caption{Errors, intervals, and runtime of the various algorithms for the example ``Underdamped Langevin equation''.}
\label{tab:UnderdampedLangevinEquationData}
\end{table}

\appendix

\section{Notation and Definitions}\label{sec:Notation}

\subsection{General notation}

Given a real number $x$, we denote by $[x]$ the largest integer smaller than or equal to $x$, and by $\lfloor x\rfloor$ the largest integer strictly smaller than $x$. For example, $[2.6] = \lfloor 2.6\rfloor = 2$, but $[2] = 2$ while $\lfloor 2\rfloor = 1$. Furthermore, $\{x\}\coloneqq x - \lfloor x\rfloor\in(0,1].$

For a normed space $U$, and a real number $r\ge 0$, we denote by $B_r$ the closed ball centered at $0$ with radius $r$, i.e. $$B_r\coloneqq \left\{x\in U\colon \|x\|\le r\right\}.$$

For two topological vector spaces $U, V$, we denote by $L(U, V)$ the set of continuous linear maps from $U$ to $V$.

For a metric space $U$ and a Banach space $V$, we denote by $C_b(U,V)$ the Banach space of continuous bounded functions, together with the norm $$\|f\|_\infty \coloneqq \sup_{x\in U}\|f(x)\|,\qquad f\in C_b(U,V).$$ For two finite-dimensional Banach spaces $U,V$, we denote by $C(U,V)$ the set of continuous maps from $U$ to $V$. The elements in $C(U,V)$ are not assumed to be bounded. Hence, $C(U,V)$ is merely a locally convex space, when equipped with the seminorms $$\|f\|_K\coloneqq \|f|_K\|_\infty,\qquad f\in C(U,V)$$ for all compact sets $K\subset U$.

\subsection{Tensor algebras}

For a vector space $V$, we denote by $T(V)$ the tensor algebra $$T(V)\coloneqq \R \oplus V \oplus V^{\otimes 2}\oplus \dots,$$ and by $T^N(V)$ the truncated tensor algebra $$T^N(V)\coloneqq \bigoplus_{n=0}^N V^{\otimes n}.$$ We denote by $\pi_n:T(V)\to V^{\otimes n}$ the projection onto the $n$-th element, and for $m\le n$ we denote by $\pi_{m,n}:T(V)\to \bigoplus_{k=m}^n V^{\otimes k}$ the projection onto the levels $m$ to $n$. Elements in $T(V)$ or $T^N(V)$ are denoted by bold letters, while their projections onto the first component are denoted by the associated normal letters. For example, if $\bm{v}\in T^N(V)$, then we write $v = \pi_1(\bm{v})\in V.$ We may also write $v_k \coloneqq \pi_k(\bm{v})$ if no confusion arises.

We define the tensor product $\otimes:T(V)\times T(V)\to T(V)$ by $$\pi_n(\bm{a}\otimes\bm{b}) = \sum_{k=0}^n \pi_k(\bm{a})\otimes \pi_{n-k}(\bm{b}),\qquad \bm{a},\bm{b}\in T(V).$$ With component-wise addition and tensor multiplication, $T(V)$ (and, by projection, also $T^N(V)$) become algebras. If $\pi_0(\bm{a}) \neq 0$, then $\bm{a}\in T(V)$ has an inverse $\bm{a}^{-1}$ given by $$\bm{a}^{-1} = \frac{1}{\pi_0(\bm{a})} \sum_{n=0}^\infty \left(\bm{1} - \frac{\bm{a}}{\pi_0(\bm{a})}\right)^n.$$ Here, $\bm{1} = (1,0,\dots,0)\in G^N(V)$ denotes the unit element with respect to the tensor product.

We further define the set $$T_1^N(V) \coloneqq \left\{\bm{v}\in T^N(V)\colon \pi_0(\bm{v}) = 1\right\}.$$ The set $T_1^N(V)$ is an affine vector subspace of $T^N(V)$ and will be seen as a vector space in its own right. Addition and scalar multiplication in $T^N_1(V)$ are of course defined as in $T^N(V)$, except that the $0$th component is left invariant. We denote by $G(V)$ the free Lie group over $V$, and by $G^N(V)$ the free Lie group truncated at level $N$. In particular, we have $$G^N(V)\subseteq T_1^N(V)\subseteq T^N(V).$$

If $\bz\in T(\R^{d+e})$, $\bx\in T(\R^d)$ and $\by\in T(\R^e)$ are such that the projection of $\bz$ onto the first $d$ components (i.e. onto $T(\R^d)$) is $\bx$, and the projection of $\bz$ onto the last $e$ components is $\by$, then we write $\bz = (\bx,\by)$. Similar notations apply for $T^N$, $T_1^N$ and $G^N$. Note that $\bx$ and $\by$ do not usually determine $\bz$.

Next, given a norm on the vector space $V$, we want to extend this norm to its tensor powers $V^{\otimes n}$.

For all $n\ge 1$, we denote by $S_n$ the group of permutations of $\{1,\dots,n\}$. 

Given a vector space $V$, an integer $n\ge 1$, and an element $v = v_1\otimes\dots\otimes v_n\in V^{\otimes n}$, we define $$\sigma v \coloneqq v_{\sigma(1)}\otimes \dots\otimes v_{\sigma(n)}$$ for all $\sigma\in S_n$. We extend this definition to arbitrary $v\in V^{\otimes n}$ by linearity.

\begin{definition}\cite[Definition 1.25]{lyons2007differential}
Let $V$ be a normed space. We say that its tensor powers are endowed with admissible norms if the following conditions hold.
\begin{enumerate}
\item For each $n\ge 1$, the symmetric group $S_n$ acts by isometries on $V^{\otimes n}$, i.e. $$\|\sigma v\| = \|v\|,\qquad \forall v\in V^{\otimes n},\ \forall \sigma\in S_n.$$
\item The tensor product has norm $1$, i.e. for all $n,m\ge 1$, $$\|v\otimes w\| \le \|v\|\|w\|,\qquad \forall v\in V^{\otimes n},\ w\in V^{\otimes m}.$$
\end{enumerate}
\end{definition}

From now on we assume that the tensor powers will always be endowed with admissible norms. We can then define two different norms on $T_1^N(V)$.

\begin{definition}
Let $V$ be a normed space. Then, we define the inhomogeneous norm $\|.\|$ on $T_1^N(V)$ by $$\|\bm{v}\| \coloneqq \max_{i=1,\dots,N} \|\pi_i(\bm{v})\|,\qquad \forall \bm{v}\in T_1^N(V).$$ Moreover, we define the homogeneous norm $\vertiii{.}$ on $T_1^N(V)$ by $$\vertiii{\bm{v}} \coloneqq \max_{i=1,\dots,N} \|\pi_i(\bm{v})\|^{1/i},\qquad \forall \bm{v}\in T_1^N(V).$$
\end{definition}

We remark that we will cite some theorems in \cite{friz2010multidimensional} where different definitions of the homogeneous and inhomogeneous norms may be used. However, if $V$ is finite-dimensional, all homogeneous and all inhomogeneous norms are equivalent (see also \cite[Theorem 7.44]{friz2010multidimensional}), so we will not emphasize that point further.

\subsection{Rough paths}

\begin{definition}
Let $x$ be a path of finite variation in a Banach space $V$. Then, we define the signature $S_N(x)_{s,t}$ of level $N$ of $x$ on the interval $[s,t]$ by $$S_N(x)_{s,t} \coloneqq \left(1,\int_s^t \dd x_u,\dots,\int_{s<u_1<\dots<u_N<t}\dd x_{u_1}\otimes\dots\otimes \dd x_{u_N}\right)\in G^N(V).$$
\end{definition}

\begin{definition}
Let $\bx\in C([0,T],T_1^N(V))$ for some Banach space $V$. For $s\le t$, we define the increment $\bx_{s,t}\in T_1^N(V)$ by $$\bx_{s,t} \coloneqq \bx_s^{-1}\otimes \bx_t.$$ Moreover, we define the $p$-variation of $\bx$ by $$\|\bx\|_{p\textup{-var};[0,T]} \coloneqq \sup_{(t_i)}\left(\sum_i \vertiii{\bx_{t_i,t_{i+1}}}^p\right)^{1/p},$$ where the supremum is taken over all finite partitions $(t_i)$ of $[0,T]$, and $\vertiii{.}$ is the homogeneous norm on $T_1^N(V)$. We denote by $C^{p\var}([0,T],T_1^N(V))$ the set of such $\bx$ of finite $p$-variation. If $\bx$ takes values in $G^N(V)$, we also write $\bx\in C^{p\var}([0,T],G^N(V)).$
\end{definition}

\begin{definition}\cite[cf. Definition 9.15]{friz2010multidimensional}
Let $V$ be a finite-dimensional Banach space, let $p \ge 1$, and let $\bx\in C^{p\var}([0,T],G^{[p]}(V))$. Then, $\bx$ is called a weakly geometric $p$-rough path. If, moreover, there exists a sequence of continuous finite variation paths $(x^n)$ taking values in $V$ such that $S_{[p]}(x^n)_{0, .}\to \bx$ in $p$-variation, then we say that $\bx$ is a geometric $p$-rough path. We denote the set of weakly geometric $p$-rough paths by $C^{p\var}([0,T],V)$ and the set of geometric $p$-rough paths by $C^{p\var,0}([0,T],V)$.
\end{definition}

\begin{remark}
\begin{enumerate}
\item By definition, if $\bx$ is a geometric $p$-rough path, then $\bx_0 = \bm{1}$. This is not necessarily the case for weakly geometric $p$-rough paths. We denote by $C^{p\var}_1([0,T],V)$ the set of $\bx\in C^{p\var}([0,T],V)$ with $\bx_0 = \bm{1}$.
\item By \cite[Proposition 8.12]{friz2010multidimensional}, if $\bx\in C^{p\var}([0,T],V)$, then there exists a sequence $(x^n)$ of Lipschitz continuous paths (or, equivalently, a sequence of finite variation paths) such that
\begin{equation}\label{eqn:ConvergenceToWeakGRP}
\|S_{[p]}(x^n)_{0,.} - \bx_{0,.}\|_\infty\to 0,\qquad \sup_n\|S_{[p]}(x^n)\|_{p\var;[0,T]} < \infty.
\end{equation}
By \cite[Proposition 8.15]{friz2010multidimensional}, we can also find Lipschitz continuous paths such that $$\sup_{0\le s\le t\le T}\|S_{[p]}(x^n)_{s,t} - \bx_{s,t}\|\to 0,\qquad \sup_n\|S_{[p]}(x^n)\|_{p\var;[0,T]} < \infty.$$
\item By interpolation, we can see that $$C^{(p-\eps)\var}_1([0,T],V)\subset C^{p\var,0}([0,T],V) \subset C^{p\var}_1([0,T],V),$$ and all of these inclusions are non-trivial.
\end{enumerate}
\end{remark}

\begin{theorem}[Lyons' extension theorem]\cite[Theorem 9.5]{friz2010multidimensional}
Let $N\ge[p]\ge 1$, let $V$ be a finite-dimensional Banach space, and let $\bx\in C^{p\var}_1([0,T],V)$. Then, there exists a unique path $S_N(\bx)\in C^{p\var}([0,T],G^N(V))$ with $S_N(\bx) = \bm{1}$ and $\pi_{0,[p]}(S_N(\bx)) = \bx$. Moreover, there exists a constant $C = C(N,p)$ such that $$\|S_N(\bx)\|_{p\textup{-var}} \le C\|\bx\|_{p\textup{-var}}.$$ The path $S_N(\bx)$ is called the Lyons lift of $\bx$.
\end{theorem}

The notation might seem slightly ambiguous, as $S_N$ can be both applied to a finite variation path $x$ and the associated geometric $1$-rough path $\bx$ (which is given by $\bx_t = S_1(x)_{0,t} = (1,x_t-x_0).$) However, we have the simple relation $$S_N(\bx)_t = S_N(x)_{0,t}.$$

\begin{definition}
A control function is a function $$\omega:\{(s,t)\in[0,T]^2\colon s\le t\}\to[0,\infty),$$ that is continuous, $0$ on the diagonal, and super-additive, i.e. for all $0\le s\le t\le u\le T$, we have $$\omega(s,t) + \omega(t,u) \le \omega(s,u).$$
\end{definition}

Let $\bx\in C^{p\var}([0,T],T_1^N(V))$. By \cite[Proposition 5.8]{friz2010multidimensional}, $\omega(s,t)\coloneqq \|\bx\|_{p\textup{-var};[s,t]}^p$ is a control function.

\begin{definition}
Let $\bx\in C^{p\var}([0,T],T_1^N(V))$ with $N\ge [p]\ge 1$ and let $\omega$ be a control. We say that $\bx$ is controlled by $\omega$ if for all $0\le s \le t\le T$, we have $$\vertiii{\bx_{s,t}} \le \omega(s,t)^{1/p}.$$
\end{definition}

Since controls are super-additive, we see that if $\bx\in C^{p\var}([0,T],T_1^N(V))$ is controlled by $\omega$, then $$\|\bx\|_{p\textup{-var};[s,t]}\le \omega(s,t)^{1/p}$$ for all $0\le s\le t\le T$. Then, by Lyons' extension theorem, if $\bx\in C^{p\var}([0,T],V)$ is controlled by $\omega$, then $S_N(\bx)$ is controlled by $C\omega$ for some constant $C = C(N,p)$.

\begin{definition}\cite[cf. Definition 8.6]{friz2010multidimensional}
Let $\bx,\by\in C^{p\var}([0,T],T_1^N(V))$, where $N\ge [p]\ge 1$. Then, we define the (inhomogeneous) variation distance $$\rho_{p\textup{-var};[0,T]}(\bx,\by) \coloneqq \sup_{(t_i)}\max_{k=1,\dots,N}\left(\sum_i\|\pi_k(\bx_{t_i,t_{i+1}} - \by_{t_i,t_{i+1}})\|^{p/k}\right)^{k/p},$$ where the supremum is taken over all finite partitions $(t_i)$ of $[0,T]$. Moreover, if $\omega$ is a control, we define the (inhomogeneous) $p\textup{-}\omega$ distance by $$\rho_{p,\omega;[0,T]}(\bx,\by) \coloneqq \sup_{0\le s\le t\le T}\max_{k=1,\dots,N}\frac{\|\pi_k(\bx_{s,t}-\by_{s,t})\|}{\omega(s,t)^{k/p}}.$$
\end{definition}

We remark that convergence in $p$-variation and convergence in $\rho_{p\textup{-var}}$ are equivalent by \cite[Theorem 8.10]{friz2010multidimensional}.

We also remark that for $\bx\in C^{p\var}_1([0,T],V)$, the convergence $S_{[p]}(x^n)\to\bx$ in $p$-variation implies the convergence $S_N(x^n)\to S_N(\bx)$ in $p$-variation by the equivalence of convergence in $p$-variation and in the $\rho_{p\textup{-var}}$-metric combined with \cite[Theorem 9.10]{friz2010multidimensional}.

\subsection{Lipschitz continuous vector fields}

\begin{definition}
Given two vector spaces $U,V,$ and $n\ge 1$, a map $A:U^{\otimes n}\to V$ is said to be a symmetric $n$-linear map if $A$ is linear and if for all $\sigma\in S_n$ and $u\in U^{\otimes n}$, we have $$A(\sigma u) = Au.$$
\end{definition}

We give the definition of Lipschitz functions in the sense of Stein.

\begin{definition}\cite[Definition 1.21]{lyons2007differential}
Let $U,V$ be two finite-dimensional Banach spaces, let $F\subseteq U$ be a closed set, and let $\gamma > 0$. Let $f^0:F\to V$ be a function, and, for $j=1,\dots,\lfloor\gamma\rfloor$, let $f^j:F\to L(U^{\otimes j},V)$ be functions taking values in the symmetric $j$-linear mappings from $U$ to $V$. The collection $f = (f^j)_{j=0}^{\lfloor\gamma\rfloor}$ is an element of $\lip^\gamma(F,V)$ if the following condition holds.

There exists a constant $M$ such that, for all $j=0,\dots,\lfloor\gamma\rfloor$, $$\sup_{x\in F}\|f^j(x)\| \le M,$$ and if there exists a function $R_j:F\times F\to L(U^{\otimes j},V)$ such that for all $x,y\in F$ and $v\in U^{\otimes j}$, $$f^j(y)(v) = \sum_{l=0}^{\lfloor\gamma\rfloor-j}\frac{1}{l!}f^{j+l}(x)(v\otimes (y-x)^{\otimes l}) + R_j(x,y)(v),$$ and $$\|R_j(x,y)\| \le M\|y-x\|^{\gamma-j}.$$ The smallest constant $M$ for which these inequalities hold is called the $\lip^\gamma(F,V)$-norm of $f$, and is denoted by $\|f\|_{\lip^\gamma}.$
\end{definition}

By the discussion in \cite[Section VI.2]{stein2016singular}, it is clear that if $f\in \lip^\gamma(F,V)$, and if $G\subseteq F$ is an open subset of $F$, then $f$ is $\lfloor\gamma\rfloor$ times differentiable in $G$ and the $\lfloor\gamma\rfloor$th derivative is $\{\gamma\}$-Hölder continuous. Moreover, we have $f^{(j)} = f^j$ on $G$ for $j=0,\dots,\lfloor\gamma\rfloor$, where $f^{(j)}$ denotes the $j$th derivative. In particular, if $f\in \lip^\gamma(U,V)$, we may identify $f$ with $f^0$.

\begin{definition}
Let $U,V$ be two finite-dimensional Banach spaces, let $\gamma > 0$ and let $f:U\to V$ be $\lfloor\gamma\rfloor$ times differentiable. We say that $f\in \lip^\gamma_{\textup{loc}}(U,V)$ if for all compact subsets $K\subseteq U$ the restriction $g_K \coloneqq (f|_K,f'|_K,\dots,f^{(\lfloor\gamma\rfloor)}|_K)\in \lip^\gamma(K,V).$ The space $\lip^\gamma_{\textup{loc}}(U,V)$ becomes a locally convex space with the semi-norms $$\|f\|_K\coloneqq \|g_K\|_{\lip^\gamma}$$ for all compact $K\subseteq U$.
\end{definition}

\begin{theorem}\label{thm:WhitneysTheoremStein}\cite[Chapter VI.2]{stein2016singular}
Let $U, V$ be finite-dimensional Banach spaces, let $F\subseteq U$ be a closed set, and let $f\in \lip^\gamma(F,V)$. Then there exists an extension $\widetilde{f}\in \lip^\gamma(U,V)$ with $\widetilde{f}|_F = f$, and a constant $C$ independent of $f$ and $F$ such that $$\|\widetilde{f}\|_{\lip^\gamma} \le C\|f\|_{\lip^\gamma}.$$
\end{theorem}

\subsection{RDE solutions}\label{sec:RDESolutions}

Given two Banach spaces $U,V$, and $\gamma >0$, we define the spaces of vector fields
\begin{align*}
\mathcal{V}(V,U) &\coloneqq L(V,C(U,U)),\\
\mathcal{V}^\gamma(V,U) &\coloneqq L(V,\lip^\gamma_\loc(U,U)),\\
\mathcal{V}^{\gamma+}(V,U) &\coloneqq \bigcup_{\eps > 0} \mathcal{V}^{\gamma+\eps}(V,U).
\end{align*}

We recall the definition of an RDE solution.

\begin{definition}\cite[Definition 10.17]{friz2010multidimensional}
Let $\bx\in C^{p\var}([0,T],\R^d)$, let $f\in\mathcal{V}(\R^d,\R^e)$, and let $y_0\in \R^e$. We call a path $y\in C([0,T],\R^e)$ a solution to the rough differential equation
\begin{align}
\dd y_t = f(y_t) \sdd\bx_t,\qquad y_0 = y_0,\label{eqn:FirstLevelRDE1}
\end{align}
if there exists a sequence of finite variation paths $(x^n)$ satisfying \eqref{eqn:ConvergenceToWeakGRP}, such that there exist solutions $y^n$ to the ODEs $$\dd y^n_t = f(y^n_t)\sdd x^n_t,\qquad y^n_0 = y_0$$ with $$\|y^n-y\|_\infty\to 0.$$
\end{definition}

As the following lemma demonstrates, is immediately clear that ODE solutions are RDE solutions.

\begin{lemma}\label{lem:ODESolutionsAreRDESolutions}
Let $x\in C([0,T],\R^d)$ be of finite variation, let $f\in\mathcal{V}(\R^d,\R^e)$, and let $y_0\in\R^e$. We can naturally associate to $x$ the geometric $p$-rough path $\bx = S_{[p]}(x)$. Let $y\in C([0,T],\R^e)$ be a solution to the ODE $$\dd y_t = f(y_t) \sdd x_t,\qquad y_0 = y_0.$$ Then $y$ is a solution to the RDE \eqref{eqn:FirstLevelRDE1}.
\end{lemma}

\begin{proof}
This follows immediately from the definition of RDE solutions by taking the sequence $(x^n)$ with $x^n = x$, and the solutions $y^n = y$.
\end{proof}

The converse direction is not always true, i.e. there may exist RDE solutions that are not ODE solutions. Below, we give such an example.

\begin{example}\label{ex:CounterExample}
For simplicity of exposition, we work in $\C = \R^2$ here, and consider integrals instead of differential equations. We will see in Section \ref{sec:RoughIntegrals} that integrals are in fact a special case of differential equations.

Consider the trivial finite variation path $x:[0,1]\to \R^2$ given by $x_t = 0$, and, for $\gamma>1$ the vector field $f(z) = \frac{\abs{z}^\gamma}{z}\in \lip^{\gamma-1}$. Clearly, $$\int_0^t f(x_t) \sdd x_t = 0$$ in the ODE (or rather, Riemann-Stieltjes) sense. However, interpreting the above as a rough integral, we can choose the paths $x^n$ as $n$ times the circle with radius $n^{-1/p}$. For $p\in(1,2)$, $\|S_1(x^n)\|_{p-\textup{var}}$ stays bounded and $(x^n)$ converges uniformly to $x$. However, the corresponding ODE solutions converge uniformly to $2\pi i t$, yielding a second RDE solution.
\end{example}

We show the following consistency relation of RDEs.

\begin{lemma}\label{lem:RoughPathDifferentPConsistency}
Let $\bx\in C^{p\var}([0,T],\R^d)$, let $f\in\mathcal{V}(\R^d,\R^e)$, and let $y_0\in\R^e$. Let $q\ge p \ge 1$, and let $\widetilde{\bx}\coloneqq S_{[q]}(\bx)$. Then, $\widetilde{\bx}\in C^{q\var}([0,T],\R^d)$ with $\|\widetilde{\bx}\|_{q\var}\le C\|\bx\|_{p\var}$ for some $C=C(p,q)$, and every solution $y$ to the RDE $$\dd y_t = f(y_t) \sdd\bx_t,\qquad y_0 = y_0$$ is a solution to the RDE $$\dd y_t = f(y_t) \sdd\widetilde{\bx}_t,\qquad y_0 = y_0.$$
\end{lemma}

\begin{remark}
It is not clear to the authors whether the above RDEs are equivalent, i.e. whether every solution of the RDE w.r.t. $\widetilde{\bx}$ is a solution w.r.t. $\bx$.
\end{remark}

\begin{proof}
If $q=p$, the statement is trivial. Assume that $q > p$.

First, we note that $\widetilde{\bx}$ is indeed a weakly geometric $q$-rough path. It obviously takes values in $G^{[q]}(\R^d)$, so the only thing to show is that it is of finite $q$-variation. This follows by using \cite[Proposition 5.3]{friz2010multidimensional} and \cite[Theorem 9.5]{friz2010multidimensional} from $$\|\widetilde{\bx}\|_{q\var} \le \|S_{[q]}(\bx)\|_{p\var} \le C\|\bx\|_{p\var},$$ where we additionally remark that $C$ depends only on $p,q$. 

Let $y$ be a solution to the RDE driven by $\bx$. By definition, there exists a sequence of finite variation paths $(x^n)$ with $$\|S_{[p]}(x^n) - \bx\|_\infty\to 0,\qquad \sup_n\|S_{[p]}(x^n)\|_{p\var} < \infty,$$ and ODE solutions $y^n$ with $$\|y^n-y\|_\infty \to 0.$$ Let $r\in (p,[p]+1)$. By \cite[Proposition 5.5]{friz2010multidimensional}, $S_{[r]}(x^n) = S_{[p]}(x^n)$ converges to $\bx$ in $r$-variation. Then,
\begin{align*}
\|S_{[q]}(x^n) - \widetilde{\bx}\|_\infty &\le \|S_{[q]}(S_{[p]}(x^n)) - S_{[q]}(\bx)\|_{r\var} \to 0
\end{align*}
by \cite[Corollary 9.11]{friz2010multidimensional}. Moreover, $$\sup_n\|S_{[q]}(x^n)\|_{q\var} \le \sup_n\|S_{[q]}(x^n)\|_{p\var} \le C\sup_n\|S_{[p]}(x^n)\|_{p\var} < \infty$$ as before. By the definition of an RDE solution for the driving path $\widetilde{\bx}$, we conclude that $y$ is a solution.
\end{proof}

We now recall the following definition of full solutions to full RDEs.

\begin{definition}\cite[Definition 10.34]{friz2010multidimensional}
Let $\bx\in C^{p\var}([0,T],\R^d)$, let $f\in\mathcal{V}(\R^d,\R^e)$, and let $\by_0\in G^{[p]}(\R^e)$. We call a path $\by\in C([0,T],G^{[p]}(\R^e))$ a solution to the full RDE $$\dd\by_t = f(y_t) \sdd\bx_t,\qquad \by_0 = \by_0,$$ if there exists a sequence of finite variation paths $(x^n)$ satisfying \eqref{eqn:ConvergenceToWeakGRP}, and such that there exist solutions $y^n$ to the ODEs $$\dd y^n_t = f(y^n_t) \sdd x^n_t,\qquad y^n_0 = \pi_1(\by_0) = y_0$$ with $$\|\by_0\otimes S_{[p]}(y^n)_{0, .} - \by_.\|_\infty\to 0.$$
\end{definition}

Of course, if $\by$ is a solution to the full RDE above, then $\by\in C^{p\var}([0,T],\R^e)$ is again a weakly geometric $p$-rough path.

\begin{definition}\label{def:FullVectorField}
For every $f\in\mathcal{V}(\R^d,\R^e)$ and for all $N\ge 1$, we define the associated full vector field $\bm{f}\in\mathcal{V}(\R^d,T_1^N(\R^e))$ by $$\bm{f}(\bz) \coloneqq \bz \otimes f(\pi_1(\bz)).$$
\end{definition}

We recall that by \cite[Theorem 10.35]{friz2010multidimensional}, a full RDE solution with respect to the vector field $f$ is just a first-level RDE solution with respect to the vector field $\bm{f}$. (As a minor side remark, the vector field in \cite[Theorem 10.35]{friz2010multidimensional} actually acts on $T^N(\R^e)$, not the smaller space $T_1^N(\R^e)$. However, since full RDE solutions always take values in $G^N(\R^e)$, which is a subset of $T_1^N(\R^e)$, this does not change the situation. The restriction to $T_1^N(\R^e)$ will become relevant in the forthcoming Lemma \ref{lem:NonExplosionExtendsToFullRDEs}.)

\subsection{Rough integrals}\label{sec:RoughIntegrals}

First, we recall the definition of a rough integral.

\begin{definition}\cite[cf. Definition 10.44]{friz2010multidimensional}
Let $\bx\in C^{p\var}([0,T],\R^d)$, and let $f\in L(\R^d,C(\R^d,\R^e))$. We say that $\by\in C([0,T],G^{[p]}(\R^e))$ is a rough integral of $f$ along $\bx$, if there exists a sequence $(x^n)$ of finite variation paths with $x^n_0 = \pi_1(\bx_0)$ satisfying \eqref{eqn:ConvergenceToWeakGRP}, and $$\left\|S_{[p]}\left(\int_0^. f(x^n_u)dx^n_u\right)_{0, .} - \by_.\right\|_\infty\to 0.$$ If $\by$ is the unique rough integral, we write $$\int f(x) \sdd\bx \coloneqq \by.$$
\end{definition}

As remarked in \cite[Section 10.6]{friz2010multidimensional}, the rough integral $\int f(x) d\bx$ is in fact given as the projection onto the last $e$ coordinates of a solution to the full RDE $$d\bz_t = f_1(z_t)d\bx_t,\qquad \bz_0 = \bz_0,$$ where $f_1\in\mathcal{V}(\R^d,\R^{d+e})$, $$f_1(x,y) = 
\begin{pmatrix}
\id\\
f(x)
\end{pmatrix}
,$$
and where $\bz_0$ is any element with $\bz_0=(\bx_0,\bm{1}).$

\section{A cost model for the Log-ODE method}\label{sec:CostModel}

If we apply the error representation formula for the Log-ODE method on a given partition, we know which intervals of the partition contribute most to the global error. To compute the solution more accurately on these intervals, we can either refine the partition or increase the degree. The aim of this section is to determine which action is more efficient in a given situation.

\subsection{Theoretical considerations}

First, let us develop a theoretical cost model for solving RDEs using the Log-ODE method.

Recall that the vector field of the degree $N$ Log-ODE method is given by $$\sum_{k=1}^N f^{\circ k} \pi_k(\log_N(\bg))(y).$$ Here, $f^{\circ k}$ is given as a tensor-valued function $$f^{\circ k}\colon \R^e \to \R^e\otimes (\R^d)^{\otimes k}.$$ This function is evaluated, and the result is then contracted with the $k$-tensor $\pi_k(\log_N(\bg))\in (\R^d)^{\otimes k}$, and finally summed over $k$.

The function $f^{\circ k}$ has $ed^k$ components, while $\pi_k(\log N(\bg))$ has $d^k$ components. Assume that evaluating one component of the function $f^{\circ k}$ or the log-signature $\pi_k(\log_N (\bg))$ carries a cost of $c_N$. The total cost of evaluating the Log-ODE vector field is hence roughly of order $$\sum_{k=1}^N e c_k d^k \approx e c_N d^N,$$ assuming that $(c_k)$ is non-decreasing.

Next, assume that we are using a partition of length $n$. On each interval we have to solve an ODE, and the cost of solving that ODE is assumed to be approximately proportional to the number of calls of the vector field times the cost of evaluating the vector field. Since we are working with an adaptive algorithm, it seems reasonable to assume that the ODEs of the various intervals of the partition are roughly equally difficult to solve, i.e. the number of calls of the vector field is roughly constant. Therefore, the cost of solving the RDE using $n$ intervals of level $N$ is roughly of order $nec_Nd^N$. If we use $(n_i)_{i=1}^\infty$ intervals of level $i$, we roughly have a cost of order $$e\sum_{i=1}^\infty n_i c_i d^i.$$

Next, we ask the question how the error changes if we increase $N$ or subdivide the interval.

Standard error bounds for the Log-ODE method on a single interval $[s,t]$ are of the form $$a_N \omega(s, t)^{\frac{N+1}{p}},$$ where $a_N$ is some constant that depends, among other things, on $N$, and $\omega$ is a control associated to the problem (essentially the $p$-variation of $\bx$ multiplied with the local Lipschitz norm of $f$). This error is then of course propagated to the end. Since the error representation formula already takes care of the error representation formula, we may assume that there is a control function $\omega$, such that we roughly have $$\abs{g(y_T) - g(\overline{y}_T)} \le \sum_{k=1}^n \abs{\left(\int_0^1 \Psi(t_k, \overline{y}_{t_k} + s e(t_k)) \sdd s\right) e(t_k)} \approx \sum_{k=1}^n a_{N_k} \omega(t_k,t_{k+1})^{\frac{N_k + 1}{p}},$$ assuming that we use the Log-ODE method of level $N_k$ on the interval $[t_k,t_{k+1}]$. 

\subsection{Suggestions for the implementation}

Given these models for the cost and the error, we now try to answer the question on when to refine the partition, and when to increase the degree. Assume for a moment that we know the constants $c_i$ and $a_i$, and the control $\omega$. Then, refining an an interval $[t_k,t_{k+1}]$ into $m$ subintervals of equal length (w.r.t. the control $\omega$) leads to an increase in the computational cost for the interval $[t_k,t_{k+1}]$ by a factor of $m$, while it decreases the error on this interval by a factor of $m \cdot m^{-\frac{N_k+1}{p}} = m^{-\left(\frac{N_k + 1}{p} - 1\right)}.$ On the other hand, increasing the degree from $N_k$ to $N_k + 1$ increases the computational cost by a factor of $\frac{c_{N_k+1}}{c_{N_k}}d$, while decreasing the error by a factor of $\frac{a_{N_k+1}}{a_{N_k}}\omega(t_k,t_{k+1})^{1/p}.$

To achieve the same decrease in the error by refining the interval, as by increasing the degree, we need to choose $$m = \left(\frac{a_{N_k}}{a_{N_k+1}}\right)^{\frac{p}{N_k + 1 - p}}\omega(t_k,t_{k+1})^{-\frac{1}{N_k + 1 - p}}.$$ Comparing the costs of these two operations, we see that it is more efficient to refine the interval if
\begin{align}
\left(\frac{a_{N_k}}{a_{N_k+1}}\right)^{\frac{p}{N_k + 1 - p}}\omega(t_k,t_{k+1})^{-\frac{1}{N_k + 1 - p}} \le \frac{c_{N_k + 1}}{c_{N_k}}d,\label{eqn:RefinementRule}
\end{align}
while it is more efficient to increase the degree otherwise.

This gives us a well-defined method for deciding when to increase the degree and when to refine the interval, assuming that we know $a_i$, $c_i$ and $\omega$. Hence, it merely remains to estimate these quantities. First, note that given $a_i$ and $c_i$, it is clear what $\omega$ is on the intervals $[t_k, t_{k+1}]$, since $$\omega(t_k, t_{k+1})^{\frac{N_k + 1}{p}} \approx \frac{1}{a_{N_k}} \abs{\left(\int_0^1 \Psi(t_k, \overline{y}_{t_k} + s e(t_k)) \sdd s\right) e(t_k)}$$ by definition. 

Moreover, the rule \eqref{eqn:RefinementRule} makes it clear that we only care about the fractions of successive $a_i$ and $c_i$, and not about their absolute values. Setting, say, $a_1 = c_1 = 1$, we can inductively determine $a_{i+1}$ (resp. $c_{i+1}$) from $a_i$ (resp. $c_i$) by picking an interval $[t_k, t_{k+1}]$, and using the Log-ODE method of level $i$ and $i+1$. By comparing the computational times $T_i$ and $T_{i+1}$, we get an estimate for $c_{i+1}$ since $\frac{c_{i+1}}{c_i} d = \frac{T_{i+1}}{T_i}$ by assumption. 

Determining $a_{i+1}$ is only slightly more involved. Denote the estimates of the propagated local errors, which we obtain by the error representation formula, by $e_i$ and $e_{i+1}$. By assumption, $$e_j = a_j \omega(t_k, t_{k+1})^{\frac{j+1}{p}},\qquad j=i, i+1,$$ indicating $$\omega(t_k, t_{k+1})^{1/p} = \left(\frac{a_i}{e_i}\right)^{\frac{1}{i+1}},\qquad a_{i+1} = \frac{e_{i+1}}{\omega(t_k,t_{k+1})^{\frac{i+2}{p}}} = e_{i+1}\left(\frac{e_i}{a_i}\right)^{\frac{i+2}{i+1}}.$$

In practice, if we have never increased the degree from $i$ to $i+1$ before, we can do this once to get a rough estimate of $c_{i+1}$ and $a_{i+1}$. Afterwards, we can refine these estimates every time we increase the degree from $i$ to $i+1$ on other intervals (e.g. by taking the medians of the estimators).

\section{The Euler approximation}\label{sec:Euler}

Let $\bx\in C^{p\var}([0,T],\R^d)$, let $f\in \mathcal{V}^{N,p}(\R^d,\R^e)$ with $N>p-1$, and let $y_0\in \R^e$. The Euler approximation of level $N$ of the RDE $$\dd y_t = f(y_t) \sdd\bx_t,\qquad y_0 = y_0$$ is given by $$y_T = y_0 + \sum_{k=1}^N f^{\circ k}(y_0) \pi_k(S_N(\bx)_{0,T}).$$ Here, $$f^{\circ 1} = f,\qquad\text{and}\qquad f^{\circ (k+1)} = D(f^{\circ k}) f.$$ This defines a 1-step scheme $A^{\textup{Euler}}$ given by $A^{\textup{Euler}}_N(f,y_0,S_N(\bx)_{0,T}) \coloneqq y_T.$

\begin{lemma}
The 1-step scheme $A^{\textup{Euler}}$ is an admissible 1-step scheme.
\end{lemma}

\begin{proof}
This follows immediately from \cite[Theorem 10.30]{friz2010multidimensional}.
\end{proof}

We remark that $A^{\textup{Euler}}$ is neither local nor group-like.

Let $\log$ be the tensor algebra logarithm given by $\log_N:T_1^N(\R^d)\to T^N(\R^d),$ $$\log_N(\bm{g}) = \sum_{j=1}^N \frac{(-1)^{j+1}}{j}(\bm{g}-\bm{1})^{\otimes j}.$$

\begin{lemma}\label{lem:LogarithmIsLocallyLipschitz}
Let $\bm{g},\widetilde{\bm{g}}\in T_1^N(\R^d)$. Then, there exists a constant $C = C(N)$ such that $$\|\log_N \bm{g} - \log_N \widetilde{\bm{g}}\| \le C\left(1\lor\|\bg\|\lor\|\widetilde{\bg}\|\right)^{N-1}\|\bg-\widetilde{\bm{g}}\|.$$
\end{lemma}

\begin{proof}
Let $k\in \{1,\dots,N\}.$ Then, 
\begin{align*}
\|\pi_k(\log_N &\bm{g} - \log_N \widetilde{\bm{g}})\|\\
&\le \sum_{j=1}^N \frac{1}{j} \left\|\pi_k\left((\bg - \bm{1})^{\otimes j}\right) - \pi_k\left((\widetilde{\bg} - \bm{1})^{\otimes j}\right)\right\|\\
&\le \sum_{j=1}^N \frac{1}{j}\sum_{\substack{i_1,\dots,i_j = 1,\\ i_1 + \dots + i_j = k}}^k \|g_{i_1}\otimes\dots\otimes g_{i_j} - \widetilde{g}_{i_1}\otimes\dots\otimes \widetilde{g}_{i_j}\|\\
&\le \sum_{j=1}^N \frac{1}{j}\sum_{\substack{i_1,\dots,i_j = 1,\\ i_1 + \dots + i_j = k}}^k\sum_{\ell=1}^j \|g_{i_1}\otimes\dots\otimes g_{i_{\ell-1}}\otimes (g_{i_\ell} - \widetilde{g}_{i_\ell}) \otimes \widetilde{g}_{i_{\ell+1}}\otimes\dots \otimes \widetilde{g}_{i_j}\|\\
&\le \sum_{j=1}^N \frac{1}{j}\sum_{\substack{i_1,\dots,i_j = 1,\\ i_1 + \dots + i_j = k}}^k\sum_{\ell=1}^j \|\bg\|^{\ell-1} \|\bg - \widetilde{\bg}\| \|\widetilde{\bg}\|^{j-\ell}\\
&\le C(N,k) (1\lor\|\bg\|\lor\|\widetilde{\bg}\|)^{N-1}\|\bg-\widetilde{\bg}\|.
\end{align*}
Taking the supremum over $k$ finishes the proof.
\end{proof}

\begin{lemma}\label{lem:EulerIsLipschitz}
Let $\bg^1,\bg^2\in G^N(\R^d)$, let $f\in\mathcal{V}^N(\R^d,\R^e)$, and let $y\in \R^e$. Then, $$\|A^{\textup{Euler}}_N(f,y,\bg^1) - A^{\textup{Euler}}_N(f,y,\bg^2)\| \le C\|\bg^1-\bg^2\|,$$ where $C = C(N,\|f\|_{\lip^N})$ is increasing in both parameters.
\end{lemma}

\begin{proof}
This is immediate from the fact that $f^{\circ k}$ is a linear map on $(\R^d)^{\otimes k}$ with
\begin{align*}
\|A^{\textup{Euler}}_N(f,y,\bg^1) - A^{\textup{Euler}}_N(f,y,\bg^2)\| &= \left\|\sum_{k=1}^N f^{\circ k}(y_0) \pi_k(\bg^1-\bg^2)\right\|\\
&\le C(\|f\|_{\lip^N}\lor\|f\|_{\lip^N}^N)\|\bg^1-\bg^2\|.\qedhere
\end{align*}
\end{proof}

\printbibliography


\end{document}